\documentclass[11pt]{amsproc}

\usepackage{multicol}
\usepackage{tikz}
\usepackage{amsmath,amscd}
\usepackage{enumerate}
\usepackage{amsmath}
\usepackage{amssymb}

\usepackage[colorlinks, bookmarks=true]{hyperref}
\usepackage{color,graphicx,shortvrb}
\usepackage[latin 1]{inputenc}

\newtheorem{theorem}{Theorem}[section]
\newtheorem{lemma}[theorem]{Lemma}

\theoremstyle{definition}

\newtheorem{example}[theorem]{Example}

\newtheorem{openproblem}[theorem]{Open problem}
\newtheorem{problem}[theorem]{Problem}
\newtheorem{corollary}[theorem]{Corollary}
\newtheorem{proposition}[theorem]{Proposition}

\theoremstyle{remark}
\newtheorem{remark}[theorem]{Remark}

\numberwithin{equation}{section}

\def\J#1#2#3{ \left\{ #1,#2,#3 \right\} }
\def\RR{{\mathbb{R}}}

\def\11{\textbf{$1$}}
\def\CC{{\mathbb{C}}}
\def\CC{{\mathbb{C}}}
\def\HH{{\mathbb{H}}}

\newcommand{\ran}{\mathrm{ran} \,}

\begin{document}

% \title[short text for running head]{full title}
\title[Local and 2-local derivations]{A survey on local and 2-local derivations on C$^*$- and von Neuman algebras}

%    Only \author and \address are required; other information is
%    optional.  Remove any unused author tags.

%    author one information
% \author[short version for running head]{name for top of paper}
\author[Ayupov]{Shavkat Ayupov}
\address{Institute of Mathematics, National University of Uzbekistan, 100125 Tashkent, Uzbekistan}
\email{sh\_ayupov@mail.ru}

\author[Kudaybergenov]{Karimbergen Kudaybergenov}
\address{Ch. Abdirov 1, Department of Mathematics, Karakalpak State University, Nukus 230113, Uzbekistan}
\email{karim2006@mail.ru}

%    author two information

\author[Peralta]{Antonio M. Peralta}
\address{Departamento de An{\'a}lisis Matem{\'a}tico, Facultad de
Ciencias, Universidad de Granada, 18071 Granada, Spain.}
\curraddr{Visiting Professor at Department of Mathematics, College of Science, King Saud University, P.O.Box 2455-5, Riyadh-11451, Kingdom of Saudi Arabia.}
\email{aperalta@ugr.es}
\thanks{Third author partially supported by the Spanish Ministry of Science and Innovation,
D.G.I. project no. MTM2011-23843, Junta de Andaluc\'{\i}a grant FQM375, and the Deanship of Scientific Research at King Saud University (Saudi Arabia) research group no. RG-1435-020.}

%    The 2010 edition of the Mathematics Subject Classification is
%    now available.  If you are citing a classification from the
%    new scheme, use the following input coding instead.
%\subjclass[2010]{Primary }

\date{November, 3rd, 2014}

\begin{abstract}We survey the results on local and 2-local derivations on C$^*$-algebras, von Neumann algebras and JB$^*$-triples.
\end{abstract}

\maketitle

\section{Introduction}\label{sec:intro}

The last twenty years witnessed a fruitful growth of the theory of local and 2-local derivations on von Neumann algebras, C$^*$-algebras, and JB$^*$-triples. The studies on local derivations were formally started by Kadison \cite{Kad90} and Larson and Sourour \cite{LarSou} in 1990. The fast growing of the theory during the last 25 years invites us to write a first survey on local and 2-local derivations and homomorphisms to revisit the main contributions and the main open problems. We combine new proofs of previously established results with new results and connections between the different techniques developed in recent years.\smallskip

The studies about derivations on Banach algebras go back to the origins of the theory of Banach algebras and C$^*$-algebras. In 1958, Kaplansky conjectured that that any derivation of a C$^*$-algebra would be automatically continuous (see \cite{Kap}). This conjecture challenged to the international community, and in particular to Sakai, who gave a proof of this fact in 1960 (cf. \cite{Sak60}).\smallskip

We recall that a {\it derivation} of a Banach algebra $A$ into a Banach $A$-module $X$ is a linear map $D: A\to X$ satisfying $D(a b) = D(a) b +a  D(b),$ ($a\in A$). Ringrose extended Sakai's theorem on automatic continuity of derivations on C$^*$-algebras by proving that every derivation from a C$^*$-algebra $A$ to a Banach $A$-bimodule is continuous (cf. \cite{Ringrose72}).\smallskip

A typical example of a derivation from a Banach algebra $A$ into a Banach $A$-bimodule would be: let $x$ be an element in $X$, the mapping $\hbox{adj}_{x} :A \to X$, $a\mapsto \hbox{adj}_{x} (a) := x a -  a x$, is a derivation. A derivation $D: A\to X$ is said to be an \emph{inner derivation} when it can be written in the form $D = \hbox{adj}_{x}$ for some $x\in X$. Another fundamental contribution, due to Sakai, states that every derivation on a von Neumann algebra is inner (cf. \cite[Theorem 4.1.6]{Sak} or \cite{Sak66,Kad66}).\smallskip

A {\it Jordan derivation} from $A$ into $X$
is a linear map $D$ satisfying $D(a^2) = a D(a) + D(a)
a,$ ($a\in A$), or equivalently, $D(a\circ b)=a\circ D(b)+ D(a)\circ b$ ($a,b\in A$), where $a\circ b = \frac{a b+b a}{2},$ whenever $a,b\in A$, or one from $\{a,b\}$ is in $A$ and the other is in $X$. Clearly, every derivation from $A$ into $X$ is a Jordan derivation. The other implication need not be true for general Banach algebras, however Johnson establishes in 1996 that every bounded Jordan derivation from a
C$^*$-algebra $A$ to a Banach $A$-bimodule is a derivation (cf. \cite{John96}). In a recent paper, due to Russo and the third author of this note, it is proved that every Jordan derivation
from a C$^*$-algebra $A$ to a Banach $A$-bimodule (or to a Jordan Banach $A$-module)
is continuous (cf. \cite[Corollary 17]{PeRu}).\smallskip

In 1990, Kadison \cite{Kad90} and Larson and Sourour \cite{LarSou} introduce the concept of local derivation in the following sense: let $X$ be a Banach $A$-bimodule over a Banach algebra $A$, a linear mapping $T:A\to X$ is said to be a \emph{local derivation} if for every $a$ in $A$ there exists a derivation $D_a :A\to X$, depending on $a$, satisfying $T(a) = D_a (a)$. Kadison proves in \cite[Theorem A]{Kad90} that each continuous local derivation of a von Neumann algebra $M$ into a dual Banach $M$-bimodule is a derivation. This theorem gave way to a cascade of results and studies on derivations on C$^*$-algebras, culminating with a definitive contribution due to Johnson, which asserts that every continuous local derivation of a C$^*$-algebra $A$ into a Banach $A$-bimodule is a derivation \cite[Theorem 5.3]{John01}. In the just quoted paper, Johnson also gives an automatic continuity result by proving that local derivations of a C$^*$-algebra $A$ into a Banach $A$-bimodule $X$ are continuous even if not assumed a priori to be so (cf. \cite[Theorem 7.5]{John01}).\smallskip

Section \ref{sec:local der} is devoted to survey the main results on local triple derivations on von Neumann algebras and on C$^*$-algebras. We revisit the main contributions due to Kadison, Bre\v{s}ar, Shul'man and Johnson. We present here new and simplified proofs obtained with new and recent methods. We also revisit the connections appearing between local derivations and generalized derivations, in the sense introduced and studied by Alaminos, Bresar, Extremera, and Villena \cite{AlBreExVill09} and Li and Pan \cite{LiPan}.\smallskip

In 1990, Kadison also affirmed that the study on local derivations should be also extended to ``local higher cohomology (for example, local 2-cocycles)'' (cf. \cite{Kad90}). Let $X$ be a Banach $A$-bimodule over a Banach algebra $A$. An $n$-multilinear operator $T: A\times\ldots\times A\to X$ is said to be an $n$-cocycle when the identity $$ a_1 T (a_2, \ldots, a_n,a_{n+1})+ \sum_{j=1}^{n} (-1)^{j} T (a_1, \ldots, a_{j-1},a_j a_{j+1}, \ldots, a_{n+1} )$$ $$ + (-1)^{n+1} T (a_1, \ldots, a_n) a_{n+1}=0, $$ holds for every $a_1,\ldots, a_n,a_{n+1}$ in $A$. The reader should recall that 1-cocycles from $A$ into $X$ are precisely the derivations of $A$ into $X$. The corresponding definition of local $n$-cocycles reads as follows: a multilinear mapping $T: A\times \ldots \times A \to  X$ is said to be a \emph{local $n$-cocycle} if for every $a_1,\ldots, a_n$ in $A$, there exists an $n$-cocycle $\Phi_{_{a_1,\ldots, a_n}}$ (depending on $a_1,\ldots, a_n$) such that $T(a_1,\ldots, a_n) = \Phi_{_{a_1,\ldots, a_n}} (a_1,\ldots, a_n)$. The question posed by Kadison in his comments can be materialized as follows:

\begin{problem}\label{problem local n-cocycles}
Is every continuous local $n$-cocycle of a C$^*$-algebra $A$ into a Banach $A$-bimodule an $n$-cocycle?
\end{problem}

In section \ref{ss: cohomology}, we survey the positive answer to Problem \ref{problem local n-cocycles} provided by Samei in \cite[Theorem 5.4]{Samei08}. We shall see that part of the technical results established to study continuous local derivations from a C$^*$-algebra $A$ into a Banach $A$-bimodule can be applied to show that  every continuous local $n$-cocycle of $A$ into $X$ is an $n$-cocycle (Theorem \ref{t local n-cocycles}).\smallskip

Every C$^*$-algebra belongs to a strictly wider class of complex Banach spaces equipped with a triple product satisfying certain algebraic and analytic hypothesis (see section \ref{ss: local triple der} for a detailed definition). A \emph{triple derivation} on a JB$^*$-triple $E$ is a linear mapping $\delta: E\to E$ satisfying that $$\delta \J abc = \J {\delta(a)}bc + \J a{\delta(b)}c + \J ab{\delta(c)},$$ for every $a,b,c\in E$. Barton and Friedman prove in \cite{BarFri} that every triple derivation on a JB$^*$-triple is automatically continuous. In \cite{Mack}, Mackey introduces local derivations on JB$^*$-triples.  A \emph{local triple derivation} on $E$ is a linear map $T : E\to E$ such that for each $a$ in $E$ there exists a triple derivation $\delta_{a}$ on $E$ satisfying $T(a) = \delta_a (a).$ Mackey also establishes a triple version of previously mentioned Kadison's theorem, proving that every local triple derivation on a JBW$^*$-triple (i.e., a JB$^*$-triple which is a dual Banach space) is a triple derivation. The problem whether every local triple derivation on a JB$^*$-triple is a triple derivation was left as an open problem by Mackey in the same paper. Very recently, Burgos, Fernández-Polo and the third author of this survey gave a complete positive answer to this problem in \cite{BurFerPe2013}. Section \ref{ss: local triple der} contains a detailed presentation of the major achievements on local triple derivations on real and complex JB$^*$-triples and C$^*$-algebras.\smallskip

Section \ref{sec:automatic cont} gathers a complete collection of results on automatic continuity of derivations, local derivations, triple derivations, and local triple derivations. \smallskip

Inspired by the Gleason-Kahane-\.{Z}elazko and the Kowalski-S{\l}odkowski theorems, \v{S}emrl introduced in \cite{Semrl97} the notions of 2-local homomorphisms and derivations. For our purposes, we recall that a (not necessarily linear nor continuous) mapping $T$ from a Banach algebra $A$ into a Banach $A$-bimodule $X$ is said to be a \emph{2-local derivation} if for every $a,b\in A$ there exists a (bounded linear) derivation $D_{a,b}: A\to X$, depending on $a$ and $b$, such that $D_{a,b} (a) = T(a)$ and $D_{a,b}(b) = T(b)$. It is established by \v{S}emrl that for every infinite-dimensional separable Hilbert space $H$, every 2-local
derivation $T : B(H) \to B(H)$ (no linearity or continuity of $T$ is assumed) is a derivation \cite[Theorem 2]{Semrl97}.\smallskip

In \cite[Remark]{Semrl97},  \v{S}emrl states that the conclusion of the above theorem also holds when $H$ is finite-dimensional. In such a case, however, he was only able to get a long proof involving tedious computations, and so, he decided not include these results. In \cite{KimKim04} Kim and Kim gave a short proof of the fact that every 2-local derivation on a finite-dimensional complex matrix algebra is a derivation. New techniques introduced by the first and second authors of this note in \cite{AyuKuday2012} have been applied to prove that every $2$-local derivation on $B(H)$, with $H$ an arbitrary Hilbert space (no separability is assumed), is a derivation. A similar result for $2$-local derivations on finite von Neumann algebras was obtained by Nurjanov, Alauatdinov and the first two authors of this survey in \cite{AKNA}. In \cite{AA2} the authors extend all the above results and give a short proof of this result for arbitrary semi-finite von Neumann algebras. When considering 2-local derivations on general von Neumann algebras, the most conclusive result asserts that every 2-local derivation on an arbitrary von Neumann algebra is a derivation (see \cite{AyuKuday2014}). In the first part of section \ref{s:2-local der} we survey the main results about 2-local derivations on von Neumann algebras. In subsection \ref{ss: local triple der} we also present another recent result, due to Oikhberg, Russo and the second and third authors of this note, which estates that every 2-local triple derivation on a von Neumann algebra is a triple derivation (cf. \cite{KOPR2014}). In the final subsection we survey some results on 2-local derivations on Arens algebras established by Nurjanov, Alauatdinov, and the frist two authors of this note in \cite{AKNA}.

\section{Local derivations on C$^*$-algebras}\label{sec:local der}

Let $X$ be a Banach $A$-bimodule over a Banach algebra $A$. We recall that a linear mapping $T:A\to X$ is said to be a {local derivation} if for every $a$ in $A$ there exists a derivation $D_a :A\to X$, depending on $a$, satisfying $T(a) = D_a (a)$. The main result established by Kadison in \cite[Theorem A]{Kad90} states that each continuous local derivation of a von Neumann algebra $M$ into a dual Banach $M$-bimodule is a derivation. Actually, Bre\v{s}ar shows in \cite{Bre92} that similar techniques and ideas can be applied to prove a slightly stronger statement, which is included here with a simplified proof.

\begin{theorem}\label{thm Kadison real rank zero}{\rm[Kadison \cite{Kad90}], [Bre\v{s}ar \cite[Theorem 2]{Bre92}]} Let $A$ be a unital C$^*$-algebra satisfying that every self-adjoint element in $A$ can be approximated in norm by a finite linear combinations of projections in $A$. Then every  continuous local derivation from $A$ into a Banach $A$-bimodule is a derivation.
\end{theorem}

\begin{proof} Let $T:A \to X$ be a local derivation from $A$ into a Banach $A$-bimodule. Let us consider a projection $p\in A$. In this case, \begin{equation}\label{eq local der on a proj} T(p)= D_p (p p ) = p D_p (p) + D_p (p) p = p T(p) + T(p) p.
\end{equation}
We consider now two orthogonal projections $p,q$ in $A$. By \eqref{eq local der on a proj}, $$T(p+q) = (p+q) T(p+q) + T(p+q) (p+q),$$ which combined with $T(p) = p T(p) + T(p) p$ and $T(q) = q T(q) + T(q)q,$ assures that \begin{equation}\label{eq local der on a proj 2} p T(q) + q T(p) + T(p) q + T(q) p =0.
\end{equation} Now, having in mind that $p\perp q$, we deduce that $$q T(p) = q (p T(p) + T(p) p) = q T(p) p,$$ and similarly $T(p) q = p T(p) q$, $T(q) p = q T(q) p,$ $p T(q) = p T(q) q$, $p T(q) p =0$, and $q T(p) q =0$. Combining these identities with \eqref{eq local der on a proj 2} we get $$0 = q (p T(q) + q T(p) + T(p) q + T(q) p) = q T(p) + q T(q) p = q T(p) +  T(q) p ,$$ which proves that \begin{equation}\label{eq local der on orthog proj}  q T(p) +  T(q) p =0,
\end{equation} for every couple of projections $p,q\in A$ with $p\perp q.$\smallskip

Suppose now that $p_1,\ldots, p_n$ are mutually orthogonal projections in $A$, $\lambda_1,\ldots,\lambda_n$ are real numbers and $\displaystyle a = \sum_{j=1}^{n} \lambda_j p_j$ is an algebraic self-adjoint element in $A$. By linearity $\displaystyle T(a^2 ) = \sum_{j=1}^{n} \lambda_j^2 T(p_j).$ On the other hand, $$T(a) a + a T(a) = T\left( \sum_{j=1}^{n} \lambda_j p_j \right) \left(\sum_{k=1}^{n} \lambda_k p_k \right) +  \left(\sum_{k=1}^{n} \lambda_k p_k \right) T\left( \sum_{j=1}^{n} \lambda_j p_j \right)$$
$$= \sum_{j,k=1}^{n} \lambda_j \lambda_k \left(T(p_j) p_k + p_k T(p_j)\right) = \hbox{(by \eqref{eq local der on orthog proj})} = \sum_{j=1}^{n} \lambda_j^2 T(p_j) = T(a^2).$$

By hypothesis, every self-adjoint element in $A$ can be approximated in norm by a finite linear combinations of projections in $A$, thus, the continuity of $T$ implies that \begin{equation}\label{eq on a single hermitian} T(a^2 ) = T(a) a + a T(a),
\end{equation} for every $a=a^*$ in $A.$ A simple polarization argument shows that, for each $c= a+ i b\in A$, with $a=a^*$ and $b=b^*$, we have $$T(c^2 ) = T(a^2 - b^2 + 2i (a\circ b)) = a T(a) + T(a) a - b T(b) - T(b) b $$ $$+ 2 i ( T(a) \circ b + a \circ T(b)) =2 T(a+i b) \circ (a+ ib) = 2 T(c) \circ c,$$
where $\circ$ denotes the canonical Jordan product $x\circ c = \frac12 ( x c + cx).$ We have therefore shown that $T$ is a Jordan derivation. We deduce from Johnson's theorem \cite{John96} that $T$ is a derivation.
\end{proof}

The following corollaries are direct consequences of the above theorem.

\begin{corollary}\label{c Kadison 90}{\rm[Kadison \cite[Theorem A]{Kad90}]} Every continuous local derivation of a von Neumann algebra $M$ into a dual Banach $M$-bimodule is a derivation.$\hfill\Box$
\end{corollary}

\begin{corollary}\label{c Kadison compact} Let $A=K(H)$ or $A=B(H),$ where $H$ is a complex Hilbert space.  Then every continuous local derivation of $A$ into a Banach $A$-bimodule is a derivation.$\hfill\Box$
\end{corollary}

For information only, we mention that, accordingly to the data basis of MathSciNet, Kadison's theorem \cite[Theorem A]{Kad90} has received over 95 citations. It would be completely impossible to mention all of them in this paper. We shall highlight some of the results which can help the reader to understand the development of the theory. We have already mentioned that every derivation on a von Neumann algebra is inner \cite[Theorem 4.1.6]{Sak}, consequently, if $T$ is a bounded linear operator of a von Neumann algebra $M$ into itself, and if for each $a$ in $M$ there is a $z_a$ in $M$ satisfying $T(a)=[a,z_a]$, then there exists $z$ in $M$ such that $T(a)=[a,z]$, for all $a$ in $M$.

The question whether every local derivation from a C$^*$-algebra $A$ into a Banach $A$-bimodule is continuous and a derivation remained open for eleven years. A partial answer was provided by Shul'man in 1994 (cf. \cite{Shu}), by showing that every continuous local derivation from a C$^*$-algebra into itself is a derivation. The complete solution to the questions originated by Kadison's paper appear in 2001. In a celebrated result Johnson shows that every continuous local derivation of a C$^*$-algebra $A$ into a Banach $A$-bimodule is a derivation \cite[Theorem 5.3]{John01}. In the same paper, Johnson also gives an automatic continuity result by proving that local derivations of a C$^*$-algebra $A$ into a Banach $A$-bimodule $X$ are continuous even if not assumed a priori to be so (cf. \cite[Theorem 7.5]{John01}). The arguments in Johnson's paper rely on the fact, proved by the same author, that every Jordan derivation from $A$ into $X$ is a derivation, and thus, it is enough to show the result when $A$ is a C$^*$-algebra generated by a single self-adjoint element (a fact we have explicitly applied in the proof of Theorem \ref{thm Kadison real rank zero} above). Johnson considers and studies ``\emph{local multipliers}'' (defined by replacing ``derivation'' by ``multiplier'' in Kadison's definition). The main tool in his technical proof is the fact that the diagonal in $\mathbb{R}^2$ is a set of synthesis for $C_0(\mathbb{R}) \widehat{\otimes}
C_0(\mathbb{R})$.

\begin{theorem}\label{thm Johnson local 2001}{\rm[Johnson, \cite[Theorem 5.3]{John01}]} Every continuous local derivation of a C$^*$-algebra $A$ into a Banach $A$-bimodule is a derivation.
\end{theorem}

Instead of presenting here the technical proof given by Johnson in \cite{John01}, we shall revisit and adapt more renewed techniques developed in recent years in the setting of JB$^*$-triples (cf. \cite[\S 3]{BurFerPe2013}, \cite{BurFerGarPe2014}, \cite{Pe2014} and section \ref{ss: local triple der}) to provide to the reader a simpler proof. In \cite[Sublemma 5]{Kad90}, Kadison isolated an interesting property satisfied by every local derivations of a von Neumann algebra into a Banach bimodule, which remains valid for local derivations of a C$^*$-algebra into a Banach bimodule.

\begin{lemma}\label{l Kad sublemma 5} Let $T: A\to X$ be a local derivation of a C$^*$-algebra into a Banach $A$-bimodule. Then $a T(b) c=0$, whenever $ab=bc=0$ in $A.$
\end{lemma}

\begin{proof}Let us consider $a,b,c\in A$ with  $ab=bc=0$. The identity $$a T(b) c= a D_b(b) c = \left(D_b (a b) - D_b (a) b\right) c = - D_b (a) bc = 0,$$ proves the desired statement.
\end{proof}

We recollect at this point some well known facts on the Arens bitransposes of a bounded bilinear map (cf. \cite{Arens51}). Let $m : X\times Y \to Z$ be a bounded bilinear mapping between Banach spaces. We can define a bounded bilinear mapping $m^*: Z^*\times X \to Y^*$ by setting $m^* (z^\prime,x) (y) := z^\prime (m(x,y))$ $(x\in X, y\in Y, z^\prime\in Z^*)$. Iterating the process, we define a mapping $m^{***} = [(m^*)^*]^*: X^{**}\times Y^{**} \to Z^{**}.$ The definition given above satisfies that the mapping $x^{\prime\prime}\mapsto  m^{***}(x^{\prime\prime} , y^{\prime\prime})$ is weak$^*$ to weak$^*$ continuous whenever we fix  $y^{\prime\prime} \in  Y^{**}$, and similarly, the mapping $y^{\prime\prime}\mapsto  m^{***}(x, y^{\prime\prime})$ is weak$^*$ to weak$^*$ continuous for every $x\in  X$. If we consider the transposed mapping $m^{t} : Y\times X\to Z,$ $m^{t} (y,x) = m(x,y)$ and the extended mapping $m^{t***t}: X^{**}\times Y^{**} \to Z^{**},$ we get a new bounded bilinear extension satisfying that the mapping $x^{\prime\prime}\mapsto  m^{t***t}(x^{\prime\prime} , y)$ is weak$^*$ to weak$^*$ continuous whenever we fix  $y \in  Y$, and the mapping $y^{\prime\prime}\mapsto  m^{t***t}(x^{\prime\prime}, y^{\prime\prime})$ is weak$^*$ to weak$^*$ continuous for every $x^{\prime\prime}\in  X^{**}$.\smallskip

In general, the mappings $m^{t***t}$ and $m^{***}$ do not coincide (cf. \cite{Arens51}). The mapping $m$ is said to be \emph{Arens regular} if $m^{t***t}=m^{***}.$ It is well known that the product of every C$^*$-algebra is Arens regular and the unique Arens extension of the product of $A$ to $A^{**}\times A^{**}$ coincides with the product of its enveloping von Neumann algebra (cf. \cite[Corollary 3.2.37]{Dales00}).\smallskip

Let $X$ be a Banach $A$-bimodule over a C$^*$-algebra $A.$ Let us denote by $\pi_1: A\times X \to X$ and $\pi_2: X\times A \to X$ the corresponding module operations given by $\pi_1(a,x) = a x$ and $\pi_2(x,a) = x a$, respectively. By an abuse of notation, given $a\in A^{**}$ and $z\in X^{**},$ we shall frequently  write $a z = \pi_1^{***} (a,z)$ and $z a = \pi_2^{***} (z,a)$. It is known that $X^{**}$ is a Banach $A^{**}$-bimodule for the just defined operations (\cite[Theorem 2.6.15$(iii)$]{Dales00}). By definition, for each $b\in A$, $b''\in A^{**}$, $z\in X$ and $z''\in X^{**}$ the mappings \begin{equation}\label{eq weakstar contin module operations}
 A^{**} \to X^{**}, a\mapsto a z'', \ a\mapsto z a,
 \end{equation}$$X^{**} \to X^{**}, x\mapsto b x, \ x\mapsto x b''$$ are weak$^*$-continuous.  It is also known that whenever $(a_\lambda)$ and $(x_\mu)$ are nets in $A$ and $X$, respectively, such that $a_\lambda \to a\in A^{**}$ in the weak$^*$ topology of $A^{**}$ and $x_\mu\to x\in X^{**}$ in the weak$^*$ topology of $X^{**}$, then \begin{equation}\label{eq product bidual module} a x = \pi_1^{***}(a,x) = \lim_{\lambda} \lim_{\mu} a_\lambda x_\mu \hbox{ and } x a= \pi_2^{***} (x,a) =  \lim_{\mu} \lim_{\lambda}  x_\mu a_\lambda
\end{equation} in the weak$^*$ topology of $X^{**}$ (cf. \cite[2.6.26]{Dales00}).\smallskip

Many different notions have been motivated by the property isolated in Lemma \ref{l Kad sublemma 5}. For example, Alaminos, Bresar, Extremera, and Villena \cite[\S 4]{AlBreExVill09} and Li and Pan \cite{LiPan} introduced the following definition: a linear operator $G$ from a Banach algebra $A$ into a Banach $A$-bimodule $X$ is said to be a \emph{generalized derivation} if there exists $\xi\in  X^{**}$ satisfying $$G(ab) = G(a)  b + a  G(b) - a \xi b \hbox{ ($a, b \in A$).}$$ A linear mapping $T : A \to X$ is called a \emph{local generalized derivation} if for each $a\in A$, there exists a generalized derivation $D_a : A \to X$ satisfying $T(a) = D_a (a)$.\smallskip

For each $x\in X$, the mapping $G_x: A\to X$, $a\mapsto a \circ x$ is a generalized derivation. Every (local) derivation of $A$ into $X$ is a (local) generalized derivation. When $A$ and $X$ are unital, every generalized derivation $G: A\to X$ is a derivation if and only if $G(1) = 0.$ It is also easy to see that when $A$ and $X$ are unital and $G_1,G_2 : A\to X$ are two generalized derivations with $G_1 (1) = G_2 (1)$, then the mapping $D= G_1 -G_2$ is a derivation.

\begin{remark}\label{remark generalized and derivations bitransposed} Let $D: A\to X$ be a continuous derivation of a C$^*$-algebra into a Banach $A$-bimodule. Given $c,d\in A^{**}$ we can find two (bounded) nets $(a_{\lambda})$ and $(b_\mu)$ in $A$ converging to $c$ and $d$ in the weak$^*$-topology of $A^{**}$, respectively. By the weak$^*$-continuity of $D^{**}$ and \eqref{eq weakstar contin module operations} and \eqref{eq product bidual module} we get $$D^{**}( cd) = D^{**} (c) d + c D(d),$$ which shows that $D^{**}$ is a derivation (this is a standard argument cf. \cite[Lemma 3]{Kad66}).\smallskip

Suppose $G: A\to X$ is a continuous generalized derivation. Let $\xi$ be an element in $X^{**}$ satisfying   $G(ab) = G(a)  b + a  G(b) - a \xi b \hbox{ ($a, b \in A$).}$ We observe that $a \xi b =  G(a)  b + a  G(b) - G(ab) \in X$, for every $a,b\in A.$ By the weak$^*$-density of $A$ in $A^{**}$, the weak$^*$-continuity of $G^{**}$ and the weak$^*$-continuity properties of the module operations of $X^{**}$ we have: \begin{equation}\label{eq remark semi-generalized der} G^{**}(ab) = G^{**}(a)  b + a  G(b) - a \xi b,
 \end{equation} for every $a\in A^{**}$ and $b\in A$ (just observe that the mapping $A^{**}\to X^{**},$ $a\mapsto a \xi$ is weak$^*$-continuous). However, the mapping $A^{**}\to X^{**},$ $a\mapsto \xi a$ need not continuous, so, it is not clear how we can assure that $G^{**}$ is a generalized derivation, that is, identity \eqref{eq remark semi-generalized der} holds for every $a,b\in A^{**}$. When $A$ and $X$ are unital,  $G(1) = \xi\in X$, thus, the weak$^*$-continuity of the mapping $A^{**}\to X^{**},$ $a\mapsto \xi a$ implies that $$G^{**}(ab) = G^{**}(a)  b + a  G^{**}(b) - a \xi b,$$ for every $a,b\in A^{**}$, and hence $G^{**}: A^{**}\to X^{**}$ is a generalized derivation. This discussion will be completed after Theorem \ref{t multipliers 1}.
\end{remark}

Surprisingly, the proof of the fact that every bounded local generalized derivation of a C$^*$-algebra $A$ into a Banach $A$-bimodule is a generalized derivation, requires a less technical argument than the one presented by Johnson \cite{John01} in the case of continuous local derivations. It is very easy to check that every local generalized derivation from a C$^*$-algebra $A$ into a Banach $A$-bimodule also satisfies the conclusion of Lemma \ref{l Kad sublemma 5}.

\begin{lemma}\label{l Kad sublemma 5 local generalized derivations} Let $T: A\to X$ be a local generalized derivation of a C$^*$-algebra into a Banach $A$-bimodule. Then $a T(b) c=0$, whenever $ab=bc=0$ in $A.$ $\hfill\Box$
\end{lemma}

In what follows, we denote by $A_{sa}$ the hermitian elements of a C$^*$-algebra $A.$\smallskip

The arguments we gave in the proof of Theorem \ref{thm Kadison real rank zero} rely on the hypothesis of norm density of algebraic elements (i.e. finite linear combinations of mutually orthogonal projections). If we try to apply a similar technique in the case of a general C$^*$-algebra $B$, we immediately find the obstruction arising from the scarcity of projections.

\begin{proposition}\label{prop zero products implies gener der} Let $X$ be a unital Banach $A$-bimodule over a unital C$^*$-algebra. Suppose $T: A \to X$ is a continuous linear operator. The following statements are equivalent: \begin{enumerate}[$(a)$]\item $aT(b)c = 0$, for every $a,b,c \in A$, with $ab = bc = 0$;
\item $aT(b)c = 0$, for every $a,b,c \in A_{sa}$, with $ab = bc = 0$;
\item $T$ is a generalized derivation.
\end{enumerate}
\end{proposition}

The proof of Proposition \ref{prop zero products implies gener der} will follow from a series of lemmas. We recall first some results and definitions. Let $A$ be a C$^*$-algebra. A continuous bilinear form $V : A\times A \to \mathbb{C}$ is said to be \emph{orthogonal} when $V (a,b) = 0$ for every $a,b\in A_{sa}$ with $a\perp b$ (see \cite[Definition 1.1]{Gold}).  Goldstein established in \cite{Gold} a beautiful result which determines the exact expression of every continuous bilinear orthogonal form on a C$^*$-algebra.

\begin{theorem}\label{thm Goldstein}\cite{Gold} Let $V : A\times A \to \CC$ be a continuous orthogonal form on a C$^*$-algebra.
Then there exist functionals $\phi,\psi\in A^*$ satisfying that
$$ V (a,b) = V_{\varphi,\psi} (a,b)= \varphi (a \circ b) + \psi ([a,b]),$$ for
all $a,b\in A$, where $a\circ b:= \frac12 (ab +ba)$, and $[a,b]:=
\frac12 (ab -ba)$. $\hfill\Box$
\end{theorem}

Let $X$ and $Y$ be Banach $A$-bimodules over a Banach algebra $A$. We recall that a mapping $f: X\to Y$ is said to be a \emph{left-annihilator-preserving} (respectively, \emph{right-annihilator-preserving}) if $f(x) a=0$, whenever $x a=0$ (respectively, $a f(x)=0$, whenever $ a x=0$) with $a\in A$, $x\in X$. A linear map $T: A\to X$ is called a \emph{left (respectively, right) multiplier} if $T (ab) = T (a)b$ (respectively, $T (ab) = a T (b)$), for every $a,b\in A$. Lin and Pan proved in \cite[Theorem 2.8]{LiPan} that every bounded and linear left-annihilator-preserving (respectively, every bounded and linear right-annihilator-preserving) map from a unital C$^*$-algebra $A$ into a unital Banach $A$-bimodule is a  left multiplier (respectively, a right multiplier). Our next lemma explores the case in which $X$ is not necessarily unital.

\begin{lemma}\label{l left mulipliers non unital} Let $T$ be a bounded linear operator from a unital C$^*$-algebra into a Banach space $X$. The following statements hold:\begin{enumerate}[$(a)$] \item Suppose $X$ is a Banach right $A$-module. Then the following statements are equivalent:\begin{enumerate}[$(1)$]
\item $T$ is left-annihilator-preserving;
\item $T(b) a=0$, whenever $b a=0$ with $a,b\in A_{sa}$;
\item $T(a) 1= T(1) a$, for every $a\in A$.
\end{enumerate}
In particular, when $X$ is a essential right $A$-module, every left-annihilator-preserving is a left multiplier.
\item Suppose $X$ is a Banach left $A$-module. Then the following statements are equivalent:\begin{enumerate}[$(1)$]
\item $T$ is right-annihilator-preserving;
\item $a T(b) =0$, whenever $a b=0$ with $a,b\in A_{sa}$;
\item $1T(a) = a T(1)$, for every $a\in A$.
\end{enumerate}
In particular, when $X$ is a essential left $A$-module, every right-annihilator-preserving is a right multiplier.
\end{enumerate}

\end{lemma}

\begin{proof} We shall only prove statement $(a)$. We shall assume the weaker hypothesis that $T(b) a=0$, whenever $b a=0$ with $a,b\in A_{sa}$. Fix an arbitrary $\phi\in X^{*}$, and define a bounded bilinear form $V_{\phi} : A\times A\to \mathbb{C}$, given by $V_{\phi} (a,b) = \phi (T(a) b)$. Given $a,b$ in $A_{sa}$ with $ab=0$ we have $T(a) b=0$, and hence $V_{\phi} (a,b)=0$. This means that $T$ is an orthogonal form. By Goldstein's theorem there exist functionals $\varphi,\psi\in A^*$ satisfying that
$$ V_{\phi} (a,b) = \varphi (a \circ b) + \psi ([a,b]),$$ for
all $a,b\in A$. In particular, $$\phi (T(a) 1) = V_{\phi} (a,1) = \varphi (a 1) = \varphi (1 a) =  V_{\phi} (1,a) = \phi (T(1) a),$$ for every $a\in A$. Since $\phi$ was arbitrarily chosen in $X^{*}$, we deduce, via Hahn-Banach theorem, that $T(a) 1 = T(1) a$, for every $a\in A$.
\end{proof}

\begin{proof}[Proof of Proposition \ref{prop zero products implies gener der}]  (compare with \cite[Proposition 1.1]{LiPan}) The implication $(a)\Rightarrow (b)$ is clear, and $(c)\Rightarrow (a)$ follows from Lemma \ref{l Kad sublemma 5 local generalized derivations}.\smallskip

We prove now $(b)\Rightarrow (c)$. Let us fix $a,b\in A_{sa}$ with $ab=0$. We define $L: A\to X$, by $L(x) = a T(bx)$. Given $c,d\in A_{sa}$ with $cd=0$, we have $L(c) d = a T(bc) d =0$, which implies that $L$ is left-annihilator-preserving. Lemma \ref{l left mulipliers non unital} assures that $a T(bx) = L(x) = L(1) x = aT(b) x$, for every $x\in A$. Fix an arbitrary $x\in A$. We have seen before that $$a (T(bx) - T(b) x)=0,$$ for every $ab=0$ in $A_{sa}$, thus, the operator $R: A \to X$, $R(z)= T(zx) - T(z) x$ is  right-annihilator-preserving on $A_{sa}$. By Lemma \ref{l left mulipliers non unital}, $T(zx) - T(z) x = R(z) = z R(1) = z T(x) - z T(1) x$, for every $z\in A$. This proves that   $T(zx) =  T(z) x +z T(x) - z T(1) x$, for every $x,z\in A$.
\end{proof}

We recall that for each C$^*$-algebra, $A$, the
\emph{multiplier algebra} of $A$, $M(A)$, is the set of all
elements $x\in A^{**}$ such that $x A , A x\subseteq A$. It is known that $M(A)$ is a C$^*$-algebra
and contains the unit element of $A^{**}$. Furthermore, $A= M(A)$ whenever $A$ is unital.\smallskip

The following theorem, which was originally obtained by gathering results in \cite[Theorem 4.5]{AlBreExVill09}, \cite[\S 3]{AlBreExVill}, \cite[Corollary 2.9]{LiPan}, \cite{BurFerGarPe2014} and \cite[\S 3]{BurFerPe2013}, states that the property isolated in Lemmas \ref{l Kad sublemma 5} and \ref{l Kad sublemma 5 local generalized derivations} actually characterizes bounded (local) generalized derivations. We present here a simplified and self-contained proof.\smallskip

\begin{theorem}\label{t multipliers 1}{\rm(}\cite[Theorem 4.5]{AlBreExVill09}, \cite[Proposition 4.3]{BurFerPe2013}{\rm)} Let $X$ be an essential Banach $A$-bimodule over a C$^*$-algebra $A$ and let $T: A \to X$ be a bounded linear operator. The following are equivalent:\begin{enumerate}[$(a)$]
%\item[$(a)'$] $T^{**}|_{M(A)} : M(A) \to X^{**}$ is a generalized Jordan derivation;
\item $T^{**} : A^{**} \to X^{**}$ is a generalized derivation;
\item $T^{**} = d + G_{T^{**} (1)}$, where the operators $d, G_{T^{**}(1)}: A^{**} \to X^{**}$ satisfy that $d$ is a derivation and $G_{T^{**}(1)}$ is a generalized derivation defined by  $G_{T^{**}(1)} (a) = T^{**} (1) \circ a = \frac12 ( a T^{**} (1) + T^{**} (1) a)$;
%\item[$(b)'$] $T^{**} : A^{**} \to X^{**}$ is a generalized Jordan derivation;
\item $T^{**} : A^{**} \to X^{**}$ is a local generalized derivation;
\item $T^{**}|_{M(A)} : M(A) \to X^{**}$ is a generalized derivation;
%\item[$(d)'$] $T^{**} : A^{**} \to X^{**}$ is a local generalized Jordan derivation;
\item $T^{**}|_{M(A)} : M(A) \to X^{**}$ is a local generalized derivation;
\item $a T^{**}(b) c = 0$, whenever $ab=bc=0$ in $M(A)$;
\item[$(f')$] $a T^{**}(b) c = 0$, whenever $ab=bc=0$ in $M(A)_{sa}$;
\item $T$ is a generalized derivation;
\item $T$ is a local generalized derivation;
\item $a T(b) c = 0$, whenever $ab=bc=0$ in $A$;
\item[$(i')$] $a T(b) c = 0$, whenever $ab=bc=0$ in $A_{sa}$.

\end{enumerate}
\end{theorem}

\begin{proof} Let $\xi = T^{**} (1)\in X^{**}$. Suppose that $T^{**}$ is a generalized derivation. Since $G_{\xi}: A^{**} \to X^{**}$, $a\mapsto G_\xi (a) = a\circ \xi$ is a generalized derivation and $T^{**}(1)  = G_{\xi} (1)$, the mapping $d= T^{**} - G_{\xi}$ is a derivation and $T^{**} = d + G_{\xi}$. This shows that statements $(a)$ and $(b)$ are equivalent. The implications  $(a)\Rightarrow (c)\Rightarrow (e)$, $(a)\Rightarrow (d)\Rightarrow (e)$, $(g)\Rightarrow (h)$, $(d)\Rightarrow (g)$ and $(f)\Rightarrow (i)$ are clear. Lemma \ref{l Kad sublemma 5 local generalized derivations} shows that $(e) \Rightarrow (f)$ and $(h)\Rightarrow (i)$. The implication $(f)  \Rightarrow (d)$ and the equivalence $(f)  \Leftrightarrow (f')$ follow from Proposition \ref{prop zero products implies gener der}.\smallskip

$(i)\Rightarrow (f)$ We suppose that $a T(b) c =0$ for every $a,b,c\in A$ with $a b= bc=0$. Let $a, b, c$ be elements in $M(A)$ with $a b= bc =0$. We may assume that $a,b$ and $c$ lie in the closed unit ball of $M(A)$. For each element $d$ in $M(A)$, we consider its polar decomposition $d = u_d |d|$ in $A^{**}$, where $u$ is a (unique) partial isometry in $A^{**}$, $|d|= (d^* d)^{\frac12},$ and $u_d^* u_d$ coincides with the range projection of $|d|$ in $A^{**}$ (cf. \cite[Theorem 1.12.1]{Sak}). The symbol $d^{[\frac13]}$ will denote the element $u_d |d|^{\frac13}\in M(A)$. It is easy to see that  $$d^{[\frac13]} (d^{[\frac13]})^* d^{[\frac13]} = u_d |d|^{\frac13} |d|^{\frac13} u_d^* u_d |d|^{\frac13} = u_d |d| = d.$$ The condition $a b= 0$ implies that $(a^* a)^m (b b^*)^k =0$, for every $m,k\in \mathbb{N}$. Considering the von Neumann subalgebras of $A^{**}$ generated by $a^* a $ and $b b^*$, and having in mind the separate weak$^*$-continuity of the product of $A^{**}$, we deduce that $a  u_b u_b^*= |a| u_b u_b^* = |a|^{\frac13} u_b u_b^*=0$. This implies that $$a^{[\frac13]} b^{[\frac13]} = u_a |a|^{\frac13} u_b |b|^{\frac13} =u_a |a|^{\frac13} u_b u_b^* u_b |b|^{\frac13} = 0
$$ and similarly $ b^{[\frac13]} c^{[\frac13]} =0$.\smallskip

Since $M(A)$ is a C$^*$-subalgebra of $A^{**}$, by weak$^*$-density of $A$ in $A^{**}$, we can take nets $(x_{\lambda})$, $(y_{\mu})$ and $(z_{\nu})$ in the closed unit ball of $A$, converging in the weak$^*$ topology of $A^{**}$ to $a^{[\frac13]}$, $b^{[\frac13]}$ and $c^{[\frac13]}$, respectively. The nets $\left(a^{[\frac13]} x_{\lambda}^* a^{[\frac13]}\right)$, $ \left( b^{[\frac13]} y_{\mu}^* b^{[\frac13]}\right) $, and $\left(c^{[\frac13]} z_{\nu}^* c^{[\frac13]}\right)$ lie in $A$ because $a,b,c\in M(A)$. The identities $ b^{[\frac13]} c^{[\frac13]} =  a^{[\frac13]} b^{[\frac13]} =0$ assure that  $$\left(a^{[\frac13]} x_{\lambda}^* a^{[\frac13]}\right) \left( b^{[\frac13]} y_{\mu}^* b^{[\frac13]}\right) =0 = \left( b^{[\frac13]} y_{\mu}^* b^{[\frac13]}\right) \left(c^{[\frac13]} z_{\nu}^* c^{[\frac13]}\right),$$ for every $\lambda, \mu$ and $\nu$. By assumption
$$\left(a^{[\frac13]} x_{\lambda}^* a^{[\frac13]}\right) T\left( b^{[\frac13]} y_{\mu}^* b^{[\frac13]}\right) \left(c^{[\frac13]} z_{\nu}^* c^{[\frac13]}\right)=0,$$ for every $\lambda, \mu$ and $\nu$. Taking weak$^*$ limit in $\nu$, it follows from the properties of $\pi_2^{***}$ (the second module operation in $X^{**}$) that $$\left(a^{[\frac13]} x_{\lambda}^* a^{[\frac13]}\right) T\left( b^{[\frac13]} y_{\mu}^* b^{[\frac13]}\right) c=0,$$ for every $\lambda,$ and $\mu$. Finally, taking weak$^*$ limits first in $\mu$ and later in $\lambda$, we have $ a T^{**} (b) c=0.$ This proves that $(i)\Rightarrow (f)$. The implication $(i')\Rightarrow (f')$ follows similarly.\smallskip

We shall finally prove that $(d)\Rightarrow (a)$. Suppose that  $S=T^{**}|_{M(A)} : M(A) \to X^{**}$ is a generalized derivation. Since $M(A)$ is a unital C$^*$-algebra and $X^{**}$ is a unital $M(A)$-bimodule, we can argue as in the final part of Remark \ref{remark generalized and derivations bitransposed} to deduce that $S^{**}= (T^{**}|_{M(A)})^{**} : M(A)^{**} \to X^{****}$ is a generalized derivation with $$S^{**} (a b) = S^{**} (a) b + a S^{**} (b) - a T^{**} (1) b,$$ for every $a,b\in M(A)^{**}$. Since the bidual of $A$ regarded as a norm closed subspace of $M(A)^{**}$ identifies with the weak$^*$-closure of $A$ in $M(A)^{**}$ and $S^{**}|_{A^{**}} = T^{**}$, we have  $$T^{**} (a b) = T^{**} (a) b + a T^{**} (b) - a T^{**} (1) b,$$ for every $a,b\in A^{**}$.
\end{proof}

Our next goal is to obtain a proof of Johnson's Theorem \ref{thm Johnson local 2001} from the above results. The proof in the unital case is a straightforward consequence of the previous Theorem \ref{t multipliers 1}. Indeed, suppose $T: A \to X$ is a continuous local derivation of a unital C$^*$-algebra into a unital Banach $A$-bimodule. Lemma \ref{l Kad sublemma 5} shows that $aT(b)c=0$, whenever $ab=bc=0$ in $A$. Theorem \ref{t multipliers 1} $(i)\Rightarrow (g)$ implies that $T$ is a generalized derivation, that is, there exists $\xi\in X^{**}$ satisfying $$T(ab) = T(a) b + a T(b) + a \xi b,$$ for every $a,b\in A$. Since $A$ and $X$ are unital, we have $\xi=T(1)\in X.$ The hypothesis of $T$ being a local derivation, implies the existence of a derivation $D_1 : A\to X$ such that $\xi=T(1) = D_1 (1) =0,$ which assures that $T$ is a derivation.\smallskip

\begin{proof}[Proof of Theorem \ref{thm Johnson local 2001}]
Suppose now that $A$ is a general C$^*$-algebra and $X$ is a Banach $A$-bimodule (not assumed to be essential). Let $T:A\to X$ be a bounded local derivation. Let $A_1= A\oplus 1 \mathbb{C}$ denote the unitization of $A$. The Banach space $X$ becomes a unital $A_1$-bimodule if we put $(a+\lambda 1) x = a x + \lambda x$ and $x(a+\lambda 1) = x a+ \lambda x$. The mapping $\widehat{T}: A_1\to X$, $\widehat{T} (a+\lambda 1) = T(a)$ is a continuous local derivation from a unital C$^*$-algebra into a unital Banach $A_1$-bimodule. The previous arguments show that $\widehat{T}$ is a derivation. In particular, $$T(a b) = \widehat{T} (a b)  = \widehat{T} (a) b + a \widehat{T} (b) = T(a) b + a T(b),$$ for every $a,b\in A,$ which proves that $T$ is a derivation.
\end{proof}

We can also deal now with the question we left open in Remark \ref{remark generalized and derivations bitransposed}. Let $G: A\to X$ be a continuous generalized derivation from a C$^*$-algebra into an essential Banach $A$-bimodule. Theorem \ref{t multipliers 1} implies that $G^{**}$ also is a generalized derivation.

\begin{example}\label{remark counterexample local derivations are not derivations} C$^*$-algebras constitute an idyllic setting where local derivations and derviations coincide. However, this good behavior is no longer true out from this special class of algebras. Kadison presented in \cite[\S 3]{Kad90} an example, based on ideas of Jensen, of a local derivation on the infinite dimensional commutative algebra, $\mathbb{C}(x)$, of all the rational functions in the variable $x$ over $\mathbb{C}$, which is not a derivation. It is shown in this example, that every derivation $\delta$ on $\mathbb{C}(x)$ is of the form $\delta(f) = \delta(x) f^\prime,$ while local derivation on $C(x)$ are precisely the linear mappings that annihilate the constant functions. The projection $\pi$ of $\mathbb{C}(x)$ onto the complement of the subspace generated by $1$ and $x$ vanishes on constant functions, and thus $\pi$ is a local derivation. However, $\pi$ cannot be a derivation, because in such a case $\pi (f)= \pi (x) f^\prime =0,$ for every $f\in \mathbb{C}(x)$, which contradicts $\pi\neq 0.$\smallskip

Let $\mathbb{C}[x]$ be the algebra of polynomial functions in the variable $x$. In a Note added in proof, Kadison aggregates that Kaplansky found a local derivations of $\mathbb{C}[x]/[x^3]$, a 3-dimensional algebra over $\mathbb{C}$, which is not derivation. We do not know if Kaplansky's example has been published or not.  For completeness reasons, the following  example has been borrowed from \cite[\S 5]{Bre07}. Let $T: \mathbb{C}[x]/[x^3] \to \mathbb{C}[x]/[x^3]$ be the linear mapping given by $T[\lambda_0 + \lambda_1 x+ \lambda_2 x^2] = [\lambda_1 x]$. Since $T[x^2] =0$ and $T[x] [x] = [x^2]$, we deduce that $T$ is not a derivation. Fix a point $a=[\lambda_0 + \lambda_1 x+ \lambda_2 x^2]$. We define a derivation $D_a : \mathbb{C}[x]/[x^3] \to \mathbb{C}[x]/[x^3],$ $$D_a [\alpha_0 + \alpha_1 x+ \alpha_2 x^2] = [\alpha_1 x + 2(\alpha_2 -\lambda_1^{-1} \lambda_2 \alpha_1) x^2], $$ if $\lambda_1\neq 0$, and $D_a =0 $ if $\lambda_1=0$. Clearly, $T(a) = D_a [a].$ Thus, $T$ is a local derivation.
\end{example}

\section{Hochschild cohomology of C$^*$-algebras and local n-cocycles}\label{ss: cohomology}

In the paper that originated the study of local derivations, Kadison (cf. \cite{Kad90}) also states that the study should be also extended to ``local higher cohomology (for example, local 2-cocycles)''.\smallskip

Let us recall some basic concepts. Let $X$ be a complex Banach $A$-bimodule over a C$^*$-algebra. Following standard notation, for each natural number $n$, the symbol $B(^n A, X)$ will denote the complex Banach space of all continuous $n$-multilinear mappings (also called $n$-cochains) from $A\times \ldots \times A$ into $X$. By convention, we set $B(^0 A, X)= X.$ Given $n\geq 1$, the \emph{$n$th-connecting map}, $\partial^{n}$, is defined as follows: $$\partial^{n} : B(^n A, X) \to B(^{n+1} A, X)$$ $$\partial^{n}T (a_1, \ldots, a_n,a_{n+1}):= a_1 T (a_2, \ldots, a_n,a_{n+1}) $$ $$+ \sum_{j=1}^{n} (-1)^{j} T (a_1, \ldots, a_{j-1},a_j a_{j+1}, \ldots, a_{n+1}) + (-1)^{n+1} T (a_1, \ldots, a_n) a_{n+1}$$ and $\partial^{0} : X \to B(A, X)$, $\partial^{0} (x) (a) := a x - x a$. It is known that $\delta^{n}\circ \delta^{n-1} =0$, for every $n\geq 1$.\smallskip

A continuous multilinear operator $\Phi\in B(^n A, X)$ is said to be an \emph{$n$-cocycle} when $\partial^{n}  \Phi= 0$. For example, 1-cocycles from $A$ into $X$ are precisely the derivations of $A$ into $X$. A bilinear mapping $\Phi: A \times A \to X$ is a 2-cocycle when the identity \begin{equation}\label{def 2-cocycles} a \Phi(b,c) - \Phi(ab,c) + \Phi(a,bc) - \Phi(a,b) c =0,\end{equation} holds for every $a,b,c\in A$.\smallskip

That is, for $n\geq 1$, the the kernel of the \emph{$n$th-connecting map} $\partial^{n}$, denoted by $Z^n(A,X),$ is the space of all $n$-cocycles of $A$ into $X$. The image of $\delta^{n-1}$ in $B(^n A, X)$ is the space of all co-boundaries of $A$ into $X$, and it is denoted by $C^n(A,X)$. The \emph{bounded $n$th Hochschild cohomology group of $A$ with coefficients in $X$} is the quotient vector space  $H^n (A, X): = Z^n(A,X) / C^n(A,X)$. By convention, $H^0(A, X) = \{x\in X : a x = x a, \forall a\in A\}$.\smallskip

A multilinear mapping $T: A\times \ldots \times A \to  X$ is said to be a \emph{local $n$-cocycle} if for every $a_1,\ldots, a_n$ in $A$, there exists an $n$-cocycle $\Phi_{_{a_1,\ldots, a_n}}$ (depending on $a_1,\ldots, a_n$) such that $T(a_1,\ldots, a_n) = \Phi_{_{a_1,\ldots, a_n}} (a_1,\ldots, a_n)$. The question posed by Kadison in his comments can be materialized as follows:

\begin{problem}\label{problem local n-cocycles b}
Is every continuous local $n$-cocycle of a C$^*$-algebra $A$ into a Banach $A$-bimodule an $n$-cocycle?
\end{problem}

We have already seen that Kadison solves this problem in the case in which $n=1$, $A$ is a von Neumann algebra and $X$ is a dual Banach $A$-bimodule. A complete positive solution for the case $n = 1$ was obtained by Johnson in \cite{John01}. In 2002, Zhang \cite{Zhang02} proves that each local 2-cocycle of a von Neumann algebra $M$ into a unital dual $M$-bimodule is a 2-cocycle. In 2007, Hou and Fu show that every local 3-cocycle of a von Neumann algebra $M$ into a unital dual $M$-bimodule is a 3-cocycle (cf. \cite{HouFu}). The definitive solution to Problem \ref{problem local n-cocycles b} was found by Samei, who proves that, for every $n \in \mathbb{N}$, bounded local $n$-cocycles of a C$^*$-algebra $A$ into a Banach $A$-bimodule $X$ are $n$-cocycles. We shall review the last result.\smallskip

The multilinear version of Lemma \ref{l Kad sublemma 5} for local $n$-cocycles reads as follows:

\begin{lemma}\label{l sublemma 5 cocycles} Let $T: A\times \ldots \times A \to  X$ be a multilinear mapping, where $A$ is a C$^*$-algebra and $X$ is a Banach $A$-bimodule. Suppose $T$ is a local $n$-cocycle. Then given $a_0, \ldots, a_{n+1}$ in $A$, with $a_j a_{j+1} = 0$ for every $j=0,1,\ldots, n$, we have $$a_0 T (a_1, \ldots , a_n) a_{n+1} = 0.$$
\end{lemma}

\begin{proof} The proof follows from the fact that for every \emph{$n$-cocycle} $\Phi\in B(^n A, X)$ and  $a_0, \ldots, a_{n+1}$ in the above hypothesis, $$a_0 \Phi (a_1, \ldots , a_n) a_{n+1} = \sum_{j=1}^{n-1} (-1)^{j+1} \Phi (a_0, \ldots, a_{j-1},a_j a_{j+1}, \ldots, a_{n}) a_{n+1}$$ $$ + (-1)^{n} \Phi (a_0, \ldots, a_{n-1}) a_{n} a_{n+1} =0.$$
\end{proof}

A multilinear mapping $T: A\times \ldots \times A \to  X$ satisfying the conclusion of Lemma \ref{l sublemma 5 cocycles} is termed \emph{$n$-hyperlocal} in \cite{Samei08}.\smallskip

Let $X$ be a Banach $A$-bimodule over a C$^*$-algebra $A$. We recall that
for each natural $m$, the Banach space $B(^m A, X)$ is a Banach $A$-bimodule with respect to the products defined by
$$(a \star T ) (a_1, \ldots , a_m) = a T (a_1, \ldots, a_m);$$
$$(T \star a) (a_1, \ldots, a_m) = T (a a_1, \ldots , a_m) +
\sum^m_{j=1}  (-1)^{j} T (a, a_1, \ldots , a_j a_{j+1}, \ldots , a_m) $$ $$+ (-1)^{m+1} T (a, a_1, \ldots , a_{m-1}) a_m,$$ compare \cite[Section 1.9]{Dales00}. When $m=1$ the module operations in $B(A,X)$ are given by $$(a\star T) (b) = a T(b), \hbox{ and } (T\star a) (b) = T(ab) - T(a) b.$$ It is further known that the mapping $$\Lambda_m : B(^{m+1}A, X)\to  B(^m A, (B(A, X),\star))$$ $$(\Lambda_m (T) (a_1, \ldots , a_m)) (a_{m+1}) = T (a_1, \ldots, a_{m+1})$$ is an $A$-bimodule isometric isomorphism \cite[Proposition 1.9.10]{Dales00}, and if $\Delta^{m}: B(^m A, (B(A, X), \star))\to B(^{m+1} A, (B(A, X), \star))$ denotes the corresponding $m$th-connecting map, then it is shown in \cite[\S 1.9]{Dales00} that the following diagram is commutative:
$$\begin{CD}
B(^{m+1}A, X) @>\ \ \ \ \Lambda_m \ \ \ \ >> B(^m A, (B(A, X),\star))\\
@VV\delta^{m+1} V @VV\Delta^{m}V\\
B(^{m+2}A, X) @>\ \ \ \ \Lambda_{m+1} \ \ \ \ >> B(^{m+1} A, (B(A, X),\star)),
\end{CD}$$ that is $\Delta^{m} \Lambda_m = \Lambda_{m+1} \delta^{m+1}$ \cite[1.9.13]{Dales00}.\smallskip

Unfortunately, in the case of $A$ being unital, the bimodule $B(^m A, X)$ need not be unital. Clearly $1\star T = T$ for $T\in B(^m A, X)$, but $ T \star 1\neq T$. However, if we consider the closed subspace $B_0(A, X)$ of all continuous operators $T\in B(A,X)$ satisfying $T(1) = 0$, it is easy to see that $B_0(A, X)$ is a closed submodule of $(B(A, X),\star)$, and the Banach $A$-bimodule $(B_0(A, X),\star)$ is unital.

\begin{proposition}\label{p Samei prop 3.2}\cite[Proposition 3.2]{Samei08} Given a natural number $n$, a unital C$^*$-algebra $A$ with unit 1, and a unital Banach $A$-bimodule $X$,  every continuous $n$-hyperlocal operator $T \in B(^n A, X)$ such that $T (a_1, \ldots , a_n)$ vanishes whenever any of $a_1, \ldots , a_n$ coincides with $1$, is an $n$-cocycle.
\end{proposition}

\begin{proof} We proceed by induction on $n$. The case $n=1$ is a direct consequence of Theorem \ref{t multipliers 1}. Suppose the statement is true for $n\geq 1$, $A$, $X$ and $T$ as above. Let  $T \in B(^{n+1} A, X)$ be an $(n+1)$-hyperlocal operator such that $T (a_1, \ldots , a_{n+1})$ vanishes whenever any of $a_1, \ldots , a_{n+1}$ is 1.\smallskip

We claim that $\Lambda_{n} (T)\in B(^n A, (B(A, X),\star))$ is an $n$-hyperlocal operator. Let $a_0, \ldots, a_{n+1}$ in $A$, with $a_j a_{j+1} = 0$ for every $j=0,1,\ldots, n$. We put $S = a_0 \star \Lambda_{n} (T) (a_1,\ldots, a_n) \star a_{n+1}: A\to X.$ Given $c,d\in A$ with $cd=0,$ we compute $$S(c) d= (a_0 \star \Lambda_{n} (T) (a_1,\ldots, a_n) \star a_{n+1}) (c) d $$ $$= (a_0 \star \Lambda_{n} (T) (a_1,\ldots, a_n) (a_{n+1} c)) d - ((a_0 \star \Lambda_{n} (T) (a_1,\ldots, a_n))(a_{n+1})) c d$$ $$= (a_0 \Lambda_{n} (T) (a_1,\ldots, a_n) (a_{n+1} c)) d  = a_0 T (a_1,\ldots, a_n ,a_{n+1} c) d= 0,$$ where in the last equality we have applied Lemma \ref{l sublemma 5 cocycles}, $a_j a_{j+1} = 0$ for $j=0\ldots,n+1$, $a_{n+1} c d=0$, and the fact that $T$ is an  $(n+1)$-hyperlocal operator. Therefore, $S$ is a left-annihilator preserving,  and hence, Lemma \ref{l left mulipliers non unital}$(a)$ implies that $S$ is a left multiplier, which proves that $S(a) = S(1) a$, for every $a\in A$. Since $$S(1) = a_0 \star \Lambda_{n} (T) (a_1,\ldots, a_n) \star a_{n+1} (1) $$ $$= a_0 \star \Lambda_{n} (T) (a_1,\ldots, a_n) (a_{n+1} 1)- (a_0 \star \Lambda_{n} (T) (a_1,\ldots, a_n))(a_{n+1}) 1 $$ $$= a_0 T(a_1,\ldots, a_{n+1}) - a_0 T(a_1,\ldots, a_{n+1}) =0,$$ we deduce that $S=0$, which shows that $\Lambda_{n} (T)$ is an $n$-hyperlocal operator.\smallskip

Now, if any of $a_1, \ldots, a_n$ coincides with $1$, it follows from the assumptions that $\Lambda_{n} (T) (a_1,\ldots, a_n) (a) = T(a_1,\ldots, a_n, a ) =0$, for every $a\in A$, thus $\Lambda_{n} (T) (a_1,\ldots, a_n)=0$. We also observe that the same assumption also implies that $\Lambda_{n} (T) (b_1,\ldots, b_n) (1) = T(b_1,\ldots, b_n, 1) =0,$ for every $b_1,\ldots, b_n\in A.$ This shows that $\Lambda_{n} (T) (b_1,\ldots, b_n) \in B_0(A,X)$ for every $b_1,\ldots, b_n\in A.$\smallskip

Therefore, $\Lambda_{n} (T) : A\times \ldots\times A\to B_0(A,X)$ is a continuous $n$-hyperlocal multilinear operator of $A$ into a unital Banach $A$-bimodule satisfying that $\Lambda_{n} (T) (a_1,\ldots, a_n)=0$ whenever any of $a_1, \ldots, a_n$ coincides with $1$. We conclude by the induction hypothesis that $\Lambda_{n} (T)$ is an $n$-cocycle (i.e. $\Delta^{n} \Lambda_n (T) = 0$). Let $\Delta^{n+1}$ and $\delta^n$ denote the $(n+1)$th- and the $n$th-connecting mappings of $A$ into $X$ and into $B(A,X)$ respectively. Since $0=\Delta^{n} \Lambda_n (T) = \Lambda_{n+1} \delta^{n+1} (T),$ we can conclude that $$0=\Lambda_{n+1} (\delta^{n+1} (T)) (a_1,\ldots, a_{n+1}) (a_{n+2}) = \delta^{n+1} (T) (a_1,\ldots, a_{n+1}, a_{n+2}),$$ for every $a_1,\ldots, a_{n+1}, a_{n+2}$ in $A$, which shows that $\delta^{n+1} (T)=0$ as desired.
\end{proof}

\begin{theorem}\label{t local n-cocycles}\cite[Theorem 5.4]{Samei08} Let $A$ be a C$^*$-algebra, and let $X$ be a Banach $A$-bimodule. Then, for
every natural $n$, every bounded local $n$-cocycle $T:A\times \ldots \times A \to X$ is an $n$-cocycle.
\end{theorem}

\begin{proof} Arguing as in the proof of Theorem \ref{thm Johnson local 2001}, let $A_1= A\oplus 1 \mathbb{C}$ denote the unitization of $A$. The Banach space $X$ becomes a unital $A_1$-bimodule if we put $(a+\lambda 1) x = a x + \lambda x$ and $x(a+\lambda 1) = x a+ \lambda x$.
The mapping $\widehat{T}: A_1\times \ldots \times A_1 \to X$, $\widehat{T} (a_1+\lambda_1 1, \ldots, a_n +\lambda_n 1) = T(a_1,\ldots,a_n)$ is a continuous local $n$-cocycle from a unital C$^*$-algebra into a unital Banach $A_1$-bimodule. Actually $T$ is a (local) $n$-cocycle if and only if $\widehat{T}$ is a (local) $n$-cocycle.  Clearly, $\widehat{T} (a_1+\lambda_1 1, \ldots, a_n +\lambda_n 1) =0$ whenever there exists $j:1,\ldots, n$ with $a_j+\lambda_j 1= 1$. Proposition \ref{p Samei prop 3.2} implies that  $\widehat{T} $ (and hence $T$) is an $n$-cocycle.
\end{proof}

\section{Local triple derivations}\label{ss: local triple der}

Generalized derivations revisited in the previous section actually constitute the first connection with the ternary structure underlying a C$^*$-algebra. We recall that every C$^*$-algebra can be equipped with a ternary product of the form \begin{equation}\label{eq ternary product on C*-algebras}\{a,b,c\} = \frac12 (a b^* c + c b^* a).
 \end{equation}When $A$ is equipped with this product it becomes a JB$^*$-triple in the sense we shall see later. A linear mapping $\delta: A\to A$ is said to be a \emph{triple derivation} when it satisfies the (triple) Leibnitz rule: \begin{equation}\label{equ ternary Leibnitz rule} \delta\{a,b,c\} = \{\delta(a),b,c\} + \{a,\delta(b),c\}+ \{a,b,\delta(c)\}.
 \end{equation} Given $a,b\in A$ we define a linear mapping $L(a,b) : A\to A$ by the assignment $c\mapsto \{a,b,c\}$. It is easy to check that the linear operator $\delta(a,b) = L(a,b)-L(b,a)$ is a bounded triple derivation on $A$.\smallskip

Barton and Friedman establish in \cite{BarFri} that every triple derivation on a C$^*$-algebra and on a JB$^*$-triple is continuous.\smallskip

According to the definition introduced by Burgos, Fernández-Polo, Garcés and the third author of this note in \cite{BurFerGarPe2014} and \cite{BurFerPe2013}, a linear map $G$ from a C$^*$-algebra $A$ into a Banach $A$-bimodule $X$ is a \emph{generalized Jordan derivation} if there exists $\xi\in X^{**}$ such that the identity $$G (a\circ b) = G(a)\circ b + a\circ G(b) - U_{a,b} (\xi),$$ holds for every $a,b$ in $A$, where $U_{a,b} (z) := (a\circ z) \circ b + (b\circ z) \circ a- (a\circ b) \circ z$. If $A$ is unital, every generalized (Jordan) derivation $D : A\to X$  with $D(1) =0$ is a (Jordan) derivation.\smallskip

Suppose now that $A$ is a unital C$^*$-algebra and $\delta : A\to A $ is a ternary derivation. The identity $$\delta(1) =\delta\{1,1,1\}= 2   \{\delta(1),1,1\} + \{1,\delta(b),1\} = 2 \delta(1) + \delta(1)^{*}$$ implies that $\delta(1)^* = - \delta(1)$, and thus $$\delta (a\circ b) = \delta \{a,1,b\} = \{\delta(a),1,b\} + \{a,\delta(1),b\}+ \{a,1,\delta(b)\}$$ $$= \delta (a) \circ b + a \circ \delta(b)+ U_{a,b} (\delta(1)^* ) =  \delta (a) \circ b + a \circ \delta(b)- U_{a,b} (\delta(1) ),$$ which shows that $\delta$ is a generalized Jordan derivation (and automatically continuous by \cite{BarFri}).

\begin{lemma}\label{l triple derivations on unital C*-algebras are generalized derivations} Let $A$ be a unital C$^*$-algebra. Let $\delta: A\to A$ be a triple derivation. Then $\delta(1)= -\delta(1)^*$. Furthermore, every triple derivation on $A$ is a generalized Jordan derivation.$\hfill\Box$
\end{lemma}

The reciprocal statement of the above lemma is not always true; for example, let $a$ be an element in $A$,
the mapping $G_{a} :A \to A$, $x\mapsto G_{a} (x):= a \circ x$, is a generalized derivation on $A$. Since, $G_{a} (1) = 2 a,$ it follows that $G_{a}$ is not a ternary derivation whenever $a^* \neq -a$.\smallskip

Let $X$ be a Banach $A$-bimodule over a C$^*$-algebra. A linear mapping $T : A \to X$ is called a \emph{local generalized (Jordan) derivation} if for each $a\in A$, there exists a generalized (Jordan) derivation $G_a : A \to X$ satisfying $T(a) = D_a (a)$.\smallskip

It was noticed by Burgos, Fernández-Polo, Garcés and Peralta that generalized Jordan derivations and generalized derivations on a unital C$^*$-algebra define the same class of operators (cf. \cite[Remark 8]{BurFerGarPe2014}). The following result is a generalization of the above fact.

\begin{proposition}\label{p generalized Jordan der are gene der} Let $G: A\to X$ be a bounded linear map from a C$^*$-algebra into an essential Banach $A$-bimodule. Then $T$ is a generalized Jordan derivation if and only if it is a generalized derivation.
\end{proposition}

\begin{proof} Suppose $G$ is a continuous generalized Jordan derivation, that is, there exists $\xi\in X^{**}$ satisfying $$G(a\circ b) = G(a)\circ b + a\circ G(b) - U_{a,b} (\xi),$$ for every $a,b\in A$.\smallskip

Pick $a,b,c\in A_{sa},$ $d\in A$ with $ab= bc=0$, $ad= dc=0$. Then \begin{equation}\label{eq a b2 c} a G(d^2) c = a (2 G(d)\circ d - d\xi d) c =a ( G(d) d + dG(d) - d\xi d ) c = 0.
 \end{equation}Let us write $b = b^+ -b^-$, where $b^+, b^-\geq 0$ in $A$ with $b^{+} b^- =0 = b^- b^+$. Since $a b= 0,$ we deduce that $a (b^+)^2 = a b b^+ =0$, which implies that $a (b^+)^{\frac12}= a b^+ =0.$ We similarly get $ a(b^-)^{\frac12} = (b^+)^{\frac12} c = (b^-)^{\frac12} c=0$.
If we put $d= (b^+)^{\frac12} +  i (b^-)^{\frac12}$, we have $a d= d c=0$, and by \eqref{eq a b2 c} $$a G(b) c = a G(d^2) c = 0.$$

We have shown that $a G(b) c= 0,$ for every $a,b,c\in A_{sa}$ with $ab= bc=0.$ Theorem \ref{t multipliers 1} implies that $G$ is a generalized derivation.
\end{proof}

The next Corollary is a straightforward consequence of the above Proposition \ref{p generalized Jordan der are gene der} and Theorem \ref{t multipliers 1}.

\begin{corollary}\label{c local Jordan der} Let $G: A\to X$ be a bounded linear map from a C$^*$-algebra into an essential Banach $A$-bimodule. Then $T$ is a local generalized Jordan derivation if, and only if, $T$ is a local generalized derivation if, and only if, $T$ is a generalized derivation.$\hfill\Box$
\end{corollary}

Let $A$ be a C$^*$-algebra. A linear mapping $T : A \to A$ is said to be a \emph{local triple derivation} if for each $a\in A$ there exists a triple derivation $\delta_a : A \to A$ satisfying $T(a) = \delta_a (a)$. Suppose that $A$ is unital. Lemma \ref{l triple derivations on unital C*-algebras are generalized derivations} implies that $T(1) = \delta_1 (1) = -\delta(1)^* = -T(1)^*$. Furthermore, Lemma \ref{l triple derivations on unital C*-algebras are generalized derivations} and Proposition \ref{p generalized Jordan der are gene der} assure that every bounded triple derivation on $A$ is a generalized derivation. Corollary \ref{c local Jordan der} (see also Theorem \ref{t multipliers 1}) implies that every bounded local triple derivation on $A$ is a generalized derivation. Moreover, the mapping $\frac12 \delta(T(1),1): A\to A$ is a triple derivation with $\frac12 \delta(T(1),1) (1) = \frac12 (T(1)-T(1)^*) = T(1)$. Thus  $\widetilde{T}=T- \frac12 \delta(T(1),1)$ is a local triple derivation satisfying $\widetilde{T} (1) =0$. It follows from the above that $\widetilde{T}$ is a generalized derivation with $\widetilde{T} (1) =0$, which implies that $\widetilde{T}$ is a derivation.

\begin{corollary}\label{c local bd triple der on unital} Every continuous local triple derivation on a unital C$^*$-algebra $A$ is a generalized derivation. Furthermore, let $\delta: A\to A$ be a bounded local triple derivation, then the mapping $\widetilde{T}=T- \frac12 \delta(T(1),1) = T - \delta\left(\frac12 T(1),1\right)$ is a generalized derivation with $\widetilde{T} (1) =0$, and consequently $\widetilde{T}$ is a derivation. $\hfill\Box$
\end{corollary}

In order to complete our study on bounded local triple derivations, we shall require the following result borrowed from \cite{BurFerGarPe2014}.

\begin{lemma}\label{l local triple derivation is symmetric}\cite[Lemma 9]{BurFerGarPe2014} Let $A$ be a unital C$^*$-algebra,
and let $T : A\to A$ be a bounded local triple derivation with $T(1) = 0$.
Then $T$ is a symmetric operator, that is, $T(a^*) = T(a)^*$, for every $a\in B$.
\end{lemma}

\begin{proof} Corollary \ref{c local bd triple der on unital} assures that $T$ is a derivation. Let $u$ be a unitary element in $A$. Since $T$ is a derivation, we have $0=T(1)=T(u u^*)=u T(u^{*})+T(u) u^{*},$ and hence \begin{equation}\label{eq first part unitary} T(u)=-u T(u^{*}) u.
\end{equation}

By hypothesis, $T$ is a local triple derivation. Therefore there exists a triple derivation $\delta_u$ such that $T(u)=\delta_u (u)$. This implies that $T(u)=\delta_u (u) = \delta_{u} (u u^{*} u)= \delta_{u} \{ u,u,u \}=2 \{u,u, T (u) \}+ \{u,T(u),u \}= 2 T(u)+ u T(u)^* u,$
which proves $$T(u)= - u  T(u)^* u.$$ Combining the above identity with \eqref{eq first part unitary} we get $T(u^*)=T(u)^*$.\smallskip

Finally, from the Russo-Dye theorem and the continuity of $T$ we obtain that $T(b^*) = T(b)^*$, for every $b$ in $A$, which concludes the proof.
\end{proof}

We can state now the main result of \cite{BurFerGarPe2014}, which can be regarded as the first generalization in the triple setting of the results proved by Kadison and Johnson for local derivations on C$^*$-algebras.

\begin{theorem}\label{th Kadison thm for triple derivations on unital C*-algebras}\cite[Theorem 10]{BurFerGarPe2014} Every bounded local triple derivation on a unital C$^*$-algebra is a triple derivation.
\end{theorem}

\begin{proof}
Let $T: A\to A$ be a bounded local triple derivation on a unital C$^*$-algebra. Corollary \ref{c local bd triple der on unital} assures that the mapping $\widetilde{T}=T- \frac12 \delta(T(1),1) = T - \delta\left(\frac12 T(1),1\right)$ is a bounded generalized derivation with $\widetilde{T} (1) =0$, and consequently $\widetilde{T}$ is a derivation. Since $\widetilde{T}$ is a bounded local triple derivation and $\widetilde{T} (1) =0$, Lemma \ref{l local triple derivation is symmetric} implies that  $\widetilde{T}$ is a $^*$-derivation on $A$ (i.e.  $\widetilde{T} (a)^* =  \widetilde{T} (a^*)$, for every $a\in A$). It is easy to check that, in these conditions, we have $$ \widetilde{T} \{a,b,c\} =  \{\widetilde{T}(a), b,c\} +  \{a, \widetilde{T}(b),c\} +  \{a,b,\widetilde{T}(c)\},$$ which shows that $\widetilde{T}$ is a triple derivation, and hence $T = \widetilde{T} + \delta\left(\frac12 T(1),1\right)$ is a triple derivation too. \end{proof}

It should be remarked here that a derivation $D$ on a unital C$^*$-algebra is a triple derivation if, and only if, it is symmetric (i.e. $D(a^*) =D(a)^*$, for every $a$). Indeed, the identity $D(a^*) = D\{1,a,1\} = 2 \{D(1), a, 1\} + \{1,D(a),1\} = 2 D(1)\circ a^* + D(a)^* = D(a)^*$ proves the desired statement, because, as we said in the proof of the previous theorem, the other implication can be easily checked. This fact was implicitly stated by Barton and Friedmann \cite{BarFri} and by Ho, Martínez, Russo and the third author of this note \cite{HoMarPeRu}.\smallskip

In the light of Theorem \ref{th Kadison thm for triple derivations on unital C*-algebras}, it seems natural to ask whether every continuous local triple derivation on a general C$^*$-algbera is a triple derivation. This question is a very particular case of a more ambitious problem, introduced in the setting of JB$^*$-triples by Mackey in \cite{Mack}.\smallskip

We have already mentioned that every C$^*$-algebra is a JB$^*$-triple with respect to the triple product defined by \eqref{eq ternary product on C*-algebras}. The general definition of JB$^*$-triples, introduced by Kaup in \cite{Ka83}, reads as follows:\smallskip

A JB$^*$-triple is a complex Banach space $E$ together with a
continuous triple product $\J ... : E\times E\times E \to E,$
which is conjugate linear in the middle variable and symmetric
bilinear in the outer variables satisfying the following axioms:
\begin{enumerate}[{\rm $(a)$}] \item (Jordan Identity) \begin{equation}\label{Jordan identity} \J ab{\J xyz} = \J {\J abx}yz - \J x{\J bay}z + \J xy{\J abz},
\end{equation} for all $a,b,x,y,z$ in $E$;
\item If $L(a,b)$ denotes the operator on $E$ given by $L(a,b) x = \J abx,$ the mapping $L(a,a)$ is an hermitian operator with non-negative
spectrum; \item $\|\J aaa\| = \|a\|^3$, for every $a\in E$.\end{enumerate} Given $a,b\in E$, the symbol $Q(a,b)$ will denote the conjugate linear operator defined by $Q(a,b) (x) = \J axb$. We shall write $Q(a)$ instead of $Q(a,a)$.\smallskip

Every C$^*$-algebra is a JB$^*$-triple via the triple product given in \eqref{eq ternary product on C*-algebras} and every JB$^*$-algebra (i.e. a complex Jordan Banach $^*$-algebra satisfying  $\|U_a ({a^*}) \|= \|a\|^3$, for every element $a$, where $U_a (x) :=2 (a\circ x) \circ a - a^2 \circ x$, cf. \cite[\S 3.8]{HancheStor}) is a JB$^*$-triple under the triple product \begin{equation}\label{eq JB*-alg-triple product} \J xyz = (x\circ y^*) \circ
z + (z\circ y^*)\circ x - (x\circ z)\circ y^*.
\end{equation} The space $B(H,K)$ of all bounded linear operators between complex Hilbert spaces, although rarely is a C$^*$-algebra, is a JB$^*$-triple with the product defined in \eqref{eq ternary product on C*-algebras}. In particular, every complex Hilbert space is a JB$^*$-triple.\smallskip

Additional illustrative examples of JB$^*$-triples are given by the so-called \emph{Cartan factors}. A Cartan factor of type 1 (also denoted by $I^{\mathbb{C}}$) is a JB$^*$-triple which coincides with the Banach space $B(H, K)$ of bounded linear operators between two complex Hilbert spaces, $H$ and $K$, where the triple product is defined by $\J xyz= 2^{-1}(xy^*z+zy^*x)$. Cartan factors of types 2 and 3 are the subtriples of $B(H)$ defined by $II^{\mathbb{C}} = \{ x\in B(H) : x=- j x^* j\} $ and $III^{\mathbb{C}} = \{ x\in B(H) : x= j x^* j\}$, respectively, where $j$ is a conjugation on $H$. A Cartan factor of type 4 or $IV$ is a complex spin factor, that is, a complex Hilbert space provided with
a conjugation $x \mapsto \overline{x}$, triple product $$\J x y z
= \left< x / y \right> z + \left< z / y \right> x - \left< x /
\bar z \right> \bar y,$$ and norm given by $\| x\|^2=\left< x / x
\right>+\sqrt {\left< x / x \right>^2-|\left< x / \overline x
\right>|^2}$. The Cartan factors of types 5 and 6 consist of matrices over the eight
dimensional complex Cayley division algebra $\mathbb{O}$; the type $VI$ is the space of all hermitian $3$x$3$ matrices over $\mathbb{O}$, while the type $V$ is the subtriple of $1$x$2$ matrices with entries in $\mathbb{O}$ (compare \cite{Loos77}, \cite{FriRu86}, \cite{DanFri87} and \cite{Ka97}).\smallskip

A \emph{triple derivation} on a JB$^*$-triple $E$ is a linear mapping $\delta: E\to E$ satisfying that $$\delta \J abc = \J {\delta(a)}bc + \J a{\delta(b)}c + \J ab{\delta(c)},$$ for every $a,b,c\in E$. A \emph{local triple derivation} on $E$ is a linear map $T : E\to E$ such that for each $a$ in $E$ there exists a triple derivation $\delta_{a}$ on $E$ satisfying $T(a) = \delta_a (a).$\smallskip

During the International Conference on Jordan Theory, Analysis and Related Topics, held in Hong Kong in 2012, Mackey posed the following problems:

\begin{problem}\label{problem local triple derivations}
Is every continuous local triple derivation on a JB$^*$-triple $E$ a triple derivation?
\end{problem}

\begin{problem}\label{problem local triple derivations continuous}
Is every local triple derivation on a JB$^*$-triple $E$ continuous?
\end{problem}

\begin{problem}\label{problem non bounded local triple derivations}
Is every local triple derivation on a JB$^*$-triple $E$ a triple derivation?
\end{problem}

A JBW$^*$-triple is a JB$^*$-triple which is also a dual Banach space. JBW$^*$-triples occupy in the category of JB-triples a similar space of that inhabited by von Neumann algebras in the class of C$^*$-algebras. Every JBW$^*$-triple admits a unique isometric predual, and a triple version of a theorem proved by Sakai in the setting of von Neumann algebras, asserts that the triple product of every JBW$^*$-triple
is separately weak$^*$ continuous (cf. \cite{BarTi}). The second dual, $E^{**}$, of a JB$^*$-triple $E$ is a JBW$^*$-triple \cite{Di86b}.\smallskip

A positive answer to the above Problem \eqref{problem local triple derivations} was given by Mackey under the additional hypothesis of $E$ being a JBW$^*$-triple (cf. \cite[Theorem 5.11]{Mack}). This result can be considered a Jordan-triple version of Kadison's original result for von Neumann algebras. Like Kadison's proof was based on the abundance of projections in any von Neumann algebra, Mackey's arguments make a strong use of the abundance of tripotents elements in JBW$^*$-triples.\smallskip

Every element $e$ in a JB$^*$-triple $E$ satisfying $\J eee =e$ is called tripotent.
When a C$^*$-algebra, $A$, is regarded as a JB$^*$-triple, the set of tripotents of $A$ is precisely the set of all
partial isometries in $A$.\smallskip

Associated with each tripotent $e$ in a JB$^*$-triple $E$, there is a
decomposition of $E$ (called \emph{Peirce decomposition}) in the
form: $$E=E_0(e)\oplus E_1(e)\oplus E_2(e),$$ where
$E_k(e)=\{x\in E:L(e,e)x=\frac k2 x\}$ for $k=0,1,2$.
The \emph{Peirce rules} are that $$\J {E_{i}(e)}{E_{j} (e)}{E_{k} (e)}\subseteq E_{i-j+k} (e)$$ if $i-j+k \in \{ 0,1,2\},$ and $\J {E_{i}(e)}{E_{j} (e)}{E_{k} (e)}=\{0\}$ otherwise. Moreover, $\J {E_{2} (e)}{E_{0}(e)}{E} = \J {E_{0} (e)}{E_{2}(e)}{E} =\{0\}.$\smallskip

The Peirce space $E_2 (e)$ is a unital JB$^*$-algebra with unit $e$,
product $x\circ_e y := \J xey$ and involution $x^{*_e} := \J
exe$, respectively. The corresponding \emph{Peirce projections}, $P_{i} (e) : E\to
E_{i} (e)$, $(i=0,1,2)$ are given by \begin{equation}\label{eq Peirce projections} P_2 (e) = Q(e)^2, \ P_1 (e) = 2 L(e,e) - 2 Q (e)^2,
\end{equation} $$\ \hbox{ and } P_0 (e) = Id-2 L(e,e) + Q (e)^2,$$ where $Id$ denotes the
identity map on $E$.\smallskip

The separate weak$^*$ continuity of the triple product of every JBW$^*$-triple implies that Peirce projections associated with a tripotent $e$ in a JBW$^*$-triple are weak$^*$ continuous.\smallskip

Elements $a,b$ in a JB$^*$-triple $E$ are said to be
\emph{orthogonal} (written $a\perp b$) if $L(a,b) =0$. It is known that $a\perp b$ if, and only if, $\J aab =0$ if, and only if, $\{b,b,a\}=0$ if, and only if, $b\perp a$ (see, for example, \cite[Lemma 1]{BurFerGarMarPe}).\smallskip

The triple version of Lemmas \ref{l Kad sublemma 5} and \ref{l Kad sublemma 5 local generalized derivations} read as follows:

\begin{lemma}\label{l 4 BurFerGarPe}\cite[Lemma 4]{BurFerGarPe2014} Let $T: E \to E$ be a local triple derivation on a JB$^*$-triple. Then the products of the form $\J a{T(b)}c$ vanish for every $a,b,c$ in $E$ with $a, c\perp b$.
\end{lemma}

\begin{proof} Let $a,b,$ and $c$ in the hypothesis of the lemma. Consider a triple derivation $\delta_{b} : E\to E$ satisfying $\delta_b (b)= T(b)$. Then $$\{ a, T(b) , c\} = \{a,\delta_b (b), c\} = \delta\{a,b,c\} - \{\delta_b (a),b, c\} - \{a,b,\delta_b (c)\} =0.$$
\end{proof}

The natural partial order in the set of tripotents of a JB$^*$-triple $E$ is defined as follows: given two tripotents $e$ and $u$ in $E$ we say that $u \leq e$ if $e-u$ is a tripotent in $E$ and $e-u \perp u$. \smallskip

Horn establishes in \cite[Lemma 3.11]{Horn}, that the set of tripotents in a JBW$^*$-triple $W$ is norm total. More precisely, every $a\in W$ be be approximated in norm by a finite linear combination of mutually orthogonal tripotents.

Let $T: E\to E$ be a local triple derivation on a JB$^*$-triple and $e$ be a tripotent in $E$. Clearly \begin{equation}\label{eq local der on a single tripotent} T(e) = \delta_e (e) = 2\{\delta_e (e),e,e\} + \{e,\delta_e (e),e\}
 \end{equation} $$= 2 \{T(e),e,e\} +\{e,T(e),e\}.$$   In particular, \begin{equation}\label{eq local der on a single tripotent 2} P_2 (e) (T(e)) = - \{e,T(e),e\} = - \{e,P_2 (e) T(e),e\} =- (P_2 (e) T(e))^{*_e},
 \end{equation} $$\hbox{ and }  P_0(e) (T(e))=0.$$
It is known that tripotents $e,v$ in $E$ are orthogonal if, and only if, $e\pm v$ is a tripotent in $E$ (compare, for example, \cite[3.6 Lemma]{IsKaRo95}). Therefore $$T(e\pm v) = 2 \{T(e\pm v),e\pm v,e\pm v\} +\{e\pm v,T(e\pm v),e\pm v\},$$ $$T(e)\pm T(v) = 2 \{T(e\pm v),e,e\} + 2 \{T(e\pm v),v,v\} +\{e\pm v,T(e\pm v),e\pm v\},$$ for every $e,v$ tripotents in $E$ with $e\perp v.$ This shows that $$2\{T(e),v,v\} + 2 \{v,T(v),e\} + \{v,T(e),v\} =0,$$ for all tripotents $e,v$ with $e\perp v$. We deduce, via Peirce arithmetic and \eqref{eq local der on a single tripotent 2}, that $\{v,T(e),v\} =0,$ and hence \begin{equation}
 \label{eq tripotents orthogonal} \{T(e),v,v\} + \{v,T(v),e\} =0,
 \end{equation}for all tripotents $e,v$ with $e\perp v$. Actually, a similar reasoning to that given above shows that \begin{equation}\label{eq orthgonal tripotents in the outer variables} \{v,T(e),w\} =0,
\end{equation} whenever $e,v$ and $w$ are tripotents in $E$ with $e\perp v,w.$\smallskip

Let us consider a finite family of mutually orthogonal tripotents $e_1,\ldots, e_n$ in $E$, and an algebraic element of the form $\displaystyle b=\sum_{i=1}^{m} \lambda_i e_i,$ where $\lambda_i\in \mathbb{R}$ for every $i.$ Clearly, $\displaystyle T(\J bbb) = \sum_{i=1}^{m} \lambda_i^{3} T^{**} (\J {e_i}{e_i}{e_i}).$ The condition $e_i\perp e_j$ for every $i\neq j$ implies that
\begin{equation}\label{eq 4 theorem main} 2 \J {T(b)}bb =2 \sum_{i,j=1}^{m} \lambda_i^2 \lambda_j \J {e_i}{e_i}{T (e_j)}
\end{equation} $$= 2\sum_{i=1}^{m} \lambda_i^3 \J {e_i}{e_i}{T^{**} (e_i)} + 2\sum_{i,j=1, i\neq j}^{m} \lambda_i^2 \lambda_j \J {e_i}{e_i}{T (e_j)};$$
\begin{equation}\label{eq 5 theorem main} \J b{T(b)}b =  \J {\sum_{i=1}^{m} \lambda_i e_i} {T\left(\sum_{j=1}^{m} \lambda_j e_j\right)}{\sum_{k=1}^{m} \lambda_k e_k} = \hbox{ (by ($\ref{eq orthgonal tripotents in the outer variables}$)) }
\end{equation} $$= \sum_{i=1}^{m} \lambda_i^3 \J {e_i}{T (e_i)}{e_i} +2 \sum_{i,j=1, i\neq j}^{m} \lambda_i^2 \lambda_j \J {e_i}{T (e_i)}{e_j}.$$

Now, applying \eqref{eq local der on a single tripotent} and \eqref{eq tripotents orthogonal}, we deduce that $$T \J bbb = 2 \J {T(b)}bb + \J b{T(b)}b.$$ Since every element $a$ in a JBW$^*$-triple $W$ can be approximated in norm, by algebraic elements of the form $\displaystyle b=\sum_{i=1}^{m} \lambda_i e_i$, we conclude that \begin{equation}\label{eq local der is a derivation on a single element} T\J aaa = 2 \J {T(a)}aa + \J a{T(a)}a,
\end{equation} for every bounded local triple derivation $T$ on $W$ and for every $a\in W.$ The following theorem can be deduced now via a standard polarization argument.

\begin{theorem}\label{t Mackey local triple derivations JBW}\cite[Theorem 5.11]{Mack}
 Every continuous local triple derivation on a JBW$^*$-triple is a triple derivation.
\end{theorem}

\begin{proof} Let $T: W\to W$ be a bounded local triple derivation on a JBW$^*$-triple.
By \eqref{eq local der is a derivation on a single element}, we have $T\J aaa = 2 \J {T(a)}aa + \J a{T(a)}a,$ for every $a\in W.$ Let $a,b$ be elements in $W$. Since $T\J {a\pm b}{a\pm b}{a\pm b} = 2 \J {T(a\pm b)}{a\pm b}{a\pm b} + \J {a\pm b}{T(a\pm b)}{a\pm b},$ it follows that
$$\pm 2 T\J aa b \pm T\J aba + 2 T \J bba + T\J bab$$ $$  = \pm 2 \J {T(a)}ab \pm 2 \J {T(b)}aa + 2 \J {T(b)}ab + 2 \J {T(b)}ba \pm 2 \J {T(a)}ba $$ $$+ 2 \J {T(a)}bb \pm 2 \J a{T(a)}b + \J b{T(a)}b + 2 \J b{T(b)}a \pm \J a{T(b)}a,$$ which implies
\begin{equation}\label{eq new nov 12}  2 T \J bba + T\J bab =  2 \J {T(b)}ab + 2 \J {T(b)}ba + 2 \J {T(a)}bb
\end{equation}
$$+ \J b{T(a)}b + 2 \J b{T(b)}a ,$$ for every $a,b\in W.$ If in \eqref{eq new nov 12}, we replace $b$ with $i b$ we get
$$2 T \J bba -T\J bab = - 2 \J {T(b)}ab + 2 \J {T(b)}ba + 2 \J {T(a)}bb$$ $$- \J b{T(a)}b + 2 \J b{T(b)}a ,$$ which added to \eqref{eq new nov 12} gives
$$T \{b,a,b\} =  2  \{T(b),a,b\} + \{ b,T(a),b\},$$ for every $a,b\in W.$ Having in mind that $$4 \{a,b,c\} = \{a+c,b,a+c\} - \{a-c,b,a-c\},$$ we deduce from the above identity that $$T \{b,a,c\} =  \{T(a),b,c\} + \{ a,T(b),c\} + \{a,b,T(c)\},$$ for every $a,b,c\in W.$
\end{proof}

The scarcity of tripotents in a general JB$^*$-triple makes intractable the application of the above arguments, appearing in Mackey's paper, in the wider setting of JB$^*$-triples. The answer to  Problem \ref{problem local triple derivations} in full generality will need some additional topological tools.\smallskip

Following standard notation, for each element $a$ in a JB$^*$-triple $E$ we denote $a^{[1]} =
a$ and $a^{[2 n +1]} := \J a{a^{[2n-1]}}a$ $(\forall n\in \mathbb{N})$. It can be easily deduced, via Jordan identity, that $\J{a^{[k]}}{a^{[l]}}{a^{[m]}}=a^{[k+l+m]}$. The symbol $E_a$ will denote the JB$^*$-subtriple of $E$ generated by the element $a$. The precise structure of $E_a$ is very well known, concretely, $E_a$ is JB$^*$-triple isomorphic (and hence isometric) to $C_0 (L)$ for some locally compact Hausdorff space $L\subseteq (0,\|a\|],$ such that $L\cup \{0\}$ is compact and $\|a\| \in L$. It is further known that there exists a triple isomorphism $\Psi$ from $E_a$ onto $C_{0}(L),$ satisfying $\Psi (a) (t) = t$ $(t\in L)$ (compare \cite[Lemma 1.14]{Ka83}). This implies that for each natural $n$, there exists (a unique) element $a^{[{1}/({2n-1})]}$ in $E_a$ satisfying $(a^{[{1}/({2n-1})]})^{[2n-1]} = a.$ When $a$ is a
norm one element, the sequence $(a^{[{1}/({2n-1})]})$ converges in the weak$^*$ topology of $E^{**}$ to a tripotent denoted by $r(a)$
and called the \emph{range tripotent} of $a$. The tripotent $r(a)$
is the smallest tripotent $e$ in $E^{**}$ satisfying that $a$ is
positive in the JBW$^*$-algebra $E^{**}_{2} (e)$. It is also known
that the sequence $(a^{[2n -1]})$ converges in the
weak$^*$ topology of $E^{**}$ to a tripotent (called the
\hyphenation{support}\emph{support} \emph{tripotent} of $a$)
$s(a)$ in $E^{**}$, which satisfies $ s(a) \leq a \leq r(a)$ in
$A^{**}_2 (r(a))$ (compare \cite[Lemma 3.3]{EdRu88}; the reader should be warned that in \cite{EdRu96}, $r(a)$ is called the support tripotent of $a$).\smallskip

We recall that a tripotent $u$ in the bidual of a JB$^*$-triple $E$ is said to be \emph{open} when $E^{**}_2 (u)\cap E$ is weak$^*$ dense in $E^{**}_2 (u)$. A tripotent $e$ in $E^{**}$ is said to be \emph{compact-$G_{\delta}$} (relative to $E$) if there exists a norm one element $a$ in $E$ such that $e$ coincides with $s(a)$, the support tripotent of $a$. A tripotent $e$ in $E^{**}$ is said to be \emph{compact} (relative to $E$) if there exists a decreasing net $(e_{\lambda})$ of tripotents in $E^{**}$ which are compact-$G_{\delta}$ with infimum $e$, or if $e$ is zero. The notions of open and compact tripotents in $E^{**}$ are introduced by Edwards and R{\"u}ttimann in \cite{EdRu96}. Closed and bounded tripotents in the bidual of a JB$^*$-triple are more recent notions developed in \cite{FerPe06} and \cite{FerPe07}. A tripotent $e$ in $E^{**}$ is said to be \emph{closed} relative to $E$ if $E^{**}_0 (e)\cap E$ is weak$^*$ dense in $E^{**}_0(e)$. If there exists a norm one element $a$ in $E$ such that $a= e + P_0 (e) (a)$, the tripotent $e$ is called \emph{bounded} (relative to $E$). The theory of compact tripotents is, in some sense, complete with Theorem 2.6 in \cite{FerPe06} (see also \cite[Theorem 3.2]{FerPe10b}), a result which asserts that a tripotent $e$ in $E^{**}$ is compact if, and only if, $e$ is closed and bounded.\smallskip

We shall also require some knowledge about the strong$^*$-topology of a JBW$^*$-triple. Let $W$ be a JBW$^*$-triple and let $\varphi$ a norm-one functional in $W_{*}$. By \cite[Proposition 1.2]{BarFri87}, the map $(x,y) \mapsto (.,.)_{\varphi}=\varphi \J xyz$ is a positive sequilinear form on $W$, where $z$ is any norm-one element of $W$ satisfying $\varphi (z) = 1$. Furthermore, the sequilinear form $(.,.)_{\varphi}$ does not depend on the choice of the point $z$, that is, if $\varphi(w)=1=\|w\|$ for any other $w\in W$, then $\varphi \J xyz = \varphi \J xyw$, for every $x,y\in W.$ Therefore, the mapping $x\mapsto \|x\|_{\varphi} = (x,x)_{\varphi}^{\frac12}$ is a pre-Hilbertian seminorm on $W$. The \emph{strong$^*$} topology of $W$, denoted by $S^{*} (W,W_{*})$, is the topology on ${W}$ generated by all the seminorms $ \|.\|_{\varphi},$ where $\varphi$ runs in the unit sphere of $W_{*}$ (cf. \cite[\S 3]{BarFri}). By  \cite[p.\ 258-259]{BarFri}, when a von Neumann algebra $M$ is regarded as a JBW$^*$-triple, its strong$^*$ topology coincides with its strong$^*$ topology in the von Neumann sense (i.e., the topology on $M$ generated by all the seminorms of the form $x\mapsto \varphi (x x^* + x^* x),$ where $\varphi$ runs in the set of normal positive functionals on $M$ \cite[1.8.7]{Sak}). We recall that the strong topology of $M$ is the topology generated by all the seminomrs of the form $x\mapsto \varphi (x x^* ),$ where $\varphi$ is a normal positive functionals on $M$.\smallskip

Among the many interesting properties of the strong$^*$-topology, we remark that the triple product of a JBW$^*$-triple is jointly strong$^*$-continuous on bounded sets (see \cite{Rodr1991} or \cite[Theorem 9]{PeRo}). It is also know that a linear functional on $W$ is weak$^*$-continuous if, and only if, it is strong$^*$-continuous, see \cite[Corollary 9]{PeRo}), thus, it follows from the bipolar theorem that for each convex $C\subseteq W$ we have
$$\overline{C}^{\sigma(W,W_*)} = \overline{C}^{S^*(W,W_*)}.$$ It is further known that a linear map between JBW$^*$-triples is strong$^*$-continuous if, and only if, it is weak$^*$-continuous (compare \cite[page 621]{PeRo}). Finally, we highlight that Bunce establishes in \cite{Bun01} that for a JBW$^*$-subtriple $F$ of a JBW$^*$-triple $W$, the strong$^*$-topology of $F$ coincides with the restriction to $F$ of the strong$^*$-topology of $W$, that is, $S^*(F,F_*) = S^* (W,W_*)|_{F}$.\smallskip

The following technical result can be directly derived from the properties of the strong$^*$-topology reviewed before, together with the expressions of the Peirce projections given in \eqref{eq Peirce projections}.

\begin{lemma}\label{l Peirce zero strong* limits}\cite[Lemma 1.5]{BurFerPe2013} Let $W$ be a JBW$^*$-triple. Suppose that $(e_\lambda)$ is a net (or a sequence) of tripotents in $W$ converging, in the strong$^*$-topology of $W$ to a tripotent $e$ in $W$. Let $(x_{\mu})$ be a net (or a sequence) in $W$, converging to some $x\in W$ in the strong$^*$-topology. Then, for each $i\in \{0,1,2\}$, the net (sequence) $P_i (e_\lambda) (x_\mu)$ tends to $P_i (e) (x)$. $\hfill\Box$
\end{lemma}

The key result in the study of local triple derivations is the following result, which was originally established in \cite{BurFerPe2013}.

\begin{proposition}\label{p loal triple derivations on compact tripotents}\cite[Proposition 2.2]{BurFerPe2013} Let $T : E \to E$ be a bounded local triple derivation on a JB$^*$-triple. Suppose $e$ is a compact tripotent in $E^{**}.$ Then the following statements hold: \begin{enumerate}[$(a)$]
\item $P_0 (e) T^{**}(e) =0$;
\item If $a$ is a norm one element in $E$ whose support tripotent is $e$ (that is, $e$ is a compact-$G_{\delta}$ tripotent), then $Q (e) T(a) = Q (e) T^{**} (e);$
\item $P_2(e) T^{**}(e) = - Q(e) (T^{**}(e)).$
\end{enumerate}
\end{proposition}

\begin{proof}
$(a)$ Lemma \ref{l Peirce zero strong* limits} assures that it is enough to prove the statement for compact-$G_{\delta}$. Indeed, suppose $e\in E^{**}$ is a compact tripotent.  Then there exists a decreasing
net $(e_{\lambda})$ of compact-$G_{\delta}$ tripotents in  $E^{**}$ converging to $e$ in the strong$^*$-topology of $E^{**}$. If the statement is true for compact-$G_{\delta}$ tripotents we have $P_0 (e_{\lambda}) T^{**}(e_{\lambda}) =0$ ($\forall \lambda$), and hence,  Lemma \ref{l Peirce zero strong* limits} implies that $P_0 (e) T^{**}(e) =0$.\smallskip

We consequently assume that $e$ is a compact-$G_{\delta}$ tripotent, that is, there exists a norm one element $a$ in $E$ with $s(a) =e$. Let $E_a$ denote the JB$^*$-subtriple of $E$ generated by $a$. We have already mentioned that there exists a subset $L\subseteq (0,1]$ with $1\in \{0\}\cup L$ compact and a triple isomorphism $\Psi$ from $E_a$ onto $C_0(L)$ such that $\Psi (a) (t) =t,$ $\forall t\in L$ (compare \cite[Lemma 1.14]{Ka83}). In this setting, we consider a sequence of norm one elements $(b_n)$ in $E_a$ such that $b_n= e + P_0 (e)(b_n),$ $\J {b_{n}}{b_{n+1}}{b_{n}}= \J {b_{n}}{b_{n}}{b_{n+1}} = b_{n+1}$, $(b_n) \to e,$ in the strong$^*$-topology of $E^{**}$. Take, for example,

\begin{figure}[ht]
\begin{center}
\begin{tikzpicture}[=>latex,thick,domain=-1:1,xscale=3,yscale=1.5,scale=1]
%ejes
\draw[<-]  (1.3,0) -- (-0.2,0) ;
\draw[<-]  (0,1.3) --(0,-0.2) ;
\coordinate [label=below:{\footnotesize$1$}] (uno) at (1,0);
\coordinate [label=left:{\footnotesize$1$}] (dos) at (0,1);
\draw[-]  (1,-0.04) --(1,0.1);
\draw[-]  (-0.04,1) --(0.011,1);
\coordinate [label=below:{\footnotesize $_{1-\frac{1}{n+1}}$}] (n) at  (0.8,0);
\draw[-]  (0.8,-0.04) --(0.8,0.1);
\coordinate [label=below:{\footnotesize $_{1-\frac{1}{n}}$}] (n2) at  (0.4,0);
\draw[-]  (0.4,-0.04) --(0.4,0.1);
\draw[-,dashed, gray, very thin]  (1,1) --(1,0);
\draw[-,dashed, gray, very thin]  (0.8,1) --(0.8,0);
\draw[-,dashed, gray, very thin]  (0,1) --(1,1);
%grafica
\draw[-,ultra thick,blue]  (0.8,1) --(1,1) ;
\draw[-,ultra thick,blue]  (0.4,0) --(0.8,1) ;
\draw[-,ultra thick,blue]  (0,0) --(0.4,0) ;
%comentario
\coordinate [label=right:{\small{\color{blue}$b_n(t)$}}] (f) at (0.3,1.1);
\coordinate [label=right:{\small{{$b_{n} (t):=\left\{%
\begin{array}{ll}
    0, & \hbox{if $0\leq t \leq 1-\frac1n$;} \\
    \hbox{affine}, & \hbox{if $1-\frac1{n}\leq t\leq 1-\frac1{n+1}$;} \\
    1, & \hbox{if $1-\frac1{n+1}\leq t\leq 1 $.} \\
\end{array}%
\right.$,}}}] (f) at (-2.5,0.5);
\end{tikzpicture}
\end{center}
\end{figure}

Fix a natural $n.$ Since, the support tripotent of $b_n,$ $s(b_n)$, is a compact tripotent in $E^{**}$, given $z,w\in E_0^{**} (s(b_n))$ we can find (bounded) nets $(c_\mu)$ and $(d_{\nu})$ in $E_0^{**} (s(b_n)) \cap E$ converging to $z$ and $w$ in the strong$^*$-topology of $E^{**},$ respectively. Since for each $\mu$ and $\nu$, $c_\mu$ and $d_{\nu}$ belong to $E_0^{**} (s(b_n))$, $b_{n+1}\in E_2^{**} (s(b_n))$ and $E_0^{**} (s(b_n)) \perp E_2^{**} (s(b_n))$, we clearly have, $c_\mu, d_{\nu} \perp b_{n+1}$ for every $\mu$ and $\nu$, and hence, by Lemma \ref{l 4 BurFerGarPe}, $$\J {c_\mu}{T(b_{n+1})}{d_\nu} =0 \ \ (\forall \mu, \nu).$$ Taking strong$^*$-limits in $\mu$ and $\nu$ we have $$\J {z}{T(b_{n+1})}{w} =0, \hbox{ for every } n\in \mathbb{N}, z,w\in E_0^{**} (s(b_n)),$$ equivalently, $$\J {P_0^{**} (s(b_n)) (x)}{T(b_{n+1})}{P_0^{**} (s(b_n))(y)} =0, \hbox{ for every } n\in \mathbb{N}, x,y\in E^{**}.$$

By the joint strong$^*$-continuity of the triple product of $E^{**}$, together with the strong$^*$-continuity of $T^{**}$ and Lemma \ref{l Peirce zero strong* limits}, we can take strong$^*$-limit in the above expression, to deduce that $\J {P_0^{**} (e)(x)}{T^{**}(e)}{P_0^{**} (e)(y)}=0,$ for every $x,y\in E^{**}$. It follows, for example, from Peirce arithmetic and the third axiom in the definition of JB$^*$-triples, that $P_0^{**} (e) T^{**}(e)=0$.\smallskip

$(b)$ Let $a$ be a norm one element in the hypothesis of statement $(b)$. We shall denote $a_0 = P_0 (e)(a)$. Again, we consider the JB$^*$-subtriple, $E_a,$ generated by $a$, and two sequences $(a_n)$ and $(b_n)$ in the closed unit ball of $E_a$ defined by
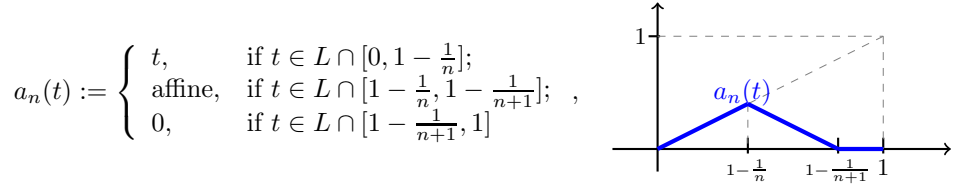
\begin{figure}[ht]
\begin{center}
\begin{tikzpicture}[=>latex,thick,domain=-1:1,xscale=3,yscale=1.5,scale=1]
%ejes
\draw[<-]  (1.3,0) -- (-0.2,0) ;
\draw[<-]  (0,1.3) --(0,-0.2) ;
\coordinate [label=below:{\footnotesize$1$}] (uno) at (1,0);
\coordinate [label=left:{\footnotesize$1$}] (dos) at (0,1);
\draw[-]  (1,-0.04) --(1,0.1);
\draw[-]  (-0.04,1) --(0.011,1);
\coordinate [label=below:{\footnotesize $_{1-\frac{1}{n+1}}$}] (n) at  (0.8,0);
\draw[-]  (0.8,-0.04) --(0.8,0.1);
\coordinate [label=below:{\footnotesize $_{1-\frac{1}{n}}$}] (n2) at  (0.4,0);
\draw[-]  (0.4,-0.04) --(0.4,0.1);
\draw[-,dashed, gray, very thin]  (1,1) --(0,0);
\draw[-,dashed, gray, very thin]  (1,1) --(1,0);
\draw[-,dashed, gray, very thin]  (0.4,0.4) --(0.4,0);
\draw[-,dashed, gray, very thin]  (0,1) --(1,1);
%grafica
\draw[-,ultra thick,blue]  (0.8,0) --(1,0) ;
\draw[-,ultra thick,blue]  (0.4,0.4) --(0.8,0) ;
\draw[-,ultra thick,blue]  (0,0) --(0.4,0.4) ;
%comentario
\coordinate [label=right:{\small{\color{blue}$a_n(t)$}}] (f) at (0.2,0.5);
\coordinate [label=right:{\small{{$a_{n} (t):=\left\{%
\begin{array}{ll}
    t, & \hbox{if $t\in L\cap [0,1-\frac1n]$;} \\
   \hbox{affine}, & \hbox{if $t\in L\cap [1-\frac1{n},1-\frac1{n+1} ]$;} \\
    0, & \hbox{if $t\in L\cap [1-\frac1{n+1},1 ]$} \\
\end{array}%
\right. ,$}}}] (f) at (-2.9,0.5);
\end{tikzpicture}
\end{center}
\end{figure}
and
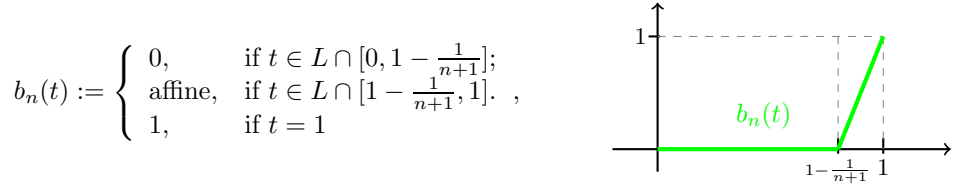
\begin{figure}[ht]
\begin{center}
\begin{tikzpicture}[=>latex,thick,domain=-1:1,xscale=3,yscale=1.5,scale=1]
%ejes
\draw[<-]  (1.3,0) -- (-0.2,0) ;
\draw[<-]  (0,1.3) --(0,-0.2) ;
\coordinate [label=below:{\footnotesize$1$}] (uno) at (1,0);
\coordinate [label=left:{\footnotesize$1$}] (dos) at (0,1);
\draw[-]  (1,-0.04) --(1,0.1);
\draw[-]  (-0.04,1) --(0.011,1);
\coordinate [label=below:{\footnotesize $_{1-\frac{1}{n+1}}$}] (n) at  (0.8,0);
\draw[-]  (0.8,-0.04) --(0.8,0.1);
%\coordinate [label=below:{\footnotesize $_{1-\frac{1}{n}}$}] (n2) at  (0.4,0);
%\draw[-]  (0.4,-0.04) --(0.4,0.1);
\draw[-,dashed, gray, very thin]  (1,1) --(1,0);
\draw[-,dashed, gray, very thin]  (0.8,1) --(0.8,0);
\draw[-,dashed, gray, very thin]  (0,1) --(1,1);
%grafica
\draw[-,ultra thick,green]  (0.8,0) --(1,1) ;
%\draw[-,ultra thick,blue]  (0.4,0) --(0.8,1) ;
\draw[-,ultra thick,green]  (0,0) --(0.8,0) ;
%comentario
\coordinate [label=right:{\small{\color{green}$b_n(t)$}}] (f) at (0.3,0.3);
\coordinate [label=right:{\small{{$b_{n} (t):=\left\{%
\begin{array}{ll}
    0, & \hbox{if $t\in L\cap [0, 1-\frac1{n+1}]$;} \\
    \hbox{affine}, & \hbox{if $t\in L\cap [1-\frac1{n+1},1 ]$.} \\
    1, & \hbox{if $t=1$}\\
\end{array}%
\right.$,}}}] (f) at (-2.9,0.5);
\end{tikzpicture}
\end{center}
\end{figure}

Clearly, $a_n\perp b_n$ ($\forall n$), $(a_n)\to a_0$ and $(b_n) \to e$ in the strong$^*$-topology of $E^{**}.$ Lemma \ref{l 4 BurFerGarPe} assures that $$\J {b_n}{T(a_n)}{b_n} = 0 \ \ (\forall n\in \mathbb{N}).$$ Taking strong$^*$ limits in the above expression we have $\J e{T^{**} (a_0)}e =0$, and hence $\J e{T^{**} (a)}e = \J e{T^{**} (e)}e.$\smallskip

$(c)$ By the arguments given at the beginning of the proof of $(a)$, we may assume that $e$ is a compact-$G_{\delta}$ tripotent in $E^{**}$. Let $a$ be a norm one element in $E$ such that $s(a) = e $. Since $T$ is a local triple derivation, we can find a triple derivation $\delta_{a}: E\to E$ such that $T(a) = \delta_{a} (a).$ We recall that by the separate weak$^*$-continuity of the triple product of $E^{**}$ and the weak$^*$-density of $E$ in $E^{**}$, we can easily see that $\delta_a^{**} : E^{**}\to E^{**}$ is a triple derivation on $E^{**}$ (compare the proof of \cite[Proposition 2.1]{HoPeRu}). Since $\delta_a$ is triple derivation, the identity in $(b)$ also holds whenever we replace $T$ with $\delta_a.$ Therefore, $$P_2(e) T^{**} (e) =  P_2 (e) T(a) = P_2(e) \delta_a (a) = P_2 (e) \delta_a^{**} (e) $$
$$\hbox{(by $(\ref{eq local der on a single tripotent 2})$)} = - Q(e) \Big(\delta_a^{**} (e)\Big) = - Q(e) \Big(\delta_a (a)\Big)  $$ $$= - Q(e) \Big(T (a)\Big) = - Q(e) \Big(T^{**} (e) \Big).$$
\end{proof}

Let $T:E\to E$ be a bounded local triple derivation on a JB$^*$-triple, and let $a$ be a norm one element in $E$. We consider again the JB$^*$-subtriple $E_a$ generated by $a$, and we identify $E_a$, via a triple isomorphism $\Psi$, with $C_0(L)$, where $L\subseteq (0,1]$ with $1\in \{0\}\cup L$ compact and $\Psi (a) (t) =t$ ($\forall t\in L$). Clearly, the range tripotent of $a$ can be approximated, in the strong$^*$ topology of $E^{**}$, by a sequence $(e_n)$ of compact-$G_{\delta}$ tripotents in $E^{**}$, that is, $(e_n) \to r(a)$ in the strong$^*$-topology. Since, by the above Proposition \ref{p loal triple derivations on compact tripotents}, $$P_0 (e_n) T^{**}(e_n) =0,$$ and $$P_2(e_n) T^{**}(e_n) = - Q(e_n) (T^{**}(e_n)),$$ taking strong$^*$-limit in $n$ we deduce, by Lemma \ref{l Peirce zero strong* limits}, that \begin{equation}\label{eq range tripotents in bidual}P_0 (r(a)) T^{**}(r(a)) =0, \hbox{ and, } P_2(r(a)) T^{**}(r(a)) = - Q(r(a)) (T^{**}(r(a))),
\end{equation} that is, the statement of Proposition \ref{p loal triple derivations on compact tripotents} also holds for range tripotents. Therefore, if  $e$ is a compact or a range tripotent in $E^{**}$,  it follows from Proposition \ref{p loal triple derivations on compact tripotents} and ($\ref{eq range tripotents in bidual}$) that $$T^{**} \J eee = T^{**} (e) = P_2(e) T^{**} (e) + P_1 (e) T^{**} (e),$$
$$2 \J ee{T^{**} (e)} = 2 P_2(e) T^{**} (e) + P_1 (e) T^{**} (e),$$ and $$\J e{T^{**}(e)}e = Q(e) T^{**} (e) = -P_2 (e) T^{**} (e),$$ which assures that \begin{equation}
\label{eq T on compact and range tripotents is a triple derivation} T^{**} \J eee = 2 \J ee{T^{**} (e)} + \J e{T^{**}(e)}e.
\end{equation}

We have seen that, for each bounded local triple derivation $T$ on a JB$^*$-triple $E$, $T^{**}$ behaves like a triple derivation on range tripotents and on compact tripotents. In order to apply triple spectral resolutions in the JB$^*$-subtriple generated by a single element we need to the behavior of $T$ when we mix orthogonal and range tripotents. The answer is given by the next result.

\begin{lemma}\label{c bitranspose compact range tripotents} Let $T : E \to E$ be a continuous local triple derivation on a JB$^*$-triple. Suppose that $e_1$ and $e_2$ are two orthogonal compact tripotents in $E^{**}$, $r_1$ and $r_2$ are two orthogonal range tripotent in $E^{**}$ with $e_1\perp r_2$. Then the identity $$T^{**} \J {u}{u}{u} = 2 \J {u}{u}{T^{**} (u)}+ \J {u}{T^{**}(u)}{u},$$ holds for $u={e_1\pm e_2},$ ${r_1\pm r_2}$, and ${e_1\pm r_2}.$
\end{lemma}

\begin{proof} It is established in \cite[Proposition 3.7]{FerPe10b} that the sum of two orthogonal compact tripotents is a compact tripotent. Therefore, $e_1\pm e_2$ is a compact tripotent in $E^{**}$, and hence, the identity for $u= e_1\pm e_2$ follows from \eqref{eq T on compact and range tripotents is a triple derivation}. The sum of two orthogonal range tripotents is again a range tripotent (compare, for example, \cite[Lemma 1]{BurFerGarMarPe}). So, the identity for $u= r_1\pm r_2$ also is a consequence of \eqref{eq T on compact and range tripotents is a triple derivation}.\smallskip

We consider now the case $u=e_1\pm r_2$. Since $e_1\perp r_2$ and $r_2$ is the range projection of a norm-one element $a$ in $E$, we can find a sequence of compact tripotents $(c_n)$ in $E^{**}$ (actually in the bidual of $E_a$) such that $c_n\leq r_2$, and hence $c_n\perp e_1$ for every $n$, and $(c_n)\to r_2$ in the strong$^*$-topology of $E^{**}.$ The first identity proved in this Lemma shows that $$T^{**} \J {e_1\pm c_{n}}{e_1\pm c_{n}}{e_1\pm c_{n}} = 2 \J {e_1\pm c_{n}}{e_1\pm c_{n}}{T^{**} (e_1\pm c_{n})}$$ $$+ \J {e_1\pm c_{n}}{T^{**}(e_1\pm c_{n})}{e_1\pm c_{n}},$$ for every $n$. Finally, if we take strong$^*$-limit in the above expression we get the desired equality.
\end{proof}

The main result for bounded local triple derivations follows now as a direct consequence of the above partial results.

\begin{theorem}\label{thm local triple derivations are triple derivations}\cite[Theorem 2.4]{BurFerPe2013} Every bounded local triple derivation on a JB$^*$-triple is a triple derivation.
\end{theorem}

\begin{proof} Let $T: E \to E$ be a bounded local triple derivation on a JB$^*$-triple $E$. Suppose that $e_1 , \ldots, e_m$ is a family of mutually orthogonal range or compact tripotents in $E^{**}$. Let us pick $i,j,k\in \{1,\ldots, m\}$ with $i,k \neq j.$ By Proposition \ref{p loal triple derivations on compact tripotents} and $(\ref{eq range tripotents in bidual})$ we know that $P_0 (e_j) T^{**} (e_j)=0.$ By assumptions, $e_i,e_k\in E^{**}_0 (e_j)$, which proves, via Peirce arithmetic, that \begin{equation}\label{eq 1 theorem main} \J {e_i}{T^{**}(e_j)}{e_k} = 0.
\end{equation}

Now, fix $i\neq j$ in $\{1,\ldots, m\}$. Since $e_i$ and $e_j$ are compact or range tripotents in $E^{**}$, Lemma \ref{c bitranspose compact range tripotents} implies that \begin{equation}\label{eq 1b theorem main} T^{**} \J {e_i }{e_i}{e_i} = 2 \J {e_i }{e_i }{T^{**} (e_i)} + \J {e_i }{T^{**}(e_i)}{e_i};
\end{equation} and $$ T^{**} \J {e_i \pm e_j}{e_i \pm e_j}{e_i \pm e_j} = 2 \J {e_i \pm e_j}{e_i \pm e_j}{T^{**} (e_i \pm e_j)}$$ $$ + \J {e_i \pm e_j}{T^{**}(e_i \pm e_j)}{e_i \pm e_j}.$$ Combining the last two identities we get $$\pm 2 \J {e_i}{e_i}{T^{**} (e_j)} + 2 \J {e_j}{e_j}{T^{**} (e_i)} \pm  \J {e_i }{T^{**}(e_j)}{e_i )} + \J {e_j}{T^{**}(e_i)}{ e_j} $$ $$\pm 2 \J {e_i}{T^{**}(e_i)}{ e_j} + 2 \J {e_i}{T^{**}( e_j)}{e_j} =0, $$ and consequently, $$+ 4 \J {e_j}{e_j}{T^{**} (e_i)} + 2 \J {e_j}{T^{**}(e_i)}{ e_j}+ 4 \J {e_i}{T^{**}( e_j)}{e_j} =0.$$ Applying ($\ref{eq 1 theorem main}$) we obtain \begin{equation}\label{eq 2 theorem main} \J {e_j}{e_j}{T^{**} (e_i)} + \J {e_i}{T^{**}( e_j)}{e_j} =0.
\end{equation}

Consider now an element $\displaystyle b=\sum_{i=1}^{m} \lambda_i e_i,$ where $e_1 , \ldots, e_m$ are as above and $\lambda_i\in \mathbb{R}$ for every $i$. The discussion before Theorem \ref{t Mackey local triple derivations JBW} combined with the identities established in $(\ref{eq 1b theorem main})$ and $(\ref{eq 2 theorem main})$ allow us to deduce  that $$T^{**} \J bbb = 2 \J {T^{**}(b)}bb + \J b{T^{**}(b)}b.$$ Since every element $a\in E$ can be approximated in norm, by elements of the form $\displaystyle b=\sum_{i=1}^{m} \lambda_i e_i$, where $e_1 , \ldots, e_m$ are as above and $\lambda_i\in \mathbb{R}$, we conclude that $$ T\J aaa = 2 \J {T(a)}aa + \J a{T(a)}a,$$ for every $a\in E.$  The arguments and polarization identities given at the end of the proof of Theorem \ref{t Mackey local triple derivations JBW} can be repeated now to show that $T$ is a triple derivation.
\end{proof}

\subsection{Real linear local triple derivations}

It seems natural to ask whether the conclusions in Theorems \ref{t Mackey local triple derivations JBW} and \ref{thm local triple derivations are triple derivations} would remain true if the hypotheses were weakened to continuous real linear local triple derivations. The arguments leading to Theorem \ref{thm local triple derivations are triple derivations} are valid to show that for every continuous real linear local triple derivation $T$ on a JB$^*$-triple $E$ satisfies that \begin{equation}\label{eq new real local triple derivation on a single element} T\J aaa = 2 \J {T(a)}aa + \J a{T(a)}a,
 \end{equation}for every $a\in E.$ However, the polarization arguments given at the end of the proof of Theorem \ref{t Mackey local triple derivations JBW} heavily depend on the complex linearity of the mapping $T$. So, it is not clear whether a real linear mapping $T$ on a JB$^*$-triple $E$ satisfying \eqref{eq new real local triple derivation on a single element} is a triple derivation. This is actually a problem to be studied in the setting of real JB$^*$-triples.\smallskip

We recall that a \emph{real JB$^*$-triple} is a norm-closed real subtriple of a (complex) JB$^*$-triple (cf. \cite{IsKaRo95}). Every real JB$^*$-triple $E$ can be also regarded as a \emph{real form} of a complex JB$^*$-triple, that is, there exist a (complex) JB$^*$-triple $E_{c}$ and a conjugate linear isometry $\tau: E_c\rightarrow E_c$ of period 2 such that $E=\{b\in E_c\::\:\tau(b)=b\}$. We can actually identify $E_{c}$ with the
complexification of $E$. Real JB$^*$-triples were introduced by Isidro, Kaup and Rodríguez Palacios in 1995 \cite{IsKaRo95}. The class of real JB$^*$-triples includes all real and complex C$^*$-algebras, all JB- and JB$^*$-algebras, and all JB$^*$-triples when they are regarded as real Banach spaces. \smallskip

A real or complex JBW$^*$-triple is a JB$^*$-triple which is also a dual Banach space. The second dual of a real or complex JB$^*$-triple is a JBW$^*$-triple (see \cite{Di86b}, \cite{IsKaRo95}). Every real or complex JBW$^*$-triple admits a unique (isometric) predual and its product is separately weak$^*$-continuous (compare \cite{BarTi} and \cite{MarPe}).\smallskip

Real forms of Cartan factors are called \emph{real Cartan factors}. The classification of real Cartan factors is due to Kaup \cite[Corollary 4.4]{Ka97} and Loos \cite[pages 11.5-11.7]{Loos77}, and the classification can be resumed as follows: Let $X$ and $Y$ be two real Hilbert spaces, let $P$ and $Q$ be two Hilbert spaces over the quaternion field
$\mathbb{H}$, and finally, let $H$ be a complex Hilbert.

\begin{enumerate}
\begin{multicols}{2}

\item $I^{\RR} := \mathcal{L}(X,Y)$\vspace{0.45cm}

\item $I^{\HH} := \mathcal{L}(P,Q)$\vspace{0.45cm}

\item $II^{\CC}:=\{z\in
\mathcal{L}(H):z^*=z\}$\vspace{0.45cm}

\item $II^{\RR}:=\{x\in
\mathcal{L}(X):x^*=-x\}$\vspace{0.45cm}

\item $II^{\HH}:= \{w\in
\mathcal{L}(P):w^*=w\}$\vspace{0.45cm}

\item $III^{\RR} :=\{x\in \mathcal{L}(X): x^* =
x\}$\vspace{0.45cm}

\item $III^{\HH}\!:=\!\{w\in\!\mathcal{L}(P)\!:\! w^*\!
=\!-w\}$\vspace{0.45cm}
\end{multicols}

\item $IV^{r,s}:=E,$ where  $E=X_1 \oplus^{\ell_{_1}}
X_2$ and $X_1$,$X_2$ are closed linear subspaces, of dimensions
$r$ and $s$, of a real Hilbert space, $X$, of dimension greater or
equal to three, so that $X_2={X_1}^{\perp}$, with triple product
$$\J x y z = \left< x / y \right> z + \left< z / y \right> x -
\left< x / \bar z \right> \bar y,$$ where $\left< . / . \right>$
is the inner product in $X$ and the involution $x\to \bar x$ on
$E$ is defined by $\bar x = ( x_1,-x_2)$ for every $x=(x_1,x_2)$.
This factor is known as a \emph{real spin factor}.

\begin{multicols}{2}

\item $V^{\mathbb{O}_{0}}:=M_{1,2}(\mathbb{O}_{0})$\vspace{0.45cm}
\item $V^{\mathbb{O}}:= M_{1,2} (\mathbb{O})$\vspace{0.45cm} \item
$VI^{\mathbb{O}_{0}}:=H_3 (\mathbb{O}_{0})$\vspace{0.45cm} \item
$VI^{\mathbb{O}}:=H_3 (\mathbb{O})$\vspace{0.45cm}
\end{multicols}

\noindent where $\mathbb{O}_{0}$ is the real split Cayley algebra
over the field of the real numbers and $\mathbb{O}$ is the real
division Cayley algebra (known also as the algebra of real
division octonions). The real Cartan factors $(ix)-(xii)$ are
called \emph{exceptional real Cartan factors}.\end{enumerate}

By a \emph{generalized real Cartan factor} we shall mean a real Cartan factor or a complex Cartan factor regarded as a real JB$^*$-triple.\smallskip

We have already commented that the proof of Theorem \ref{thm local triple derivations are triple derivations} can be literary adapted to show that every continuous local derivation $T$ on a real JB$^*$-triple $E$ satisfies the identity in \eqref{eq new real local triple derivation on a single element}. Therefore if we consider the symmetrized triple product $$ \left< a,b,c\right>:=\frac13 \left(\J abc + \J cab + \J bca\right),$$ which is trilinear and symmetric, a real polarization formula of the form gives that $T$ is a triple derivation of the symmetrized Jordan triple product $ \left<.,.,.\right>.$

\begin{proposition}\label{p real local triple derivations are symmetrized triple derivations}\cite[Corollary 2.5]{BurFerPe2013} Every continuous local triple derivation on a real JB$^*$-triple is a triple derivation for the symmetrized triple product $ \left<a,b,c\right>:=1/3 \left(\J abc + \J cab + \J bca\right)$.
\end{proposition}

Even the reciprocal of Proposition \ref{p real local triple derivations are symmetrized triple derivations} is, at first look, a non trivial problem.

\begin{problem}\label{problem real linear derivations for the syymetrized product are local triple derivations on real JB*-triples} Is every triple derivation for the symmetrized triple product on a real JB$^*$-triple a local triple derivation?
\end{problem}

The rank, $r(E)$, of a real or complex JB$^*$-triple $E$, is the minimal cardinal number $r$
satisfying $card(S)\leq r$ whenever $S$ is an orthogonal subset of
$E$, i.e. $0\notin S$ and $x\perp y$ for every $x\neq y$ in $S$.\smallskip

Let $T: E\to E$ be a triple derivation for the symmetrized triple product on a rank one real JB$^*$-triple. It is known that every non-zero element $x$ in $E$ satisfies that $u:=\frac{x}{\|x\|}$ is a minimal tripotent (i.e. $E_1 (u) :=\{ x\in E : Q(u) (x) =x  \} = \RR u$). It is easily verified, via Peirce arithmetic and the identity $T(u) = 2 \{T(u),u,u\}+ \{u,T(u),u\}$, that the (inner) triple derivation $$\delta = \frac{1}{2 \|x\|} \delta(T(x)+3P_1(u)T(x),u),$$ %$$=\frac{1}{2 \ \|x\|} \Big(L(T(x)+3P_1(u)T(x),u)-L(u,T(x)+3P_1(u)T(x))\Big)$$
satisfies $$2 \delta(x) = \J {T(x)}uu + 3 \J { P_1 (u)T(x)}uu - \J u{T(x)}u - 3 \J u{P_1 (u)T(x)}u$$ $$= \|x\|\left(  \J {T(u)}uu + 3 \J { P_1 (u)T(u)}uu - \J u{T(u)}u \right)$$ $$= 2 \left( P_2 (u)T(x) + P_1 (u)T(x)\right) = 2T(x).$$

\begin{proposition}\label{p local derivations in rank-one cartan factors}
Let $E$ be a real JB$^*$-triple of rank one. Every (real linear) triple derivation of the symmetrized Jordan triple
product, $T : E \to E$, is a local triple derivation.$\hfill\Box$
\end{proposition}

The main goal in the setting of real JB$^*$-triples is the following:

\begin{problem}\label{problem local triple derivations on real JB*-triples}\cite[Problem 2.6]{BurFerPe2013} Is every bounded local triple derivation on a real JB$^*$-triple a triple derivation?
\end{problem}

We shall revisit here the answer to the above problem provided by Fernández Polo, Molino and the third author of this note in \cite{FerMolPe}. The first goal is to study this problem in the case of generalized Cartan factors.\smallskip

\begin{proposition}\label{p Cartan factors of rank bigger than one}\cite[Proposition 2.6]{FerMolPe} Let $C$ be a generalized real Cartan factor of rank $>1$ and let $T: C \to C$ be a linear map. The following are equivalent:\begin{enumerate}[$(a)$]
\item $T$ is a triple derivation;
\item $T$ is a local triple derivation;
\item $T$ is a triple derivation of the symmetrized triple product.
\end{enumerate}
\end{proposition}

\begin{proof} The implication $(b)\Rightarrow (c)$ follows from Proposition \ref{p real local triple derivations are symmetrized triple derivations}, while $(a)\Rightarrow (b)$ is clear. To prove $(c)\Rightarrow (a)$, let $T: C\to C$ be a triple derivation of the symmetrized triple product $\left<. \ ,.\ ,.\right>$. In this case $\{\exp (t T) : C \to C\}_{t\in \mathbb{R}}$ is a one-parameter group of automorphisms of the symmetrized triple product. By \cite[Theorem 4.8 ]{IsKaRo95}, $\exp (t T)$ is a surjective isometry for every real $t$. Since $C$ is of rank $>1$, Corollary 2.15 in \cite{FerMarPe} proves that $\exp(t T)$ is a triple automorphism of the original triple product and hence $$\exp(t T) \J xyz = \J {\exp(t T) (x)}{\exp(t T) (y)}{\exp(t T) (z)},$$ for every $x,y,z\in C$ and $t\in \mathbb{R}.$ Finally, the identity $$\frac{\partial}{\partial t}{\mid_{_{t=0}}} \left(\exp(t T) \J xyz \right)=\frac{\partial}{\partial t}{\mid_{_{t=0}}} \left( \J {\exp(t T) (x)}{\exp(t T) (y)}{\exp(t T) (z)}\right),$$ gives $T\J xyz = \J {T(x)}yz + \J x{T(y)}z + \J xy{T(z)}.$
\end{proof}

Combining Propositions \ref{p local derivations in rank-one cartan factors} and \ref{p Cartan factors of rank bigger than one} we easily deduce that local triple derivations and triple derivations of the symmetrized triple product on a generalized real Cartan factor define the same objects, this is a positive answer to Problem \ref{problem real linear derivations for the syymetrized product are local triple derivations on real JB*-triples} when $E$ is a generalized real Cartan factor.\smallskip

Local triple derivations on rank one generalized Cartan factor are the next objective. The next two results and the counterexample following them are taken from \cite{FerMolPe}.

\begin{lemma}\label{l realtypeI} Let $E$ be a rank one generalized real Cartan factor of type $I^{\mathbb{R}}$, and let  $T: E \rightarrow E$ be a real linear mapping. The following statements are equivalent:
   \begin{enumerate}[$(a)$] \item $T$ is a local triple derivation;
\item $T$ is a triple derivation for the symmetrized triple product;
\item $T$ is a bounded skew-symmetric operator {\rm(}i.e. $T^*=-T${\rm)};
\item $T$ is a triple derivation.
\end{enumerate}
\end{lemma}

\begin{proof} We recall that $E$ is a real Hilbert space and the triple product of $E$ is given by  $$\{x,y,z\}=1/2 (\langle x/y\rangle z+\langle z/y \rangle x),$$ where $\langle \cdot/\cdot \rangle$ denotes the inner product on $E$.\smallskip

The implication $(a)\Rightarrow (b)$ is proved in \cite[Corollary 2.5]{BurFerPe2013}. The equivalence $(c)\Leftrightarrow (d)$ was established in \cite[Lemma 3, Section 3.3]{HoMarPeRu}, while $(d)\Rightarrow (a)$ is obvious.\smallskip

We shall prove $(b)\Rightarrow (c)$. Let $T$ be a derivation for the symmetrized triple product $\left<\cdot,\cdot,\cdot \right>$. For each $x \in  E$, we have that  $T\{x,x,x\}=2\{Tx,x,x\}+\{x,Tx,x\}$ and hence $$\Vert x \Vert ^2 T(x) =\langle T(x)/x \rangle x+\Vert x \Vert ^2 T(x)+ \langle x/T(x)\rangle x,$$ which gives, $$\langle T(x)/x \rangle=-\langle x/T(x)\rangle, \hspace{0.6cm} \forall x \in E,$$ and hence $T^* = -T.$
\end{proof}

Local triple derivations can be also described on rank-one real spin factors.

\begin{lemma}\label{l caseofspinfactor} Let $E$ be a real spin factor of rank one and let $T: E \rightarrow E$ be a {\rm(}real{\rm)} linear mapping. The following statements are equivalent:
\begin{enumerate}[$(a)$] \item $T$ is a local triple derivation;
\item $T$ is a triple derivation for the symmetrized triple product;
\item $T$ is a bounded skew-symmetric operator {\rm(}$T^*=-T${\rm)};
\item $T$ is a triple derivation.
\end{enumerate}
\end{lemma}

\begin{proof}
We recall that $E$ is a real Hilbert space with inner product $\langle\cdot/\cdot \rangle$, whose triple product is defined by $\{x,y,z\}=\langle x/y \rangle z+\langle z/y \rangle x-\langle x/z\rangle y$ (compare \cite[Theorem 4.1 and Proposition  5.4]{Ka97}).\smallskip

The arguments given in the proof of Lemma \ref{l realtypeI} remain valid to prove the implications $(a)\Rightarrow (b)$, $(d)\Rightarrow (a)$, while $(c)\Leftrightarrow (d)$ was essentially obtained in \cite[Section 3.2]{HoMarPeRu}.\smallskip

$(b)\Rightarrow (c)$ If $T$ is a triple derivation for the symmetrized triple product $\left<\cdot,\cdot,\cdot \right>$, then $$T\{x,x,x \}=2 \{T(x),x,x \}+\{x,T(x),x\},$$ $$\Vert x \Vert ^2 T(x)=2\langle T(x)/x \rangle x+2 \Vert x \Vert ^2 T(x) -2 \langle T(x)/x \rangle x$$ $$ +2\langle x/T(x)\rangle x-\Vert x \Vert ^2 T(x),$$ and hence $ \langle x/T(x)\rangle =0,$ for all $x \in E$, which concludes the proof.\smallskip
\end{proof}

Unfortunately, the previous two lemmas do not cover all possible rank one generalized Cartan factors, there exists additional examples of of rank one real Cartan factors having an essentially ``complex'' or ``quaternionic'' nature, for which there exist local triple derivations which are not triple derivation.

\begin{example}\label{example local not derivation}\cite[Example 2.4]{FerMolPe} Let $H= \mathbb{C}^2$ be the 2-dimensional complex Hilbert space equipped with its natural inner product $\langle(\lambda_1,\lambda_2) | (\mu_1,\mu_2)\rangle = \lambda_1 \overline{\mu_1} + \lambda_2 \overline{\mu_2}$. We equip $H$ with its structure of (rank-one) complex Cartan factor of type $I^{\mathbb{C}}$ with product $$2 \{ \lambda, \mu, \nu \} = \langle\lambda| \mu \rangle \nu  + \langle \nu | \mu \rangle \lambda.$$

Let $T: H = \mathbb{R}^{4} \to H = \mathbb{R}^{4}$ be the real linear mapping given by $T(\lambda_1,\lambda_2)= \left( \Re\hbox{e} (\lambda_2), - \Re\hbox{e} (\lambda_1) \right).$ Clearly, $T$ is not $\mathbb{C}$-linear. It is not hard to check that $\Re\hbox{e} \langle T(\lambda_1,\lambda_2)| (\lambda_1,\lambda_2) \rangle=0,$ and thus,
                       $$2\{T(\lambda_1,\lambda_2), (\lambda_1,\lambda_2), (\lambda_1,\lambda_2) \} + \{(\lambda_1,\lambda_2), T(\lambda_1,\lambda_2), (\lambda_1,\lambda_2) \} $$ $$ =\langle T(\lambda_1,\lambda_2)| (\lambda_1,\lambda_2)\rangle (\lambda_1,\lambda_2) + (|\lambda_1|^2 +|\lambda_2|^2) T(\lambda_1,\lambda_2)$$ $$ + \langle (\lambda_1,\lambda_2)| T(\lambda_1,\lambda_2) \rangle (\lambda_1,\lambda_2) $$
$$ =  (|\lambda_1|^2 +|\lambda_2|^2) T(\lambda_1,\lambda_2) + 2 \Re\hbox{e} \langle T(\lambda_1,\lambda_2)| (\lambda_1,\lambda_2)\rangle (\lambda_1,\lambda_2)$$
$$ = (|\lambda_1|^2 +|\lambda_2|^2) T(\lambda_1,\lambda_2) = T\{(\lambda_1,\lambda_2), (\lambda_1,\lambda_2), (\lambda_1,\lambda_2) \},$$ which shows that $T\J xxx = 2 \J {T(x)}xx + \J x{T(x)}x$, for every $x\in H$. A priori, this is not enough to guarantee that $T$ is a local derivation. However, Proposition \ref{p local derivations in rank-one cartan factors} assures that $T$ is a local triple derivation.\smallskip

On the other hand, the identities  $$T\{(1,0),(i,0),(1,0) \}=(0,0)$$ and $$ 2\{T(1,0),(i,0),(1,0)\} + \{(1,0),T(i,0),(1,0) \} = (0,i),$$ tell that $T$ is not a triple derivation.
\end{example}

It is shown in \cite[Proposition 2.5]{FerMolPe} that every real linear triple derivation on a complex JB$^*$-triple must be $\mathbb{C}$-linear

\begin{proposition}\label{p derivations are complex linear}
Let $E$ be a complex JB$^*$-triple. Every real linear triple derivation $\delta : E \to E$ is complex linear.
\end{proposition}

\begin{proof} Having in mind that, for each $a,b$ in $E$, $L(a,b): E \to E$ is
$\CC-$linear, every (real linear) inner derivation on $E$ is $\CC-$linear.\smallskip

Suppose now that $\delta: E\to E$ is a real linear derivation.  Since every real JB$^*$-triple $E$ satisfies the Inner Approximation Property defined in \cite[Theorem 4.6]{BarFri} and \cite[Theorem 5]{HoMarPeRu}) (that is, the space of all inner triple derivations on $E$ is dense in the space of all triple derivations on $E$, with respect to the strong operator topology of $B(E)$), given $\varepsilon > 0$ and $x\in E$, there exists a inner derivation $\displaystyle \widehat{\delta} =  \sum_{j=1}^{n} \delta(a_{j},b_{j}) = \sum_{j=1}^{n} L(a_{j},b_{j})- L(b_j,a_j) $ such that $\Vert \delta(x)-\widehat{\delta}(x) \Vert< \frac{\varepsilon}{2} $ and $\Vert \delta(ix)-\widehat{\delta}(ix) \Vert<
   \frac{\varepsilon}{2} $. Therefore $$\Vert i\delta(x)-\delta(ix) \Vert \leq \Vert i\delta(x)-i\widehat{\delta}(x) \Vert + \Vert \widehat{\delta}(ix)-\delta(ix) \Vert < \varepsilon.$$ The arbitrariness of $\varepsilon$ and $x$ guarantee the desired statement.
\end{proof}

We recall that a subspace $I$ of a real JB$^*$-triple $E$ is a
\emph{triple ideal} %(respectively, an \emph{inner ideal})
if $\{E,E,I\}+\{E,I,E\} \subseteq I$. %(respectively, $\{I,E,I\} \subseteq I$).
It is known that a subtriple $I$ of $E$ is a triple ideal if and only if $\J EEI \subseteq I$ or $\J EIE \subseteq I$ or $\J EII\subseteq I$ (compare \cite{Bun86}). It is very easy to see that every local triple derivation $T$ on $E$ satisfies that for each norm closed ideal $I\subseteq E$, we have  $T(I)\subseteq I$.\smallskip

We can establish sufficient conditions on a real JB$^*$-triple $E$ to assure that every continuous local triple derivation on $E$ is a derivation.

\begin{theorem}\label{t sufficient conds on general real JB*-triples}\cite[Theorem 3.4]{FerMolPe} Let $E$ be a real JB$^*$-triple whose second dual contains no rank-one generalized real Cartan factors of types $I^{\mathbb{C}}$, $I^{\mathbb{H}}$ and $V^{\mathbb{O}}:= M_{1,2} (\mathbb{O})$. Then every continuous local triple derivation on $E$ is a triple derivation.
\end{theorem}

\begin{proof}
Let $T : E \to E$ be a continuous local triple derivation on $E$. Proposition \ref{p real local triple derivations are symmetrized triple derivations} implies that $T$ is a triple derivation of the symmetrized triple product. The separate weak$^*$-continuity of the triple product of $E^{**}$ together with the weak$^*$-continuity of $T^{**}$, and the weak$^*$-density of $E$ inn $E^{**}$ can be applied to show that $T^{**} : E^{**} \to E^{**}$ is a triple derivation of the symmetrized triple product.\smallskip

The atomic decomposition established in \cite[Theorem 3.6]{PeSta} assures that $E^{**}$ decomposes as an orthogonal sum
$$E = A \oplus^{\infty} N,$$ where $A$ and $N$ are weak$^*$-closed triple ideals of $E^{**}$, $A$ being the weak$^*$-closed real linear
span of all minimal tripotents in $E^{**}$, $N$ containing no \hyphenation{mi-nimal} minimal tripotents and $A \perp N$.
It is also proved in \cite[Theorem 3.6]{PeSta} that $A$ is an orthogonal sum of weak$^*$-closed triple ideals which are generalized real Cartan factors. That is,  there exists a family of mutually orthogonal, weak$^*$-closed triple ideals $\{C_i: i\in \Lambda\}\cup \{N\}$ of $E^{**}$ such that $A= \bigoplus_{i}^{\ell_{\infty}} C_i$ and $$E^{**} =\left(\bigoplus_{i}^{\ell_{\infty}} C_i\right) \bigoplus^{\ell_{\infty}} N.$$ The comments before this theorem assure that $T^{**} (N) \subseteq N$ and $T^{**} (C_i) \subseteq C_i$, for every $i\in \Lambda$.\smallskip

Let us remark that every real JB$^*$-triple of rank one is precisely one of the following: a rank-one type $I^{\mathbb{R}}$, $I^{\mathbb{C}}$, $I^{\mathbb{H}}$, a rank-one real spin factor $IV^{n,0}$, and $V^{\mathbb{O}}:= M_{1,2} (\mathbb{O})$ (cf. \cite[Proposition 5.4]{Ka97}). So, by hypothesis, each $C_i$ is a generalized real Cartan factor of rank $>2$ or a rank-one generalized real Cartan factor of type $I^{\mathbb{R}}$ (i.e. $B(H,\mathbb{R})$, for a real Hilbert space $H$), or a real spin factor of rank one. Now, Lemmas \ref{l realtypeI} and \ref{l caseofspinfactor} and Proposition \ref{p Cartan factors of rank bigger than one} imply that $T^{**}|_{C_i} : C_i \to C_i$ is a triple derivation for every $i$, and hence $T^{**}|_{\bigoplus_{i}^{\ell_{\infty}} C_i} : \bigoplus_{i}^{\ell_{\infty}} C_i \to  \bigoplus_{i}^{\ell_{\infty}} C_i$ is a triple derivation too.\smallskip

We further known that if $j : E \hookrightarrow E^{**}$ denotes the canonical embedding, and $\pi: E^{**} \to A$ the canonical projection of $E^{**}$ onto $A$, then the mapping $\pi \circ j : E \to A$ is an isometric triple embedding (cf. \cite[Proposition 3.1]{FerMarPe}). Since $T^{**}|_{A} : A \to A$ is a triple derivation, and $\pi \circ j  T = T^{**} \pi \circ j $, we have $$\Phi T (\J xyz) =  T^{**}|_{A} \J {\Phi(x)}{\Phi(y)}{\Phi(z)} $$ $$= \J {T^{**}|_{A}\Phi(x)}{\Phi(y)}{\Phi(z)} + \J {\Phi(x)}{T^{**}|_{A}\Phi(y)}{\Phi(z)} $$
$$+ \J {\Phi(x)}{\Phi(y)}{T^{**}|_{A}\Phi(z)} = \Phi \left(\J {T(x)}{y}{z} +  \J {x}{T(y)}{z} + \J {x}{y}{T(z)}\right),$$ which proves that $T$ is a triple derivation.
\end{proof}

\section{Some comments on automatic continuity}\label{sec:automatic cont}

There are some interesting results guaranteeing the continuity of derivations and local derivations. Sakai, solving a conjecture posed by Kaplansky, proves in 1960 that any derivation of a C$^*$-algebra is automatically continuous \cite{Sak60}. Ringrose establishes in \cite{Ringrose72} that every derivation from a C$^*$-algebra $A$ into a Banach $A$-bimodule is continuous. A remarkable extension of the above result is due to Johnson, who authored the following theorem:

\begin{theorem}\label{t automatic continuity of local derivations}\cite[Theorem 7.5]{John01} Let $A$ be a C$^*$-algebra and $X$ a Banach $A$-bimodule. If $T$ is a local derivation, not assumed a priori to be continuous, from $A$ into $ X$, then $T$ is continuous.$\hfill\Box$
\end{theorem}

Combining Theorem \ref{thm Johnson local 2001} with the above theorem we get:

\begin{theorem}\cite{John01} Every local derivation of a C$^*$-algebra $A$ into a Banach $A$-bimodule is a derivation.$\hfill\Box$
\end{theorem}

In the setting of JB$^*$-triples, Barton and Friedman prove in \cite{BarFri} that every triple derivation on a JB$^*$-triple is automatically continuous. Their proof actually shows that  every derivation on a JB$^*$-triple $E$ is dissipative and hence continuous. We recall that a linear mapping $T$ on a Banach space $X$ is called \emph{dissipative} if for each $x\in X$ and each functional $\phi\in X^*$ with $\|x\|= \|\phi\| = \phi (x) =1$ we have $\Re\hbox{e} \phi (T(x)) \leq 0$. It is known that $T$ is continuous whenever it is dissipative (compare \cite[Proposition 3.1.15]{BraRo}). Suppose $T: E \to E$ is a local triple derivation on a (real or complex) JB$^*$-triple, and let us pick an element $x$ in $E$ and a functional $\phi$ in $E^*$ with $\|x\|= \|\phi\| = \phi (x) =1$. By assumptions, there exists a triple derivation $\delta_x : E \to E$ satisfying $T(x) = \delta_x (x).$ Since $\delta_x$ is dissipative, we have $\Re\hbox{e} \phi T(x) = \Re\hbox{e} \phi \delta_x (x) \leq 0$, which asserts that $T$ is dissipative.

\begin{theorem}\label{t automatic continuity}\cite[Theorem 2.8]{BurFerPe2013} Every local triple derivation on a (real or complex) JB$^*$-triple is continuous.$\hfill\Box$
\end{theorem}

As a consequence of Theorems \ref{thm local triple derivations are triple derivations}, \ref{t sufficient conds on general real JB*-triples} and \ref{t automatic continuity} we get:

\begin{theorem}\label{t local triple derivations on real or complex JB-triples are derivations} Every local triple derivation on a JB$^*$-triple is a triple derivation. Every local triple derivation on a real JB$^*$-triple whose second dual contains no rank-one generalized real Cartan factors of types $I^{\mathbb{C}}$, $I^{\mathbb{H}}$ and $V^{\mathbb{O}}:= M_{1,2} (\mathbb{O})$ is a triple derivation. $\hfill\Box$
\end{theorem}

\section{2-local derivations on von Neumann algebras}\label{s:2-local der}

In this section  we study  2-local derivations on von Neumann algebras. In 1997, \v{S}emrl \cite{Semrl97} introduced the concepts of $2$-local derivations and $2$-local automorphisms.

%\begin{definition}
Let ${A}$ be an algebra. A (a non-necessarily linear nor continuous) mapping $\Delta:{A}\rightarrow {A}$  is called a \textit{2-local derivation} if for every $x, y\in {A},$ there is a derivation $D_{x, y}: {A}\rightarrow {A},$ depending on $x$ and $y,$ such that  $\Delta(x)=D_{x, y}(x)$  and $\Delta(y)=D_{x, y}(y).$
%\end{definition}

\v{S}emrl describes $2$-local derivations on the algebra $B(H)$ of all bounded linear operators on an infinite-dimensional separable Hilbert space $H.$

\begin{theorem} \cite{Semrl97}.
Let $H$ be an infinite-dimensional separable Hilbert space, and let $B(H)$ be the algebra of all linear bounded operators on $H.$
Then every 2-local derivation $T : B(H) \rightarrow B(H)$ (no linearity or continuity of $\theta$ is assumed) is a derivation.
\end{theorem}

In \cite[Remark]{Semrl97},  \v{S}emrl states that the conclusion of the above theorem also holds when $H$ is finite-dimensional. In such a case, however, he was only able to get a long proof involving tedious computations, and so, he decided not include these results. In \cite{KimKim04} Kim and Kim gave a short proof of the fact that every 2-local derivation on a finite-dimensional complex matrix algebra is a derivation.

\begin{theorem}\cite{KimKim04} Let $M_n$ be the $n \times  n$-matrix algebra over $\mathbb{C},$
and let $T : M_n \rightarrow  M_n$ be a 2-local derivation. Then $T$ is a derivation.
\end{theorem}

The methods of the proofs of the aforementioned results from~\cite{KimKim04} and \cite{Semrl97} are essentially based on the fact that, for a separable or finite dimensional Hilbert space $H,$ the algebra $B(H)$ can be generated by two elements. For example \v{S}emrl \cite{Semrl97} makes use of the following two operators:
$$
u=\sum\limits_{n=1}^\infty \frac{1}{2^n}e_{n,n},\,
v=\sum\limits_{n=2}^\infty e_{n-1,n},
$$
where $\{e_{m,n}\}_{m,n=1}^\infty$ is a set of matrix units in
$B(H).$ Kim and Kim consider, in \cite{KimKim04}, the
following two matrices:
\begin{center}
\[u=\left( \begin{array}{cccc}
\frac{\textstyle 1}{\textstyle 2}& 0 &\ldots & 0\\
0 & \frac{\textstyle 1}{\textstyle 2^2} &\ldots & 0\\
\vdots& \vdots &\ddots & \vdots\\
0 & 0 &\ldots & \frac{\textstyle 1}{\textstyle 2^n}
\end{array} \right),
\,\, v=\left( \begin{array}{ccccc}
0 & 1 & 0 & \ldots & 0 \\
0 & 0 &  1 & \ldots & 0\\
\vdots& \vdots& \vdots &\vdots & \vdots\\
0 & 0 &\ldots & 0 & 1\\
0 & 0 &\ldots & 0 &0
\end{array} \right).\]
\end{center}
Further,  for  $u,\, v$  as above one can choose a derivation $D_{u, v}:B(H)\rightarrow B(H)$
such that $\Delta(u)=D_{u, v}(u)$  and $\Delta(v)=D_{u, v}(v)$ and show that $\Delta(x)=D_{u,v}(x)$ for all $x\in B(H).$ This means that any 2-local derivation on $B(H)$ for a separable or finite dimensional Hilbert space $H,$ is completely determined by its values on the elements $u$ and $v.$\smallskip

Some years later, Zhang and Li \cite{Zhang} extend the previously mentioned result of Kim and Kim for arbitrary symmetric digraph matrix algebras and constructed  an example of a $2$-local derivation which is not a derivation on the algebra of all upper triangular complex $2\times 2$-matrices.

\begin{example} {\rm(}see \cite{Zhang}{\rm)}.
Let us consider the algebra of all upper-triangular complex $2\times 2$-matrices
$$ {A}=\left\{x=\left(%
\begin{array}{cc}
  \lambda_{11} & \lambda_{12} \\
  0 & \lambda_{22} \\
\end{array}%
\right): \lambda_{ij}\in \mathbb{C}\right\}.
$$ Define an operator $\Delta$ on ${A}$ by
$$
\Delta(x) = \left\{
\begin{array}{ll}
0, & \textrm{if } \lambda_{11}\neq \lambda_{22}\textrm{,}\\
\ & \ \\
 \left(%
\begin{array}{cc}
  0 & 2\lambda_{12} \\
  0 & 0 \\
\end{array}%
\right), & \textrm{if } \lambda_{11}= \lambda_{22}.
\end{array} \right.
$$
Then $\Delta$ is a $2$-local derivation, which is not a
derivation.
\end{example}

As it was mentioned above the proofs of papers \cite{KimKim04} and \cite{Semrl97} are essentially based on the fact that, in the case of $H$ being separable or finite dimensional, the algebra $B(H)$ is generated by two elements. Since the algebra $B(H)$ is not generated by two elements (like in the case of a non-separable Hilbert space), one cannot directly apply the methods of the above papers. In \cite{AyuKuday2012} the first and second authors of this note suggested a new technique, which allows to generalize the above mentioned results of \cite{KimKim04} and \cite{Semrl97} to arbitrary Hilbert spaces.  Namely, the new method can be applied to prove that every $2$-local derivation on $B(H)$, with $H$ an arbitrary Hilbert space (no separability is assumed), is a derivation. A similar result for $2$-local derivations on finite von Neumann algebras was obtained by Nurjanov, Alauatdinov and the first two authors of this note in \cite{AKNA}. In \cite{AA2} the authors extend all the above results and give a short proof of this result for arbitrary semi-finite von Neumann algebras. Finally a solution of this problem for general von Neumann algebras was recently obtained in  \cite{AyuKuday2014}. We shall revisit here the proof, given in  \cite{AyuKuday2014}, of the fact that every 2-local derivation on an arbitrary von Neumann algebra  is a derivation.\smallskip

Let $M$ be a von Neumann algebra on a complex Hilbert space $H.$ We recall that any derivation $D$ on $M$ is an inner derivation, that is, there exists an element $a\in M$ such that $$D(x)=[a, x]= ax-xa,$$ for all $x\in M$ (cf.  \cite{Sak}). Therefore, for a von Neumann algebra $M$ the
above definition of $2$-local derivation is equivalent to the following one: A map $\Delta : M\rightarrow M$ is a $2$-local derivation, if for any two elements $x$, $y\in M$ there exists an element $a_{x,y}\in M$ such that
$$\Delta(x)=[a_{x,y}, x],\hbox{ and }\, \Delta (y)=[a_{x,y}, y].$$
If  $\Delta :M\rightarrow M$ is  a $2$-local derivation, then it easily follows from the definition that $\Delta$ is homogenous. At
the same time, \begin{equation}\label{joor}  \Delta(x^2)=\Delta(x)x+x\Delta(x),
\end{equation} for each $x\in M.$\smallskip

In \cite{Bre}, Bre\v{s}ar proves that any Jordan derivation (i.e. a linear map satisfying the above equation \eqref{joor}) on a semi-prime algebra is a derivation. Since every von Neumann algebra is semi-prime, to prove that a $2$-local derivation $\Delta :M\rightarrow M$ is a derivation, it suffices to prove that $\Delta:M\rightarrow M$ is additive.\smallskip

The following theorem is the main result of the section.

\begin{theorem}\label{main}
Let $M$ be an arbitrary von Neumann algebra. Then any $2$-local
derivation $\Delta: M\rightarrow M$ is a derivation.
\end{theorem}

The proof will follow from a series of partial results.\smallskip

In the hypothesis of the above Theorem \ref{main}. Let $e$ be a central projection in $M$. Since $D(e)=0$\label{ref comments on central projections} for every derivation $D$ on $M$, it is clear that $\Delta(e)=0$ for any $2$-local derivation $\Delta$ on $M$. Take $x\in M$ and let $D$ be a derivation on $M$ such that $\Delta(ex)=D(ex), \Delta(x)=D(x)$.
Then we have $\Delta(ex)=D(ex)=D(e)x+eD(x)=e\Delta(x)$. This means that, for each central projection $e\in M$, every $2$-local derivation $\Delta$ maps $eM$ into $eM$. Thus, if necessary, we may consider the restriction of $\Delta$ onto $eM$. Since an arbitrary von Neumann algebra can be decomposed along a central projection into the direct sum of a semi-finite and a purely infinite (type $III$) von Neumann algebra, we may consider these cases independently.

\subsection{Semi-finite von Neumann algebras}

In this subsection we present a short proof of a result in \cite{AA2} on the description of 2-local derivations for arbitrary semi-finite von
Neumann algebras.\smallskip

Let $M$ be a semi-finite von Neumann algebra and let $\tau$ be a faithful normal semi-finite trace on $M.$ Denote by $M_\tau$ the definition ideal of $\tau,$ i.e. the set of all elements $x \in M$ such that $\tau(|x|) < +\infty.$ Then $M_\tau$ is an $^*$-algebra and, moreover $M_\tau$ is a two sided ideal of $M.$\smallskip

It is clear that any derivation $D$ on $M$ maps the ideal $M_\tau$ into itself. Indeed, since $D$ is inner, i.e. $D(x) = [a, x]$
$(x\in M)$ for an appropriate $a \in  M,$ we have that $D(x) = ax - xa \in M_\tau$ for all $x \in  M_\tau.$ Therefore any $2$-local derivation on $M$ also maps $M_\tau$ into itself.

\begin{theorem}\label{semi-finite}\cite[Theorem 2.1]{AA2}.
Let $M$ be a semi-finite von Neumann algebra, and let
$\Delta:M\rightarrow M$ be a $2$-local derivation. Then $\Delta$
is a derivation.
\end{theorem}

\begin{proof} Let $\Delta : M \rightarrow  M$ be a $2$-local derivation and let $\tau$  be a faithful
normal semi-finite trace on $M.$ For each $x \in M$ and $y \in
M_\tau$ there exists an element  $a_{x,y}$ in $M$ such that
$$
\Delta(x) = [a_{x,y},x],\, \Delta(y) = [a_{x,y}, y].
$$
Then
$$
\Delta(x)y + x\Delta(y)= [a_{x,y}, x]y + x [a_{x,y}, y]=[a_{x,y},
xy],
$$
i.e.
$$
[a_{x,y}, xy] = \Delta(x)y + x\Delta(y).
$$
Since $M_\tau$ is an ideal and $y \in M_\tau,$  the elements $a_{x,y}xy, xy, xya_{x,y}$ and $\Delta(y)$ also belong to $M_\tau$, and hence, we have
$$
\tau(a_{x,y}xy) = \tau(a_{x,y}(xy)) = \tau((xy)a_{x,y}) =
\tau(xya_{x,y}).
$$
Therefore,
$$
0 = \tau(a_{x,y}xy - xya_{x,y}) = \tau([a_{x,y}, xy]) =
\tau(\Delta(x)y + x\Delta(y)),
$$
i.e.
$$
\tau(\Delta(x)y) = - \tau(x\Delta(y)).
$$
For arbitrary $u, v \in M$ and $w \in M_\tau$  set $x = u + v, y =
w.$ Then $\Delta(w) \in M_\tau$ and
\begin{eqnarray*}
\tau(\Delta(u + v)w)   & = &  -\tau((u +
v)\Delta(w))=-\tau(u\Delta(w))- \tau(v\Delta(w))= \\
& = & \tau(\Delta(u)w) + \tau(\Delta(v)w) =\tau((\Delta(u) +
\Delta(v))w),
\end{eqnarray*}
and so
$$
\tau((\Delta(u + v)-\Delta(u) - \Delta(v))w) = 0
$$
for all $u, v \in M$ and $w \in M_\tau.$ Denote $b = \Delta(u + v)
- \Delta(u) -\Delta(v).$ Then
\begin{equation}\label{trr}
\tau(bw)=0
\end{equation}
for every $w\in M_\tau.$ Now take a monotone increasing net $\{e_\alpha\}$ of projections in $M_\tau$ such that $e_\alpha\uparrow \mathbf{1}$ in $M.$ Since  $\{e_\alpha b^\ast\}\subset M_\tau,$ it follows that from \eqref{trr} that $\tau(be_\alpha b^\ast)=0,$ for all $\alpha.$ At the same time $be_\alpha b^\ast\uparrow bb^\ast$ in $M.$ Since the trace $\tau$ is normal we have $ \tau(be_\alpha b^\ast)\uparrow \tau(bb^\ast), $ i.e.
$\tau(bb^\ast)=0.$ The trace $\tau$ is faithful, thus, this implies that $bb^\ast=0,$ i.e. $b = 0.$ Therefore $\Delta(u + v) =
\Delta(u) + \Delta(v)$ for all $u, v \in M,$ i.e. $\Delta$ is an additive map on $M.$ As it was mentioned above, this proves that $\Delta$ is a derivation on $M.$
\end{proof}

\subsection{Purely infinite von Neumann algebras} In this subsection we present\hyphenation{pre-sent} the proof of Theorem \ref{main} for the case in which $M$  is a purely infinite von Neumann algebra. This result will complete the proof for arbitrary von Neumann algebras. Our proof is essentially based on the Bunce-Wright-Mackey-Gleason theorem for signed measures on projections of a von Neumann algebra established in \cite{BuWri94} (see \cite{AyuKuday2014}).\smallskip

The proof will require several lemmata. For a self-adjoint subset $S\subseteq M$ denote by $S'$ is the
commutant of $S,$ i.e. $$ S'=\{y\in B(H): xy=yx, \forall\, x\in S\}. $$

The first step of our proof is to show that any 2-local derivation on an arbitrary von Neumann algebra is additive on an abelian von Neumann
subalgebra generated by a self-adjoint element.

\begin{lemma}\label{masas}\cite[Lemma 2.2]{AyuKuday2014} Let $g\in M$ be a self-adjoint element and let $\mathcal{W}^\ast(g)=\{g\}''$ be the abelian von Neumann subalgebra
generated by the element $g.$ Then there exists an element $a\in M$ such that
\[ \Delta(x)=ax-xa,\] for all $x\in \mathcal{W}^\ast(g).$  In particular, $\Delta$ is
additive on $\mathcal{W}^\ast(g).$
\end{lemma}

\begin{proof}  By the definition there exists an element $a\in M$
(depending on $g$) such that
$$ \Delta(g)=ag-ga. $$
Let us show that $\Delta(x)=[a, x]$ for all $x\in \mathcal{W}^\ast(g).$ Let $x\in \mathcal{W}^\ast(g)$ be an
arbitrary element. By hypothesis, there exists an element $b\in M$ such that
$$ \Delta(g)=[b, g],\,  \Delta(x)=[b, x]. $$
Since
$$[a, g]=\Delta(g)=[b, g], $$
we get
$$ (b-a)g=g(b-a). $$
Thus,
$$ b-a\in \{g\}'=\{g\}'''=\mathcal{W}^\ast(g)',$$
i.e. $b-a$ commutes with any element from $\mathcal{W}^\ast(g).$ Therefore
$$ \Delta(x)=[b, x]=[b-a, x]+[a, x]=[a, x] ,$$ for all $x\in \mathcal{W}^\ast(g),$ and the proof is complete.
\end{proof}

Given a von Neumann algebra $M$, we shall denote by  $\mathcal{P}(M)$ the lattice of all projections in $M.$ Let $X$ be a Banach space. A mapping $\mu: \mathcal{P} (M)\to X$ is said to be \emph{finitely additive} when
$$\mu \left(\sum\limits_{i=1}^n p_i\right) = \sum\limits_{i=1}^{n} \mu (p_i), $$
for every family $p_1,\ldots, p_n$ of mutually orthogonal projections in $M.$ If the set $$ \left\{ \|\mu (p)\|: p \in \mathcal{P} (M) \right\} $$ is bounded, we shall say that $\mu: \mathcal{P} (M)\to X$ is  \emph{bounded}.\smallskip

The Bunce-Wright-Mackey-Gleason theorem (\cite{BuWri92, BuWri94}) states that if $M$ has no summand of type $I_2$, then every bounded finitely additive mapping $\mu: \mathcal{P} (M)\to X$ extends to a bounded linear operator from $M$ to $X$.\smallskip

We recall that every family $(p_j)$ of mutually orthogonal projections in a von Neumann algebra $M$ is summable with respect to the weak$^*$-topology of $M$, and the sum $\displaystyle p= \hbox{w$^*$-}\sum_{j} p_j$ is another projection in $M$ (cf. \cite[Page 30]{Sak}). It is further known that  $(p_j)$ is summable with respect to the strong$^*$-topology of $M$ with the same limit, i.e., $\displaystyle p= \hbox{w$^*$-}\sum_{j} p_j = \hbox{strong$^*$-}\sum_{j} p_j.$ We shall simply write $\displaystyle p= \sum_{j} p_j$.\smallskip

Suppose that in the above paragraphs, $X$ is another von Neumann algebra $W$. Let $\tau$ denote the weak$^*$-, the strong, or the strong$^*$-topology of $W$. Then a mapping $\mu: \mathcal{P} (M)\to W$ is said to be \emph{$\tau$-completely additive} (respectively, \emph{$\tau$-countably
or $\tau$-sequentially additive}) when
\begin{equation}\label{eq completely additive}
\mu\left(\sum\limits_{i\in I} e_i\right) =\hbox{$\tau$-}\sum\limits_{i\in I}\mu(e_i)
\end{equation} for every family (respectively, sequence)  $\{e_i\}_{i\in I}$ of mutually orthogonal
projections in $M.$ In the lexicon of \cite{Shers2008} and \cite{Doro}, a completely
additive mapping $\mu : \mathcal{P} (M)\to \mathbb{C}$ is called a \emph{charge}.
The  Dorofeev--Sherstnev theorem (\cite[Theorem 29.5]{Shers2008} or \cite[Theorem 2]{Doro})  states that any charge on a  von Neumann algebra with no summands of type $I_n$ is bounded, and hence it extends to a bounded functional on $M$.\smallskip

The following result is the main step in the  proof of the automatic additivity of 2-local derivations on von Neumann algebras, and it plays a crucial role in the proof of Theorem~\ref{main}.\smallskip

\begin{lemma}\label{comad}\cite[Lemma 2.3]{AyuKuday2014} Let $M$ be a von Neumann algebra on a complex Hilbert space $H$, and let $\Delta: M\to M$ be a 2-local derivation. Then the restriction $\Delta|_{\mathcal{P}(M)} : \mathcal{P}(M)\to M $ is strong-completely additive, i.e.
\begin{equation}\label{ca}
\Delta\left(\sum\limits_{i\in I}e_i\right) =\hbox{strong-}
\sum\limits_{i\in I}\Delta(e_i),
\end{equation} for every family $\{e_i\}_{i\in I}$ of mutually orthogonal projections in $M$.
\end{lemma}

\begin{proof} We may assume that $M$ is a von Neumann subalgebra of some $B(H)$. Let us recall that the strong topology of the von Neumann algebra $B(H)$ is the topology determined by the semi-norms $a\mapsto \|a\|_{\xi} :=\|a(\xi)\|,$ where $\xi$ runs in $H$ (cf. \cite[\S 1.15]{Sak}). By \cite[Proposition 1.24.5]{Sak}, every normal functional $\varphi\in M_*$, can be extended to a normal functional $\widetilde{\varphi}$ in $B(H)_*$ with $\|\varphi\|= \|\widetilde{\varphi}\|$. Consequently, the strong topology of $M$ coincides with the restriction to $M$ of the strong-topology of $B(H)$ (see also  \cite{Bun01}).\smallskip

We shall consider the following two cases:\smallskip

\emph{Case 1.} We assume that $\{e_n\}_{n\in \mathbb{N}}$ is a sequence of mutually orthogonal projections in $M.$ Put $g=\sum\limits_{n\in I}\frac{\textstyle 1}{\textstyle n}e_n\in M.$ By Lemma~\ref{masas} there exists an element $a\in M$ such that $\Delta(x)=ax-xa$ for all $x\in \mathcal{W}^\ast(g).$ Since $e_n\in
\mathcal{W}^\ast(g),$ for all $n\in \mathbb{N},$ by the joint strong$^*$-continuity of the product of $M$, we obtain that $$\Delta\left(\sum\limits_{n\in \mathbb{N}}e_n\right)=\left[a, \sum\limits_{n\in \mathbb{N}}e_n\right]=\hbox{strong$^*$-}\sum\limits_{n\in \mathbb{N}}[a, e_n]=\hbox{strong$^*$-}\sum\limits_{n\in \mathbb{N}}  \Delta(e_n),$$
i.e.  $\Delta$ is a strong$^*$--sequentially additive map.\smallskip

\emph{Case 2.} Let $\xi\in H$ be a fixed point and let $\{e_i\}_{i\in I}$ be an arbitrary family of
orthogonal projections in $M.$ For each natural $n\in \mathbb{N}$ we set
$$ I_n=\left\{i\in I: \|\Delta(e_i)(\xi)\|_H\geq 1/n\right\}. $$
Suppose that there exists $k\in \mathbb{N}$ such that $I_k$ is infinite.
If necessary, passing to subset we can assume that
$I_k$ is countable. Then the
series
$$
\sum\limits_{i\in I_k} \Delta(e_i)(\xi)
$$
does not converge in $H.$ On the other hand, since $I_k$ is a countable set, by \emph{Case 1}, we have
$$
\Delta\left(\sum\limits_{i\in
I_k}e_i\right)(\xi)=\sum\limits_{i\in I_k}  \Delta(e_i)(\xi),
$$ which is impossible. Therefore, $I_k$ is a finite set for all $k\in \mathbb{N}.$ So, the set
$$
I_0=\left\{i\in I: \Delta(e_i)(\xi)\neq 0\right\}=\bigcup\limits_{n\in \mathbb{N}}I_n
$$
is a countable set.\smallskip

Let us consider the projection $e=\sum\limits_{i\in I\setminus I_0} e_i.$ We claim that $\Delta(e)(\xi)=0.$ Indeed, for every $i\in I\setminus I_0$ take an
element $a_i\in M$ such that
$$ \Delta(e)=a_i e-ea_i,\, \Delta(e_i)=a_i e_i-e_ia_i. $$
Since $\Delta(e_i)(\xi)=0$ we get $a_i(e_i(\xi))=e_i(a_i(\xi))$ for all $i\in I\setminus I_0.$\smallskip

We further have:
$$ (\Delta(e)e_i)(\xi)=(a_ie-ea_i)(e_i(\xi))=a_i(e(e_i(\xi)))-e(a_i(e_i(\xi)))=$$
$$ =a_i(e_i(\xi))-e(e_i(a_i(\xi)))=a_i(e_i(\xi))-e_i(a_i(\xi))=0, $$
i.e.
$$ \Delta(e)e_i(\xi)=0,$$ for all $i\in I\setminus I_0.$ Thus,
$$ \Delta(e)e(\xi)=\Delta(e)\sum\limits_{i\in I\setminus I_0}e_i(\xi)= \sum\limits_{i\in I\setminus I_0}\Delta(e)e_i(\xi)=0, $$
i.e. $$ \Delta(e)e(\xi)=0.$$ We similarly show
$$ e\Delta(e)(\xi)=0. $$

Now, since
$$ \Delta(e)=[a_i, e]=[a_i,e]e+e[a_i,e]=\Delta(e)e+e\Delta(e) $$
we obtain that
$$ \Delta(e)(\xi)=0$$ as we claimed.\smallskip

Finally,
$$ \Delta\left(\sum\limits_{i\in I}e_i\right)(\xi)= \Delta\left(e+\sum\limits_{i\in I_0}e_i\right)(\xi) =\hbox{[\emph{Case 1}]}=\Delta(e)(\xi)+\Delta\left(\sum\limits_{i\in I_0}e_i\right)(\xi)$$
$$= \Delta\left(\sum\limits_{i\in I_0}e_i\right)(\xi)=\hbox{[\emph{Case 1}]}=\sum\limits_{i\in I_0}\Delta\left(e_i\right)(\xi)=\left(
\sum\limits_{i\in I}\Delta\left(e_i\right)\right)(\xi).
$$

\end{proof}

In the following two lemmata we suppose that  $M$ is an \textit{infinite} von Neumann algebra, and $\Delta: M\to M$ is a 2-local derivation.

\begin{lemma}\label{addi}\cite[Lemma 2.4]{AyuKuday2014}
The restriction $\Delta|_{M_{sa}} : M_{sa} \to M$ to the self-adjoint part of $M$ is additive.
\end{lemma}

\begin{proof} We regard $M$ as a von Neumann subalgebra of some $B(H)$. Given $\xi,$ and  $\eta$ in $H$ we define a linear functional
$f_{\xi,\eta} = \xi\otimes \eta$ on $M$ given by
$$ f_{\xi,\eta}(x) = \xi\otimes\eta (x) =\langle x(\xi),\eta\rangle,\, x\in M, $$
where $\langle\cdot,\cdot\rangle$ is the inner product on $H.$\smallskip

Lemma~\ref{comad} implies that  the restriction $f_{\xi,\eta}\circ \Delta|_{P(M)}$ of the superposition
$$f_{\xi,\eta}\circ \Delta(x) = f_{\xi,\eta}(\Delta(x)), x\in M, $$
onto the lattice $\mathcal{P}(M)$ is a charge. Taking into account that  $M$ is infinite, we
deduce from \cite[Theorem~30.08]{Shers2008} that the charge $f_{\xi,\eta}\circ \Delta$ is bounded.\smallskip

We shall show now that $\Delta|_{M_{sa}}$ is additive. Since  $f_{\xi,\eta}\circ \Delta|_{P(M)}$ is a bounded additive measure on $\mathcal{P}(M),$ the Bunce-Wright-Mackey-Gleason theorem (\cite[Theorem B]{BuWri92}) implies the the existence of a unique bounded linear functional $\widetilde{f}_{\xi,\eta}$ on $M$ such that $$\widetilde{f}_{\xi,\eta}|_{P(M)}=f_{\xi,\eta}\circ \Delta|_{P(M)}.$$

Let us show that $\widetilde{f}_{\xi,\eta}|_{M_{sa}}=f_{\xi,\eta}\circ \Delta|_{M_{sa}}.$ Take an arbitrary element $x\in M_{sa}.$  By Lemma~\ref{masas} there exists an element $a\in M$ such that $\Delta(y)=ay-ya$ for all $y\in \mathcal{W}^\ast(x).$ In particular, $\Delta$ is linear on $\mathcal{W}^\ast(x),$ and therefore $f_{\xi,\eta}\circ \Delta|_{\mathcal{W}^\ast(x)}$ is a bounded linear functional which is an extension of the signed measure $f_{\xi,\eta}\circ
\Delta|_{P(\mathcal{W}^\ast(x))}.$ By the uniqueness of the extension we have $\widetilde{f}_{\xi,\eta}(x)=f_{\xi,\eta}\circ\Delta(x).$ So $f_{\xi,\eta}\circ\Delta|_{M_{sa}}$ is a bounded linear functional for all $\xi,\eta\in H.$ This means, in particular, that
$$ f_{\xi,\eta}(\Delta(x+y)) = f_{\xi,\eta}(\Delta(x)) + f_{\xi,\eta}(\Delta(y)) = f_{\xi,\eta}(\Delta(x) + \Delta(y)), $$
i.e. $f_{\xi,\eta}(\Delta(x+y) - \Delta(x)-\Delta(y)) = 0$ for all $\xi,\eta\in H,$ which proves $$\Delta(x+y) - \Delta(x)-\Delta(y)
= 0,$$ for all $x,y \in M_{sa}$, i.e.  $\Delta|_{M_{sa}}$ is additive. The proof is complete.
\end{proof}

\begin{lemma}\label{jor} There exists an element  $a\in M$  such that $\Delta(x)=\hbox{adj}_a(x) = ax-xa$ for all $x\in M_{sa}.$
\end{lemma}

\begin{proof}
Consider the extension $\widetilde{\Delta}$ of $\Delta|_{M_{sa}}$ on $M$ defined by:
$$ \widetilde{\Delta}(x_1+ix_2)=\Delta(x_1)+i\Delta(x_2),\, x_1, x_2\in M_{sa}. $$
Taking into account the homogeneity of $\Delta,$ Lemma~\ref{addi} and the equality~\eqref{joor}, we obtain that $\widetilde{\Delta}$ is a Jordan derivation on $M$. As we mentioned above, by\cite[Theorem 1]{Bre}, any Jordan derivation on a semi-prime algebra is a derivation. Since $M$ is semi-prime $\widetilde{\Delta}$ is a derivation on $M$. Therefore there exists an element $a\in M$ such that $\widetilde{\Delta}(x)=ax-xa,$ for all $x\in M.$ In particular, $\Delta(x)=\hbox{adj}_a(x) = ax-xa$, for all $x\in M_{sa}.$
\end{proof}

The final step in our proof is to show that if two $2$-local derivations coincide on $M_{sa}$ then they are equal on the whole von Neumann algebra $M$. The reader is referred to \cite[Lemma 2.12]{AyuKuday2014} for the proof of the next lemma.

\begin{lemma}\label{jorr}\cite[Lemma 2.12]{AyuKuday2014}. Let $\Delta: M\to M$ be a 2-local derivation on a type $III$ von Neumann algebra. If $\Delta|_{M_{sa}}\equiv 0$ then  $\Delta= 0.$ $\hfill\Box$
\end{lemma}

\begin{proof}[\textit{Proof of Theorem~\ref{main} for a type $III$ von Neumann algebra.}]
By Lemma~\ref{jor} there exists an element $a\in M$ such that $\Delta(x)=[a, x]$ for all
$x\in M_{sa}.$ Consider the $2$-local derivation $\Delta-\mbox{adj}_a.$ Since $(\Delta-\mbox{adj}_a)|_{M_{sa}}\equiv 0,$ Lemma~\ref{jorr}
implies that $\Delta=\mbox{adj}_a.$
\end{proof}

We culminate this subsection with a result on 2-local derivations for a certain subclass of C$^\ast$-algebras (see \cite{KimKim05}).\smallskip

Recall that an approximately finite, or $AF$  C$^*$-algebra is a unital C$^*$-algebra $\mathcal{A}$ which is an inductive limit of an increasing sequence of finite-dimensional C$^*$-algebras $\mathcal{A}_n,$ $n \geq 1,$  with unital embeddings $j_n : \mathcal{A}_n \to  A_{n+1}.$  An equivalent definition is to say that $\mathcal{A}$ is an $AF$ C$^*$-algebra if it has an ascending sequence of finite-dimensional C$^*$-subalgebras whose closed union is $A.$

\begin{theorem} \cite{KimKim05}.
Let $\mathcal{A}$ be an $AF$ C$^*$-algebra and let $\phi :\mathcal{A}\to \mathcal{A}$ be a continuous 2-local derivation.
Then $\phi$ is a derivation.$\hfill\Box$
\end{theorem}

\subsection{2-local triple derivations on von Neumann algebras}

In this subsection we consider 2-local triple derivations on  von
Neumann algebras.\smallskip

As in the previous section, on a C$^*$-algebra $A$, we shall consider a ternary product of the form
$$ \{a,b,c\} = \frac1 2 (a b^\ast c + c b^\ast a).$$ By $M_b$ we shall denote the Jordan multiplication mapping by the element
$b,$ that is $M_b (x)= b\circ x = \frac12 (b x+ x b).$\smallskip

Let $\delta: A \to A$ be a triple derivation on a unital C$^*$-algebra. By \cite[Lemmas 1 and 2]{HoMarPeRu}, $\delta(\mathbf{1})^\ast =-\delta (\mathbf{1}),$
$M_{\delta(\mathbf{1})} = \delta (\frac12 \delta (\mathbf{1}),\mathbf{1})$ is an inner triple derivation on $A$, and the difference $D = \delta - \delta (\frac12 \delta (\mathbf{1}),\mathbf{1})$ is a Jordan $^\ast$-derivation on $A,$ more concretely,\label{ref to jordan *der}
$$ D (x\circ y ) = D(x) \circ y + x \circ D(y), \hbox{ and } D(x^\ast) =D(x)^\ast, $$
for every $x,y\in A.$ By \cite[Corollary 2.2]{BarFri}, $\delta$ (and hence $D$) is a continuous operator. We have already mentioned in
previous sections that every bounded Jordan derivation from a C$^\ast$-algebra $A$ to a Banach $A$-bimodule is an associative derivation (cf. \cite{John96}). Therefore, $D$ is an associative $^\ast$-derivation in the usual sense. When $A=M$ is a von Neumann algebra, $D$ is an inner derivation, that is, there exists $a\in A$ satisfying $D(x) = \hbox{adj}_{a} (x)= [a, x],$ for every $x\in A$ (cf. \cite[Theorem 4.1.6]{Sak}). Since $D$ is a $^*$-derivation, we can assume that $a$ is skew-hermitian. So, for every triple derivation $\delta$ on a von Neumann algebra $M,$ there exist skew-hermitian elements $a,b\in M$ satisfying
\begin{equation}\label{liejor}
\delta (x) = [a,x] + b\circ x,
\end{equation} for every $x\in M.$\smallskip

A mapping $T:M\rightarrow M$ is called a \textit{2-local triple derivation} if for every $x, y\in M,$ there is a triple derivation $D_{x, y}:M\rightarrow M,$ depending on $x$ and $y,$ such that $T(x)=D_{x, y}(x)$ and $T(y)=D_{x,y}(y).$\smallskip

We can state now the main result of this subsection (see \cite{KOPR2014}).

\begin{theorem}\label{mainthm}\cite[Theorem 2.14]{KOPR2014}
Let $M$ be  an arbitrary von Neumann algebra and let   $T: M\to M$
be a  2-local triple derivation. Then $T$ is a triple derivation.
\end{theorem}

The proof will be obtained from several partial results.\smallskip

Let $T : M\rightarrow M$ be a $2$-local triple derivation. By \eqref{liejor} for any two elements $x$, $y\in M$ there exist skew-hermitian elements $a_{x,y}, b_{x,y}\in M$ such that $$ T(x)=[a_{x,y}, x]+b_{x,y}\circ x,\hbox{ and, } T (y)=[a_{x,y}, y]+b_{x,y}\circ y. $$
The first step of our proof  shows  that if $T(\mathbf{1})=0$ and $x, y$ both are
hermitian, the ''Jordan part'' in \eqref{liejor} can be chosen to be zero.\smallskip

\begin{lemma}\label{l sead}\cite[Lemma 2.2]{KOPR2014} Let $T: A\to A$ be a 2-local triple derivation on a unital C$^\ast$-algebra satisfying $T(\mathbf{1})=0.$ Then $T(x)=T(x)^\ast$ for all $x\in A_{sa}.$
\end{lemma}

\begin{proof} Let $a_{x,1}, b_{x,1}$ be skew hermitian elements in $M$ such that $$T(x)=[a_{x,1}, x]+b_{x,1}\circ x,\hbox{ and } 0= T (1)=[a_{x,1}, 1]+b_{x,1}\circ 1 = b_{x,1}. $$ Therefore, $T(x)^* = [a_{x,1}, x]^* = [x^*, a_{x,1}^*] = - [x,a_{x,1}] = [a_{x,1}, x]=  T(x)$.
\end{proof}

\begin{lemma}\label{l sea}\cite[Lemma 2.3]{KOPR2014}  Let $T: M\to M$ be a 2-local triple derivation on a von Neumann algebra satisfying
$T(\mathbf{1})=0.$ Then for every $x, y\in M_{sa}$ there exists a skew-hermitian element $a_{x,y}\in M$ such that
$$T(x)=[a_{x,y}, x],\hbox{ and  } T(y)=[a_{x,y}, y].$$
\end{lemma}

\begin{proof}
For every $x, y\in M_{sa}$ we can find skew-hermitian elements
$a_{x,y}, b_{x,y}\in M$ such that
\begin{center}
$T(x)=[a_{x,y}, x]+b_{x,y}\circ x,$  and
$T(y)=[a_{x,y},y]+b_{x,y}\circ y.$
\end{center}

Taking into account that $T(x)=T(x)^\ast$ (see Lemma \ref{l sead}),  we obtain
\begin{eqnarray*}
[a_{x,y},x]+b_{x,y}\circ x & = & T(x)=T(x)^\ast=[a_{x,y},
x]^\ast+(b_{x,y}\circ x)^\ast =\\
& = & [x, a_{x,y}^\ast]+x\circ b_{x,y}^\ast = [x, -a_{x,y}]-x\circ
b_{x,y}=\\
&=& [a_{x,y},x]-b_{x,y}\circ x,
\end{eqnarray*} i.e.
$b_{x,y}\circ x=0,$ and similarly $b_{x,y}\circ y=0.$ Therefore
$T(x)=[a_{x,y},x],$ $T(y)=[a_{x,y},y],$ and the proof is complete.
\end{proof}

The following  observation  plays a useful role in our study.\smallskip

Let $M$ be a von Neumann algebra. If $x \in M_{sa}$, we denote by $s(x)$ the range projection of $x^*=x$ -- that is, the projection
onto $(\ker (x))^\perp = \overline{\ran (x)}$. We say that $x$ \emph{has full support} if $s(x) = \mathbf{1}$ (equivalently, $\ker (x) = \{0\}$).

\begin{lemma}\label{l product}\cite[Lemma 2.5]{KOPR2014} Let $M$ be a von Neumann algebra. Suppose $u \in M_+$ has full support, $c \in M$ is self-adjoint,
and $\sigma(c^2 u) \cap (0,\infty) = \emptyset,$ where $\sigma(a)$ is the spectrum of the element $x.$ Then $c = 0.$ Consequently, if
$u$ and $c$ are as above, and $uc + cu = 0$ {\rm(}or $c^2 u = - cuc \leq 0${\rm)}, then $c = 0.$
\end{lemma}

\begin{proof} For the fist statement of the lemma, suppose $\sigma(c^2 u) \cap (0,\infty) = \emptyset$. Note that
$$(-\infty,0]\supseteq \sigma(c^2 u) \cup \{0\} = \sigma (c c u) \cup \{0\} \supseteq \sigma(c u c). $$
However, $cuc$ is positive, hence $\sigma(cuc) \subset [0,
\|cuc\|]$, with $\displaystyle \max_{\lambda \in \sigma(cuc)} = \|cuc\|.$ Thus,
$c u^{1/2} u^{1/2} c = cuc = 0,$ which means that $c u^{1/2} =
u^{1/2} c=0$ and hence $s(c) \leq \mathbf{1}-s(u^{1/2}) = \mathbf{1}-s(u)=0,$ which leads to $c = 0.$\smallskip

To prove the second part, we observe that $c^2 u = - cuc \leq 0$ implies  $\sigma(c^2 u) \subset (-\infty,0].$
\end{proof}

The following results are ternary versions of Lemmas \ref{masas} and \ref{addi}, respectively.

\begin{lemma}\label{l linearity on single generated von Neumann subalgebras}
Let $T: M \to M$ be a {\rm(}not necessarily linear nor continuous{\rm)} $2$-local
triple derivation on a von Neumann algebra. Let $z\in M$ be a self-adjoint element and let
$\mathcal{W}^\ast(z)=\{z\}''$ be the abelian von Neumann subalgebra of $M$
generated by the element $z$ and the unit element. Then there exist skew-hermitian
elements $a_z, b_z\in M$, depending on $z$, such that
\begin{equation*}\label{spat}
T(x)=[a_z,x] + b_z\circ x = a_z x-x a_z + \frac12 (b_z x+x b_z)
\end{equation*}
for all $x\in \mathcal{W}^\ast(z).$  In particular, $T$ is linear on $\mathcal{W}^\ast(z).$
\end{lemma}

\begin{proof} We can assume that $z\neq 0$. Note that the abelian von Neumann subalgebras
generated by $\mathbf 1$ and $z$ and by $\mathbf 1$  and $\mathbf 1+\frac{\textstyle z}{\textstyle  2\|z\|}$
coincide. So, replacing $z$ with  $\textbf{1}+\frac{\textstyle z}{\textstyle  2\|z\|}$
we can assume that $z$ is an invertible positive element.\smallskip

By definition, there exist skew-hermitian elements $a_{z}, b_{z}\in M$
(depending on $z$) such that $$ T(z) = [a_{z},z]+ b_{z} \circ z.$$

Define a mapping $T_0 : M \to M$ given by $T_0 (x) = T (x)- ([a_{z},z]+ b_{z} \circ z),$
$\ x\in M.$ Clearly, $T_0$ is a 2-local triple derivation on $M$.
We shall show that $T_0\equiv 0$ on $\mathcal{W}^\ast(z)$. Let $x\in \mathcal{W}^\ast(z)$ be an
arbitrary element. By assumptions, there exist skew-hermitian
elements $c_{z,x}, d_{z,x}\in M$ such that
$$
T_0(z)=[c_{z,x}, z]+ d_{z,x}\circ z,\hbox{ and, }  T_0(x)=[c_{z,x}, x]+d_{z,x}\circ x.
$$
Since
$$
0=T_0(z)=[c_{z,x}, z]+d_{z,x}\circ z,
$$
we get
$$
[c_{z,x}, z]+d_{z,x}\circ z=0.
$$
Taking into account that  $z$ is a hermitian element we can easily see that
$c_{z,x} z=z c_{z,x}$ and $d_{z,x} z =- z d_{z,x}.$

Since $z$ has a full support, and $d_{z,x}^2 z =- d_{z,x} z d_{z,x}$, Lemma \ref{l product} implies that  $d_{z,x}=0.$
Further
$$
c_{z,x}\in \{z\}'=\{z\}'''=\mathcal{W}^\ast(z)',
$$
i.e. $c_{z,x}$ commutes with any element in $\mathcal{W}^\ast(z).$
Therefore
$$
T_0(x) = [c_{z,x}, x]+d_{z,x}\circ x=0
$$
for all $x\in \mathcal{W}^\ast(z).$ The proof is complete.
\end{proof}

\begin{proposition}\label{additgen}\cite[Proposition 2.13]{KOPR2014} Let $T: M\to M$ be a  2-local triple derivation on an arbitrary von Neumann algebra. Then the restriction $T|_{M_{sa}}$ is additive.
\end{proposition}

It should be noted that the proof of this result is divided into two cases: finite and properly infinite von Neumann algebras. In the case of a finite von Neumann algebras we follow the same argument in the proof of Theorem~\ref{semi-finite}, and we use a faithful normal semi-finite trace. For properly infinite von Neumann algebras we use the following ternary version of Lemma~\ref{comad}. As in the case of  associative derivations,  the proof is essentially based on the
Bunce-Wright-Mackey-Gleason theorem for bounded measures on projections of von Neumann algebras (see \cite{KOPR2014}).

\begin{proposition}\label{p AyupovKudaybergenov sigma complete additivity ternary}\cite[Proposition 2.7]{KOPR2014}
Let $T: M\to M$ be a  2-local triple derivation on a von Neumann
algebra. Then the following statements
hold:\begin{enumerate}[$(a)$]
\item The restriction $T|_{P(M)}$ is sequentially strong$^*$-additive, and consequently sequentially weak$^*$-additive;
\item $T|_{P(M)}$ is weak$^*$-completely additive, i.e.,
\[
T\left(\hbox{weak$^*$-}\sum\limits_{i\in I} p_i\right)
=\hbox{weak$^*$-} \sum\limits_{i\in I}T(p_i)
\]
 for every family $(p_i)_{i\in I}$ of mutually orthogonal
projections in $M.$
\end{enumerate}
\end{proposition}

\begin{proof}$(a)$ Let $(p_n)_{n\in \mathbb{N}}$ be a sequence of mutually
orthogonal projections in $M.$ Let us consider the element $z=\sum\limits_{n\in
I}\frac{\textstyle 1}{\textstyle n} p_n.$ By Lemma \ref{l linearity on single generated von Neumann subalgebras}
there exist skew-hermitian elements $a_{z},b_{z}\in M$ such that $T(x)=[a_{z},x] + b_{z}\circ x$ for
all $x\in \mathcal{W}^\ast(z).$ Since $\sum\limits_{n=1}^{\infty} p_n, p_m\in
\mathcal{W}^\ast(z),$ for all $m\in \mathbb{N},$ and the product of $M$ is jointly strong$^*$-continuous, we obtain that
$$
T \left(\sum\limits_{n=1}^{\infty} p_n\right)=\left[a_{z}, \sum\limits_{n=1}^{\infty} p_n\right] +
b_{z}\circ \left(\sum\limits_{n=1}^{\infty} p_n \right)$$ $$=\sum\limits_{n=1}^{\infty}[a_{z},
p_n] + \sum\limits_{n=1}^{\infty} b_{z}\circ p_n=\sum\limits_{n=1}^{\infty}  T(p_n),
$$ i.e. $T|_{\mathcal{P}(M)}$ is a countably or sequentially strong$^*$-additive mapping.\medskip

$(b)$ Let $\varphi$ be a positive normal functional in $M_*$, and let $\|.\|_{\varphi}$
 denote the prehilbertian seminorm given by $\|z\|_{\varphi}^2 = \frac12 \varphi (z z^* + z^* z)$
 ($z\in M$). Let $\{p_i\}_{i\in I}$ be an arbitrary family of mutually orthogonal projections in $M.$
 For every $n\in \mathbb{N}$ define $$I_n=\{i\in I: \left\|T (p_i)\right\|_{\varphi} \geq 1/n\}.$$

We claim, that $I_n$ is a finite set for every natural $n$. Otherwise, passing to a subset if necessary,
we can assume that there exists a natural $k$ such that
$I_k$ is infinite and countable. In this case the series
$\sum\limits_{i\in I_k} T(p_i) $ does not converge with respect to the semi-norm  $\|.\|_{\varphi}$.
On the other hand, since $I_k$ is a countable set, by $(a)$, we have
$$T \left(\sum\limits_{i\in I_k} p_i\right) = \hbox{strong$^*$-}\sum\limits_{i\in I_k} T(p_i), $$
which is impossible.  This proves the claim.\smallskip

We have shown that the set
$$ I_0=\left\{i\in I: \left\|T (p_i)\right\|_{\varphi} \neq 0\right\}=\bigcup\limits_{n\in \mathbb{N}}I_n $$
is a countable set, and $\left\|T (p_i)\right\|_{\varphi}= 0$, for every $i\in I\backslash I_0$.\smallskip

Set $p=\sum\limits_{i\in I\setminus I_0} p_i\in M.$ We shall show that $\varphi (T(p)) =0.$ Let $q$
 denote the support projection of $\varphi $ in $M$.
Having in mind that $\left\|T (p_i)\right\|^2_{\varphi}= 0$, for every $i\in I\backslash I_0$,
we deduce that $T(p_i) \perp q$ for every $i\in I\backslash I_0$.\smallskip

Replacing $T$ with $\widehat{T}=T - \delta (\frac12 T ({\mathbf 1}),{\mathbf 1})$  we can assume that $T({\mathbf 1}) = 0$
(cf. comments in page \pageref{ref to jordan *der}) and $T(x) = T(x)^*$, for every $x\in M_{sa}$ (cf. Lemma \ref{l sead}).
By Lemma \ref{l sea}, for every $i\in I\setminus I_0$ there exists a skew-hermitian element $a_i=a_{p,p_i}\in M$ such that
$$ T (p)=a_i p- p a_i,\hbox{ and, } T(p_i)=a_i p_i-p_i a_i. $$
Since $T(p_i) \perp q$ we get $(a_i p_i-p_i a_i) q = q (a_i p_i-p_i a_i) =0$, for all $i\in I\setminus I_0.$ Thus, since $p a_i p_i q= p_i a_i q$,
$$ (T(p) p_i)  q = (a_i p - p a_i) p_i q = a_i p_i q - p a_i p_i q  $$
$$= a_i p_i q - p_i a_i  q = (a_i p_i-p_i a_i) q =0,$$ and similarly
$$ q (p_i T(p)) = 0,$$ for every $i\in I\setminus I_0$. Consequently,
\begin{equation}\label{eq qpT(p)=0} (T(p) p) q=  T(p) \left(\sum\limits_{i\in I\setminus I_0}p_i \right) q = 0 =
q \left(\sum\limits_{i\in I\setminus I_0}p_i \right) T(p)=q (p T(p)).
\end{equation}
 Therefore, $$T(p) = \delta_{p,{\mathbf 1}} (p) = \delta_{p,{\mathbf 1}} \{p,p,p\}= 2 \{ \delta_{p,{\mathbf 1}}(p),p,p\} +
 \{p,\delta_{p,{\mathbf 1}} (p), p\}  $$ $$= 2 \{T(p),p,p\} +  \{p,T (p), p\} = p T(p) + T(p) p + p T(p)^* p $$
 $$=  p T(p) + T(p) p + p T(p) p,$$ which implies that $$\varphi (T(p)) = \varphi (p T(p) + T(p) p + p T(p) p) $$
 $$= \varphi (qp T(p) q) +\varphi(q T(p) p q) + \varphi (q p T(p) p q)= \hbox{(by \eqref{eq qpT(p)=0})}= 0.$$\smallskip

Finally, by $(a)$ we have
$$T\left(\sum\limits_{i\in I_0} p_i\right)= \|.\|_{\varphi}\hbox{-}\sum\limits_{i\in I_0} T\left(p_i\right).$$
Two more applications of $(a)$ give:
$$ \varphi \left(T\left(\sum\limits_{i\in I} p_i\right)\right)  = \varphi \left(
T\left(p +\sum\limits_{i\in I_0} p_i\right) \right)=
 \varphi \left(T(p) + T\left(\sum\limits_{i\in I_0} p_i\right)\right)$$
$$ = \varphi \left(T(p)\right) + \varphi \left( T\left(\sum\limits_{i\in I_0} p_i\right)\right) = \sum\limits_{i\in I_0} \varphi \left( T\left( p_i\right)\right).$$

By the Cauchy-Schwarz inequality, $0\leq \left|\varphi T(p_i)\right|^2 \leq \left\|T (p_i)\right\|^2_{\varphi}= 0$, for every $i\in I\backslash I_0$,
and hence $\displaystyle\sum\limits_{i\in I_0} \varphi \left( T\left( p_i\right)\right) = \sum\limits_{i\in I} \varphi \left( T\left( p_i\right)\right).$
The arbitrariness of $\varphi$ shows
that $T\left(\hbox{weak$^*$-}\sum\limits_{i\in I} p_i\right) =\hbox{weak$^*$-}
\sum\limits_{i\in I}T(p_i)$.
\end{proof}

$T: M\to M$ be a  2-local triple derivation on a von Neumann algebra, and let $\phi$ be a normal functional in the predual of $M.$ Proposition \ref{p AyupovKudaybergenov sigma complete additivity ternary}
implies that the mapping $\phi\circ T|_{\mathcal{P}(M)}: \mathcal{P}(M) \to \mathbb{C}$ is
a completely additive mapping or a charge on $M$. Under the additional hypothesis of $M$ being a continuous von Neumann algebra or,  more generally, a von Neumann algebra with no Type I$_n$-factors ($1< n< \infty$) direct summands (i.e. without direct summand isomorphic to a matrix algebra $M_n(\mathbb{C})$, $1<n<\infty$), the Dorofeev--Sherstnev theorem (\cite[Theorem 29.5]{Shers2008} or \cite[Theorem 2]{Doro}) guarantees that $\phi\circ T|_{\mathcal{P}(M)}$ is a bounded charge, that is,
the set $\left\{ |\phi\circ T (p)| : p \in \mathcal{P}(M) \right\}$ is bounded, and so an application of the uniform boundedness principle gives:

\begin{corollary}\label{c 2 local triple derivations are bounded charges}\cite[Corollary 2.8]{KOPR2014} Let $M$ be a von Neumann algebra with no Type I$_n$-factor direct summands {\rm(}$1< n< \infty${\rm)} and let $T: M\to M$ be a {\rm(}not necessarily linear nor continuous{\rm)} 2-local triple derivation. Then the restriction $T|_{\mathcal{P}(M)}$ is a bounded weak$^*$-completely additive mapping.$\hfill\Box$
\end{corollary}

When $M$ is a von Neumann algebra with no Type I$_n$-factor direct summands {\rm(}$1< n< \infty${\rm)}, and $T: M\to M$ is a {\rm(}not necessarily linear nor continuous{\rm)} 2-local triple derivation. We can combine Corollary \ref{c 2 local triple derivations are bounded charges} above with the Bunce-Wright-Mackey-Gleason theorem \cite{BuWri92, BuWri94}, and argument similar to that employed in the proof of Lemma \ref{addi} shows:

\begin{proposition}\label{p addit no type I_n}\cite[Proposition 2.9]{KOPR2014} Let $T: M\to M$ be a {\rm(}not necessarily linear nor continuous{\rm)}
2-local triple derivation on a von Neumann algebra with no Type I$_n$-factor direct
summands {\rm(}$1< n< \infty${\rm)}. Then the restriction
$T|_{M_{sa}}$ is additive.$\hfill\Box$
\end{proposition}

The comments made on page \pageref{ref comments on central projections} show that that for every 2-local derivation $\Delta$ on a von Neumann algebra $M$, $\Delta(ex)=e\Delta(x)$ for every central projection $e\in M$ and for every $x\in M.$ In the case of 2-local triple derivations we can prove:

\begin{lemma}\label{l:preserves central projections}\cite[Lemma 2.11]{KOPR2014}
If $T$ is a  $2$-local triple derivation on a von Neumann algebra
$M,$ and $p$ is a central projection in $M,$ then $T(Mp) \subset
Mp.$ In particular, $T(px) = pT(x)$ for every $x\in M$.
\end{lemma}

\begin{proof}
Consider $x \in Mp$, then $x = pxp = \{x,p,p\}$. $T$ coincides
with a triple derivation $\delta_{x,p}$ on the set $\{x,p\}$, hence
%$$T(p) =\delta_{x,p} (p) = \delta_{x,p} \{p, p, p\} = 2 \{ \delta_{x,p} (p), p, p \} + \{p,
%\delta_{x,p} (p), p\} $$ $$=  \delta_{x,p} (p) p + p \delta_{x,p} (p) + p \delta_{x,p} (p)^* p $$
%belongs to the two-sided ideal $Mp.$ Furthermore,
$$T(x)=\delta_{x,p} (x) = \delta_{x,p} \{ x,p, p\} = \{ \delta_{x,p} (x), p,p \} + \{x,
\delta_{x,p} (p), p \} + \{x,p,\delta_{x,p} (p)\}$$ lies in $Mp .$\smallskip

For the final statement, fix $x\in M,$ and consider skew-hermitian elements $a_{x,xp},$ $b_{x,xp}\in M$ satisfying
$$T(x) = [a_{x,xp},x]+ b_{x,xp}\circ x,\hbox{ and } T(xp) = [a_{x,xp},xp]+ b_{x,xp}\circ (xp).$$ The assumption $p$ being central implies that $pT(x) = T(px).\qedhere$
\end{proof}

Lemma \ref{l:preserves central projections} assures that the proof of Proposition \ref{additgen} will follow from Proposition \ref{p addit no type I_n} and from the next result taken from \cite{KOPR2014}.

\begin{proposition}\label{p addit finite}\cite[Proposition 2.12]{KOPR2014} Let $T: M\to M$ be a {\rm(}not necessarily linear nor continuous{\rm)} 2-local triple derivation
on a finite von Neumann algebra.   Then the restriction $T|_{M_{sa}}$ is additive.
\end{proposition}

The reader is referred to \cite{KOPR2014} for a detailed proof of the above proposition.\smallskip

The following two results are ternary versions of Lemmas~\ref{jor} and \ref{jorr}.

\begin{lemma}\label{jort}\cite[Lemma 2.15]{KOPR2014} Let $T: M\to M$ be a
2-local triple derivation on a  von Neumann algebra with
$T(\mathbf{1})=0.$ Then there exists a skew-hermitian element
$a\in M$ such that $T(x)=[a, x],$ for all $x\in M_{sa}.$
\end{lemma}

\begin{proof}
Let $x\in M_{sa}.$ By Lemma~\ref{l sea} there exist a
skew-hermitian element $a_{x,x^2}\in M$ such that
$$
T(x)=[a_{x,x^2},x],\, T(x^2)=[a_{x,x^2}, x^2].
$$
Thus
$$
T (x^2)=[a_{x,x^2},x^2]=[a_{x,x^2},x]x+x[a_{x,x^2},x]=T(x)x+xT(x),
$$
i.e.
\begin{equation}\label{jord}
T (x^2)=T(x)x+xT(x),
\end{equation}
 for every $x\in M_{sa}$.\smallskip

By Proposition~\ref{additgen} and Lemma \ref{l sead}, $T|_{M_{sa}} : M_{sa}
\to M_{sa}$ is a real linear mapping. Now, we consider the linear extension
$\hat{T}$ of $T|_{M_{sa}}$ to $M$ defined by
$$ \hat{T}(x_1+ix_2)=T(x_1)+i T(x_2),\, x_1, x_2\in M_{sa}.$$

Taking into account the homogeneity of $T,$ Proposition~\ref{additgen}
and the identity~\eqref{jord} we obtain that $\hat{T}$ is a
Jordan derivation on $M.$  By \cite[Theorem 1]{Bre} any Jordan
derivation on a semi-prime algebra is a derivation. Since $M$ is
von Neumann algebra, $\hat{T}$ is a derivation on $M$ (see also \cite{Sinclair70} and \cite{John96}).
Therefore there exists an element $a\in M$ such that $\hat{T}(x)=[a,x]$
for all $x\in M.$ In particular, $T(x)=[a, x]$ for all $x\in
M_{sa}.$ Since $T(M_{sa}) \subseteq M_{sa}$, we can assume that
$a^\ast=-a$, which completes the proof. \end{proof}

\begin{lemma}\label{l:vanish on self-adjoint} Let $T: M\to M$ be a {\rm(}not necessarily linear nor continuous{\rm)}
2-local triple derivation on a  von Neumann algebra. If
$T|_{M_{sa}}\equiv 0,$ then $T\equiv 0.$
\end{lemma}

\begin{proof}
Let $x\in M$ be an arbitrary element and let $x=x_1+ix_2,$ where
$x_1, x_2\in M_{sa}.$ Since $T$ is homogeneous, if necessary,
passing to the element $(1+\|x_2\|)^{-1} x,$ we can suppose that
$\|x_2\|<1.$ In this case the element $y=\mathbf{1}+x_2$ is
positive and invertible. Take skew-hermitian elements $a_{x,y},
b_{x,y}\in M$ such that
$$
T(x)=[a_{x,y},x]+b_{x,y}\circ x,
$$
$$
T(y)=[a_{x,y},y]+b_{x,y}\circ y.
$$
Since $T(y)=0,$ we get $[a_{x,y},y]+b_{x,y}\circ y=0.$ Passing to
the adjoint, we obtain $[a_{x,y},y]-b_{x,y}\circ y=0.$ By adding
and subtracting these two equalities  we obtain that $[a_{x,y},
y]=0$ and $ib_{x,y} \circ y=0.$ Taking into account that
$ib_{x,y}$ is hermitian, $y$ is positive and invertible,
Lemma~\ref{l product} implies that $b_{x,y}=0.$\smallskip

We further note that
$$
0=[a_{x,y},y]=[a_{x,y}, \mathbf{1}+x_2]=[a_{x,y}, x_2],
$$
i.e.
$$
[a_{x,y}, x_2]=0.
$$
Now,
$$
T(x)=[a_{x,y},x]+b_{x,y}\circ x=[a_{x,y}, x_1+ix_2]=[a_{x,y},
x_1],
$$
i.e.
$$
T(x)=[a_{x,y}, x_1].
$$
Therefore,
$$
T(x)^\ast=[a_{x,y}, x_1]^\ast=[x_1, a_{x,y}^\ast]=[x_1,
-a_{x,y}]=[a_{x,y}, x_1]=T(x).
$$
So
\begin{equation}\label{iii} T(x)^\ast=T(x).
\end{equation}

Now replacing $x$ by $ix $ on~\eqref{iii} we obtain from the
homogeneity of $T$ that
\begin{equation}\label{iv}T(x)^\ast=-T(x).
\end{equation}
Combining \eqref{iii} and \eqref{iv} we obtain that $T(x)=0,$
which finishes the proof. \end{proof}

\begin{proof}[\textit{Proof of Theorem~\ref{mainthm}}] Let us
define $\widehat{T}=T - \delta \left(\frac12 T ({\mathbf
1}),{\mathbf 1}\right).$ Then $\widehat{T}$ is a 2-local triple
derivation on $M$ with $\widehat{T}({\mathbf 1}) = 0$  and
$\widehat{T}(x) = \widehat{T}(x)^\ast,$ for every $x\in M_{sa}$
(cf. Lemma \ref{l sead}). By Lemma~\ref{jort} there exists an
element $a\in M$ such that $\widehat{T}(x)=[a,x]$ for all $x\in
M_{sa}.$ Consider the $2$-local triple derivation
$\widehat{T}-\mbox{adj}_a.$ Since
$(\widehat{T}-\mbox{adj}_a)|_{M_{sa}}\equiv 0,$
Lemma~\ref{l:vanish on self-adjoint} implies that
$\widehat{T}=\mbox{adj}_a,$ and hence $T =\mbox{adj}_a + \delta
\left(\frac12 T ({\mathbf 1}),{\mathbf 1}\right),$ that finishes the proof.
\end{proof}

\begin{openproblem}
Is every 2-local derivation on a C$^*$-algebra a derivation?
\end{openproblem}

\begin{openproblem}
Is every 2-local triple derivation on a JB$^*$-triple a triple derivation?
\end{openproblem}

\subsection{$2$-local derivation on  Arens algebras} \label{sec:2}
In  this subsection we present some results on 2-local derivations on Arens algebras (see \cite{AKNA}).\smallskip

Let $M$ be an arbitrary semi-finite von Neumann algebra with a faithful normal semi-finite trace $\tau.$
A linear subspace  $\mathcal{D}$ in  $H$ is said to be \emph{affiliated} with  $M$ (denoted as  $\mathcal{D}\eta M$), if $u(\mathcal{D})\subset \mathcal{D}$ for every unitary  $u$ from
the commutant $$M'=\{y\in B(H):xy=yx, \,\forall x\in M\}$$ of the von Neumann algebra $M.$\smallskip

A linear operator  $x: \mathcal{D}(x)\rightarrow H,$ where the domain, $\mathcal{D}(x),$ of $x$ is a linear subspace of $H,$ is said to be \textit{affiliated} with $M$ (denoted as $x\eta M$) if $\mathcal{D}(x)\eta M$ and $u(x(\xi))=x(u(\xi))$ for all  $\xi\in \mathcal{D}(x)$  and for every unitary  $u\in M'.$\smallskip

A linear subspace $\mathcal{D}$ in $H$ is said to be \textit{strongly dense} in $H$ with respect to the von Neumann algebra  $M,$ if \begin{enumerate}[$1)$]\item $\mathcal{D}$ is affiliated with  $M$;
\item  There exists a sequence of projections $\{p_n\}_{n=1}^{\infty}$ in $\mathcal{P}(M)$ such that $p_n\uparrow\textbf{1},$ $p_n(H)\subset \mathcal{D}$ and $p^{\perp}_n=\textbf{1}-p_n$ is finite in $M$ for all $n\in\mathbf{N}$.
\end{enumerate}

We recall that a closed linear operator $x$ is said to be $\tau$\textit{-measurable} with respect to the von Neumann algebra $M,$ if  $x\eta M$ and $\mathcal{D}(x)$ is $\tau$-dense in $H,$ i.e. $\mathcal{D}(x)\eta M$ and given $\varepsilon>0$ there exists a projection   $p\in M$ such that $p(H)\subset\mathcal{D}(x)$  and $\tau(p^{\perp})<\varepsilon.$ Denote by  $S(M,\tau)$ the set of all   $\tau$-measurable
operators with respect to  $M.$\smallskip

Given $p\geq1$ put $L^{p}(M, \tau)=\{x\in S(M, \tau):\tau(|x|^{p})<\infty\}.$ It is known \cite{Yea} that
$L^{p}(M, \tau)$ is a Banach space with respect to the norm $$\|x\|_p=(\tau(|x|^{p}))^{1/p},\quad x\in L^{p}(M, \tau).$$ If we consider the intersection
$$L^{\omega}(M, \tau)=\bigcap\limits_{p\geq1}L^{p}(M, \tau),$$
it is established in \cite{Abd} that  $L^{\omega}(M, \tau)$ is a locally convex complete metrizable $\ast$-algebra with respect to the  topology $t$ generated by the family formed by all norms $\{\|\cdot\|_p\}$ with ${p\geq1}.$ The algebra $L^{\omega}(M, \tau)$ is called a (non commutative) \textit{Arens algebra}. Note that $L^{\omega}(M, \tau)$ is  $\ast$-subalgebra in $S(M,\tau)$ and if $\tau$ is a finite trace then $M\subset L^\omega(M,\tau).$\smallskip

The spaces
$$L^{\omega}_2(M, \tau)=\bigcap\limits_{p\geq 2}L^p(M, \tau)$$
and
$$ M+ L^{\omega}_2(M, \tau)=\{x+y: x\in M, y\in L^{\omega}_2(M, \tau)\}, $$ also play an interesting role in the theory of (2-local) derivations. For example, $L^{\omega}_2(M, \tau)$  and $M+ L^{\omega}_2(M, \tau)$ are a $\ast$-algebras and $L^{\omega}(M, \tau)$ is an ideal in $M+L^{\omega}_2(M, \tau)$ (see \cite{Alb3}). We observe that if $\tau(\textbf{1})<\infty$ then $M+ L^{\omega}_2(M, \tau)=L^{\omega}_2(M, \tau)=L^{\omega}(M, \tau).$ Theorem 3.7 in \cite{Alb3} proves that if $M$ is a  von Neumann algebra with a faithful normal semi-finite trace $\tau$, then any derivation $D$ on $L^{\omega}(M, \tau)$ is spatial, it is further known that any derivation is implemented by an element in $M+L^{\omega}_2(M, \tau),$ i.e.,
\begin{equation}
\label{SP} D(x)=ax-xa, \quad x\in L^{\omega}(M, \tau),
\end{equation} for some $a\in M+L^{\omega}_2(M, \tau).$\smallskip

The arguments in the proof of Theorem~\ref{semi-finite} can be easily modified to the case of 2-local derivation on Arens
algebras to prove the next result:

\begin{theorem}\label{TL}\cite[Theorem 2.4]{AKNA}
Let $M$ be a   von Neumann algebra  with a faithful normal semi-finite trace $\tau.$ Then any 2-local derivation $\Delta$ on
the algebra $L^{\omega}(M,\tau)$ is a derivation and has the form given in (\ref{SP}).
\end{theorem}

The following corollary is a consequence of the above theorem.

\begin{corollary}\label{C2}\cite[Corollary 2.5]{AKNA}
Let $M$ be a commutative von Neumann algebra with a faithful normal semi-finite trace $\tau.$ Then any 2-local derivation $\Delta$ on the algebra  $L^{\omega}(M,\tau)$ is identically zero.
\end{corollary}

\end{document}